\documentclass[11pt]{amsart}
\usepackage{amssymb,mathrsfs,graphicx,enumerate}
\usepackage[pagewise]{lineno}
\usepackage{amsmath,amsfonts,amssymb,amscd,amsthm,bbm}
\usepackage{float}
\usepackage{mathtools}
\usepackage{epsfig,subfigure}
\usepackage{soul}

\title[Singular C-S Model on the real line]{Complete classification of the asymptotical behavior for singular C-S Model on the real line}

\author[Zhang]{Xiongtao Zhang}
\address[Xiongtao Zhang]{\newline Center for Mathematical Sciences, Huazhong University of Science and Technology, Wuhan, China}
\email{xtzhang@hust.edu.cn}

\author[Zhu]{Tingting Zhu}
\address[Tingting Zhu]{\newline School of Mathematics and Statistics, Huazhong University of Science and Technology, Wuhan, China }
\email{ttzhu@hust.edu.cn}

\newtheorem{theorem}{Theorem}[section]
\newtheorem{lemma}{Lemma}[section]
\newtheorem{corollary}{Corollary}[section]

\newtheorem{remark}{Remark}[section]

\newtheorem{definition}{Definition}[section]

\newcommand{\bbr}{\mathbb R}

\begin{document}
\date{\today}


\thanks{The work of X. Zhang is supported by the National Natural Science Foundation of China (Grant No. 11801194).}


\begin{abstract}
In this paper, we study the singular Cucker-Smale (C-S) model on the real line. For long range case, i.e. $\beta<1$, we prove the uniqueness of the solution in the sense of Definition \ref{D2.1} and the unconditional flocking emergence. Moreover, the sufficient and necessary condition for collision and sticking phenomenon will be provided. For short range case, i.e. $\beta>1$, we construct the uniform-in-time lower bound of the relative distance between particles and provide the sufficient and necessary condition for the emergence of multi-cluster formation. For critical case, i.e. $\beta=1$, we show the uniform lower bound of the relative distance and unconditional flocking emergence. These results provide a complete classification of the collective behavior for C-S model on the real line.
\end{abstract}
\maketitle \centerline{\date}

\section{Introduction} \label{sec:1}
\setcounter{equation}{0}
\vspace{0.3cm}
Collective behaviors of multi-agent systems are ubiquitous in the nature, like flocking of birds, synchronization of oscillators, herding of sheeps and alignment of robots \cite{D-M1, D-M2, D-M3, L-P-L-S, P-L-S-G-P, T-T, T-B, V-C-B-C-S, Wi}, etc. During recent decades, several different dynamical models were proposed to describe the collective behaviors, such as Kuramoto model \cite{Ku}, Winfree model \cite{Wi}, Vicsek model \cite{V-C-B-C-S}, Cucker-Smale (C-S) model \cite{C-S2} etc. These seminal works and models have been extensively studied and widely applied in various area such as engineering, biology, physics and social science communities. In order to study the emergence of flocking through a mean field interaction network, F. Cucker and S. Smale in their seminal work \cite{C-S2} provided the C-S model as follows,
\begin{equation}\label{A1}
\left\{\begin{aligned}
&\dot{x_i}=v_i,\quad t>0,\quad i=1,\cdots,N,\\
&\dot{v_i}=\frac{\kappa}{N}\sum_{j=1}^N\psi(\|x_i-x_j\|)(v_j-v_i),
\end{aligned}
\right.
\end{equation}
where $\kappa$ is a nonnegative coupling strength and $\psi$ is a communication weight measuring mutual interactions, which is assumed to satisfy a regularized inverse power law $ \psi(\|x_i-x_j\|)=(1+\|x_i-x_j\|^2)^{-\frac{\beta}{2}}$. Although it seems to be a very simple model, the phenomena related to this model are quite rich. In fact, a lot of analysis has been done to explain the emergence of collective behavior in the dynamics of the C-S model. To name a few, avoidance of collision \cite{A-C-H-L, C-D}, flocking in random environment \cite{A-H, C-M, H-L-L}, mean-field limit in deterministic and stochastic sense \cite{B-C-C, C-C-R}, various type of collective behaviors emergence \cite{B-H, C-D-P, C-F-R-T, C-K-M-T, C-S2, D-M1, D-M2, D-M3, F-H-T, H-H-K, H-L, H-S, H-T, She}, local flocking \cite{C-H-H-J-K1, C-H-H-J-K2, H-K-Z}, bonding force \cite{P-K-H}, generalized flocking \cite{M-T}, singular and hyperbolic limits \cite{P-S}, kinetic equation \cite{D-F-T}, application to flight formation \cite{P-E-G} and flocking with leaders \cite{L-X}, etc. Instead of the original regular communication setting, In the present paper, we will consider the C-S model with singular communication weight ,i.e.
\begin{equation}\label{A2}
\psi(\|x_i-x_j\|)=\frac{1}{(\|x_i-x_j\|^2)^{\frac{\beta}{2}}}.
\end{equation}
In fact, recently, the C-S model with singular interaction \eqref{A2} attracts a lot of attention from various area. This is mainly due to that the Coulomb type interaction will automatically generates the repulsion and leads to the avoidance of collision \cite{C-C-M-P}, which is more physical and very important for application in engineering such as formation of unmanned aerial vehicles. However, this singular communication weight causes a lot of difficulty in the mathematical analysis. For instance, the uniqueness of the solution to \eqref{A1} cannot be guaranteed by the fundamental theory of ODE, because the vector field on the right hand side of \eqref{A1} is no more Lipschitz. Therefore, comparing to the extensive study on regular communication weight, there are very few works concerning on the C-S model with singular interaction: flocking dynamics and mean-field limit \cite{H-L},  avoidance of collision \cite{A-C-H-L, C-C-M-P}, global existence of weak solutions in particle and kinetic level \cite{C-C-H, M-P, P, Pe2}. More precisely, in \cite{P}, the author constructed the global existence of weak solution without uniqueness. Meanwhile, the author in \cite{P} found the finite time flocking phenomenon but only in two particles system. In \cite{C-C-H, C-C-M-P}, the authors proved the collision avoidance for C-S model in any finite time, but the uniform-in-time lower bound is still unknown. Moreover, in \cite{H-P-Z, H-K-P-Z}, the authors studied the C-S model with regular short range interaction and constructed a sufficient and necessary condition for the emergence of mono-cluster and multi-cluster formation, respectively. While, it is not clear wether the similar results can be obtained in the singular case. Therefore, may we address three natural questions as follows, \newline

\begin{itemize}
\item (Q1) Can we derive the uniqueness of the solution to the C-S model \eqref{A1} with singular communication \eqref{A2}?
Moreover, can we find the sufficient and necessary condition for emergence of finite time flocking in $N$ particle system?
\vspace{0.2cm}
\item (Q2) Can we construct the uniform-in-time lower bound between two particles of the C-S model with singular interaction, so that the collision avoidance occurs when time tends to infinity and asymptotical equilibrium state can be constructed? 
\vspace{0.2cm}
\item (Q3) Can we obtain the sufficient and necessary condition for the emergence of mono-cluster and multi-cluster formation, and thus construct the critical coupling strength $\kappa_c$ for flocking emergence as in \cite{H-K-P-Z}? 
Moreover, can we obtain the information of cluster number and the asymptotical velocity of each cluster?
\vspace{0.2cm}
\end{itemize}

In order to answer the above three questions which arise from the previous works, we will study the C-S model \eqref{A1} with singular interaction in one dimensional case. The advantage of the one dimensional model is that it is equivalent to a first order model as in \cite{H-P-Z, H-K-P-Z}, for which we can obtain more delicate estimates and thus provide the answers to the above questions. More precisely, we will study the equivalent first order equation as follows,
\begin{align}
\left\{\begin{aligned} \label{A3}
& \dot{x}_i = \nu_i +  \frac{\kappa}{N} \sum_{k = 1}^{N} \Psi(x_k - x_i), \quad t > 0,~~i= 1, \cdots, N, \\
& \Psi(-x) = -\Psi(x), \quad x \in \bbr, \qquad 0 < \Psi^{\infty} = \lim_{t \to \infty} \Psi(x) < \infty,
\end{aligned}
\right.
\end{align}
where we call $\nu_i$ the natural velocity which depends on the initial data of the second order C-S model. Please see Section \ref{sec:2.1} for details of the derivation. So far, we cannot extend this analysis into higher dimension since the integrable structure does not exist in multi-dimensional case.\newline 

The main results in this paper will be interpreted in three-fold. First, for long range case $\beta<1$, we will show that the classical solution to the equation \eqref{A3} is unique, and thus prove the weak solution to \eqref{A3} defined  in \cite{P} is unique in one dimensional case. More over, we prove that the collision will occur only finite times and the finite time flocking will occur if and only if $\beta<1$ and the natural velocities are identical. Thus we conclude that there are infinite many initial data for second order C-S model which leads to finite time flocking, but these initial data are of measure zero in the phase space. This refines the results in \cite{P} (see Theorem \ref{T3.1} and Corollary \ref{C3.1} in Section \ref{sec:3} for details). Second, for short range case $\beta > 1$, we construct the uniform in time lower bound of the distance between two particles, which depends on the particle number $N$. This improves the results in \cite{C-C-M-P} (see Theorem \ref{TD1} in Section \ref{sec:4.1} for details). Finally, for short range case $\beta>1$, we provide the sufficient and necessary condition for emergence of mono-cluster and multi-cluster, and derive the critical coupling strength for mono-cluster flocking emergence. Meanwhile, we will also provide the algorithm to count the cluster number and calculate the asymptotical velocity of each cluster (see Theorem \ref{T4.2}, Theorem \ref{T4.3} and Corollary \ref{C4.1} in Section \ref{sec:4} for details).\newline

The rest of the paper is organized as follows. In Section \ref{sec:2}, we will provide some preliminary definitions and previous results which will be used in the later sections. In Section \ref{sec:3}, we will study the long range case $\beta<1$ and prove the uniqueness of the solution. Moreover, we will show the unconditional flocking and the asymptotical lower bound of the distance between two particles, which shows the existence of the equilibrium state. In Section \ref{sec:4}, we will consider the short range case $\beta>1$ and show the sufficient and necessary condition for emergence of multi-cluster formation. Then, we will extend the results to second order model and discuss the differences between the collective behavior of first and second order models. In Section \ref{sec:5}, we will discuss the differences between the regular C-S model and the singular one. Finally, Section \ref{sec:6} is devoted to a brief summary of our main results and discussion of possible future works. \newline

\section{Preliminary}\label{sec:2}
\setcounter{equation}{0}
\vspace{0.3cm}
In this section, we will firstly derive the system \eqref{A3} for the case $0<\beta<1$ and $\beta\geq 1$ respectively. Then we will introduce some preliminary definitions and previous results that we will use in the later sections. Finally we will provide some notations that will be used through the paper. 
\subsection{Derivation of the first order system}\label{sec:2.1}
%
In this subsection, we will show the details of the construction of the equivalent first order equation \eqref{A3} in both long range and short range cases. Comparing to the regular case \cite{H-P-Z, H-K-P-Z}, the construction of the first order equation in the present paper is a little different due to the singularity of the communication function.\newline 

\noindent$\bullet$ Case 1: $0 < \beta < 1$. It is obvious that the anti-derivative $\Psi(x)$ of the communication function $\psi$ is increasing, odd in its domain and concave on $[0, +\infty)$. Therefore, the spatial configuration $x_i$ naturally satisfies
\begin{equation}\label{B1}
\left\{\begin{aligned}
&\dot{x}_i(t) = \nu_i + \frac{\kappa}{N} \sum_{k=1}^{N} \Psi(x_k(t) - x_i(t)), \qquad t > 0, \, i =1,2,\dots,N,\\
& \nu_i(X^0, V^0) := v_i^0 - \frac{\kappa}{N} \sum_{k=1}^{N} \Psi(x_k^0 - x_i^0),
\end{aligned}
\right.
\end{equation}
where the natural velocities $\nu_i$ depend on the initial conditions of the second order system and the nonlinear interaction function is defined as follows 
 \begin{equation}\label{B2}
 \Psi(x) := \int_0^x \psi(r) dr= sgn(x)\frac{1}{1-\beta} |x|^{1-\beta}, \qquad x \in (-\infty, +\infty).
 \end{equation}

\noindent$\bullet$ Case 2: $\beta \ge 1$. In this case, the singular communication weight is in the form $\psi(r) =  \frac{1}{|r|^\beta}, \beta \ge 1$. As the communication function is non-integrable at origin, we have to avoid the origin when we define the potential function as the anti-derivative of the communication function $\psi$. In order to distinguish the higher order case from the long range case, we will instead use $\Phi(x)$ to represent the potential function and we have,
\begin{equation*}
\Phi(x) =
\begin{cases}
\displaystyle  \int_1^x \psi(r) dr, & x > 0, \\
\displaystyle -  \int_1^{-x} \psi(r) dr, & x < 0.
\end{cases}
\end{equation*}
Note the integration is taken from $1$ to $x$ because $\psi$ is not integrable at origin. Therefore the repulsion can be observed very clearly when $x$ is closed to zero. More precisely, for the critical exponent $\beta=1$ and higher order exponent $\beta>1$, we have the following explicit formula respectively, 
\begin{equation}\label{B3}
\Phi(x) =
\begin{cases}
\displaystyle sgn(x)\log |x|, & \quad x\neq 0,\quad \beta=1  \\
\displaystyle sgn(x)(\beta - 1)^{-1}(1 - |x|^{1-\beta }),  & \quad x\neq 0,\quad \beta>1.
\end{cases}
\end{equation}
For simplicity, we will only study the case $\beta>1$ in detail and the critical case $\beta=1$ can be treated similarly, which will be discussed in the summary. It is obvious that the function $\Phi(x)$ is odd in the domain $(-\infty, 0) \cup (0, +\infty)$. Moreover, it's monotonic increasing on the positive and negative real line respectively with a discontinuous jump at the origin. Then the second order C-S system can be reduced to the first order system as below

\begin{equation} \label{B5}
\left\{\begin{aligned}
&\dot{x}_i(t) = \nu_i + \frac{\kappa}{N} \underset{k \ne i}{\sum_{k=1}^N} \Phi(x_k(t) - x_i(t)), \quad t>0, \ i =1,2,\ldots,N,\\
&\nu_i = v_i^0 - \frac{\kappa}{N} \underset{k \ne i}{\sum_{k=1}^{N}} \Phi(x_k^0 - x_i^0).
\end{aligned}
\right.
\end{equation}
\vspace{0.2cm}

It is obvious that the mean velocity and mean position are both conserved for the second order model \eqref{A1} and first order reduction model  \eqref{A3}. Therefore, there will be no equilibrium state if the mean velocity is nonzero. Thus, according to the Galilean invariance of the system \eqref{A1} and \eqref{A3}, without loss of generality, we assume the initial mean velocity and initial mean position are both zero, i.e.
\begin{equation}\label{B6}
\sum_{i=1}^Nv_i^0=\sum_{i=1}^Nx_i^0=\sum_{i=1}^N\nu_i=0.
\end{equation}
\vspace{0.2cm}
\subsection{Preliminary concepts and previous results}\label{sec:2.2} In this subsection, we will introduce some preliminary concepts and previous results that will be used in the later sections. First, in \cite{P}, the author provide a definition of weak solution to the second order C-S model \eqref{A1} with singular communication function of order $\beta<1$.
 \vspace{0.2cm}
\begin{definition}\label{D2.1}\cite{P}
Let $\{T_n\}$ be the set of times of collision. For each interval $[T_n,T_{n+1}]$ (we assume that $T_{-1}=0$) we consider the problem \eqref{A1} and \eqref{A2} with the initial data $x(T_n)$, $v(T_n)$. Then we say $x$ solves \eqref{A1} and \eqref{A2} with $\beta<1$ on time interval [0,T] with initial data $x(0)$, $v(0)$ if and only if for all $T_n$ and all $t\in(T_n,T_{n+1})$, the function $x\in (C^1([0,T]))^{Nd}$ is a weak in $(W^{2,1}([T_n,t]))^{Nd}$ solution of the C-S model \eqref{A1} and \eqref{A2}, and the initial data are correct (i.e. $x(0)=x(T_{-1})$ and $v(0)=v(T_{-1})$) and for some $n$, we have $T\leq T_n$.
\end{definition}
\vspace{0.2cm}
Definition \ref{D2.1} actually provides a solution $(x,v)\in (C^1([0,T]), C([0,T]))$ to the C-S model \eqref{A1} with singular communication \eqref{A2} where $\beta<1$. Then with the first order reduction in Section \ref{sec:2.1}, we immediately obtain that it is equivalently to say there exists a $C^1$ solution to the first order system \eqref{B1} with nonlinear interaction \eqref{B2} in one dimensional case. More precisely, we provide following equivalence Lemma.
\vspace{0.2cm}
\begin{lemma}\label{L2.1}
For the C-S model \eqref{A1} and singular communication function \eqref{A2} with $\beta<1$ in one dimensional case, the following two assertions are equivalent to each other.
\begin{enumerate}
\item There exists a weak solution in the sense of Definition \ref{D2.1} to the second order system \eqref{A1} with singular communication function \eqref{A2}.
\item there exists a $C^1$ solution to the system \eqref{B1} with interaction \eqref{B2}.
\end{enumerate}
\end{lemma}  
\vspace{0.2cm}
\noindent The proof directly follows the first order reduction and we will omit the details. We emphasize that Lemma \ref{L2.1} shows that the well-posedness of the system \eqref{A1} with $\beta<1$ in one dimensional case is equivalent to the well-posedness of the system \eqref{B1}. As the right hand side vector field in the system \eqref{B1} with interaction \eqref{B2} is continuous, the existence of the $C^1$ solution to \eqref{B1} is guaranteed by the fundamental theory of ODE. Therefore, we only need to show the uniqueness of the system \eqref{B1} with interaction \eqref{B2}.\newline

On the other hand, for $\beta\geq 1$, avoidance of collision will be guaranteed according to \cite{C-C-M-P}. Therefore, the fundamental theory of ODE directly concludes the existence and uniqueness of the solution to the system \eqref{A1} and the equivalent first order system \eqref{B5} with interaction \eqref{B3}. However, when $\beta>1$, the multicluster formation may occur, therefore we need to mathematically define the concept of flocking and multi-cluster formation. 

\begin{definition}\label{D2.2}\cite{H-K-P-Z}
Let $X = (x_1, \ldots, x_N)$ be a solution to the first order C-S model  \eqref{B5} with nonlinear interaction \eqref{B3}.
\begin{enumerate}[(i)]
\item
The $i$-th and $j$-th C-S particles are in the same cluster if and only if the relative position $x_j - x_i$ is uniformly bounded in time:
\[\sup_{0 \le t < \infty} |x_j(t) - x_i(t)| < \infty.\]
\item
The $i$-th and $j$-th particles exhibit cluster-flocking if and only if their relative position $x_j- x_i$ has a limiting value as $t \to \infty$:
\[\lim_{t \to \infty} |x_j(t) - x_i(t)| = x^{\infty}_{ij} < \infty.\]
\item 
Let $\mathcal{I} \subset \{1,2,\ldots,N\} =:\mathcal{N}$. Then $\mathcal{I}$ is a (maximal) cluster if and only if the following two properties hold:
\begin{equation*}
\left\{
\begin{aligned}
& \lim_{t \to \infty} |x_i(t) - x_j(t) | < \infty, \quad i, j \in \mathcal{I},\\
& \lim_{t \to \infty} |x_i(t) - x_j(t)| = \infty, \quad i \in \mathcal{I}, \ j \not\in \mathcal{I}.
\end{aligned}
\right.
\end{equation*}
\item
We call a clustering number $N_c$ as the number of distinct clusters of $\{1,2,\ldots,N\}$.
\end{enumerate}
\end{definition}
\vspace{0.2cm}

\subsection{Notations} Finally, for notational simplicity, we let $x_i^0 = x_i(0)$, $v_i^0 = v_i(0)$, $X^0 = X(0)$ and $V^0 = V(0)$. More over we denote the diameters of phases, velocities and natural velocities of particles respectively as follows
\begin{equation*}
\begin{cases}
\displaystyle x_M(t) =\max_{i \in \mathcal{N}} x_i(t), & x_m(t) =\min\limits_{i \in \mathcal{N}} x_i(t), \qquad D_x(t) = \max\limits_{1 \le i,j \le N} |x_i(t)  - x_j(t)| = x_M(t) - x_m(t), \\
\displaystyle v_M(t) = \max_{1 \le i \le N} v_i(t), & v_m(t) = \min\limits_{1 \le i \le N} v_i(t),\quad \ D_{v}(t)=\max\limits_{1 \le i,j \le N} |v_i(t)-v_j(t)| = v_M(t) - v_m(t),\\
\displaystyle \nu_M = \max_{1 \le i \le N} \nu_i, & v_m = \min\limits_{1 \le i \le N} \nu_i, \quad \ \qquad \ D_{\nu} =\max\limits_{1 \le i,j \le N} |\nu_i-\nu_j| = \nu_M - \nu_m.
\end{cases}
\end{equation*}
\vspace{0.3cm}

\section{Lower order singularity}\label{sec:3}
\setcounter{equation}{0}
\vspace{0.3cm}
In this section, we will study the well-posedness of the system \eqref{A1} in one dimensional case when $0<\beta<1$. According to the previous discussion in Section \ref{sec:2}, it is equivalent to study the well-posedness of the first order system \eqref{B1} with interaction \eqref{B2}. Then, the existence of the solution is guaranteed since the right hand side of \eqref{B1} is continuous, but collision may occur in this case. Thus it is the most important to show the uniqueness of the solution around the collision time. The following lemma shows that the collision between two particles are determined by their natural velocities and relative position.

\begin{lemma}\label{CL1}
Let $X = (x_1, \ldots, x_N)$ be a $C^1$ solution to \eqref{B1} with $0<\beta<1$ and initial data $X^0$. Suppose that the initial positions of $i$-th and $j$-th particles have the order $x_i^0 < x_j^0$. Then the following trichotomy hold.
\begin{enumerate}[(i)]
\item
If $\nu_i < \nu_j$, then $x_i$ and $x_j$ will never collide:
\[|\{t_* \in (0,+\infty): x_i(t_*) = x_j(t_*)\}| = 0. \]
\item
If $\nu_i > \nu_j$, then $x_i$ and $x_j$ will collide exactly once:
\[|\{t_* \in (0, +\infty) : x_i(t_*) = x_j(t_*)\}| = 1.\]
\item
If $\nu_i = \nu_j$, then $x_i$ and $x_j$ will collide in finite time. Moreover, assume that $x_i$ and $x_j$ collide at the instant $t_c$, then $x_i$ and $x_j$ will stick together after $t_c$:
\[x_i(t) = x_j(t), \qquad t \in [t_c, +\infty).\]
\end{enumerate}
\end{lemma}
\begin{proof}
$(i)$~We prove this assertion by contradiction. Suppose $x_i$ and $x_j$ collide in finite time. Then there exists $t_* \in (0, +\infty)$ such that
\begin{equation}\label{C-1}
 x_i(t) < x_j(t), \quad t \in (0, t_*) \text{\quad and \quad} x_i(t_*) = x_j(t_*).
\end{equation}
Note that $x_j(t) - x_i(t)$ at $t = t_*$ satisfies
\begin{align*}
\begin{aligned}
\frac{d}{dt}(x_j - x_i) |_{t = t_*} = \nu_j - \nu_i + \frac{\kappa}{N} \sum_{k=1}^{N}[\Psi(x_k(t_*) - x_j(t_*)) - \Psi(x_k(t_*) - x_i(t_*))] = \nu_j - \nu_i > 0.
\end{aligned}
\end{align*}
As $X$ is a $C^1$ solution, then $\frac{d}{dt}(x_j - x_i)$ is continuous with respect to any $t \in (0, +\infty)$. Therefore, combining the above inequality $\frac{d}{dt}(x_j - x_i)(t_*) > 0$ and the continuity of $\frac{d}{dt}(x_j - x_i)$, we can find a positive constant $\delta$ such that 
\[\frac{d}{dt}(x_j - x_i)(t) > 0, \quad t \in (t_* - \delta, t_*).\]
Then we integrate above inequality on both sides from $t$ to $t_*$ to obtain
\[(x_j - x_i)(t_*) - (x_j - x_i)(t) > 0, \quad t \in (t_* - \delta, t_*).\]
Thus for any $t \in (t_* - \delta, t_*)$, we obtain $x_j(t) - x_i(t) < 0$, which obviously contradicts to \eqref{C-1} and we obtain the desired result.
\newline

\noindent $(ii)$~In the condition that $\nu_i > \nu_j$, we first claim that there will be at least one collision between $i$-th and $j$-th particles. 
Suppose not, i.e.,
\begin{equation}\label{C-2}
x_i(t) < x_j(t), \qquad t \in [0, + \infty).
\end{equation}
Then the dynamics of the difference $x_j(t) - x_i(t)$ is determined by the differential equation below for any time $t$,
\begin{align}\label{C-3}
\begin{aligned}
\dot{x}_j(t) - \dot{x}_i(t) &= \nu_j + \frac{\kappa}{N} \sum_{k=1}^N \Psi(x_k(t) - x_j(t)) - \nu_i - \frac{\kappa}{N} \sum_{k=1}^N \Psi(x_k(t) - x_i(t)) \\
&= \nu_j - \nu_i + \frac{\kappa}{N} \sum_{k=1}^N[\Psi(x_k(t) - x_j(t)) - \Psi(x_k(t) - x_i(t))].
\end{aligned}
\end{align}
Next we are going to determine the sign of the last term in \eqref{C-3}. For this purpose, we split the summation in \eqref{C-3} into three parts as follows,
\begin{align}\label{C-4}
\begin{aligned}
&\sum_{k=1}^N[\Psi(x_k(t) - x_j(t)) - \Psi(x_k(t) - x_i(t))] \\
&= - \sum_{x_k(t) \le x_i(t) \le x_j(t)}[\Psi(x_j(t) - x_k(t)) - \Psi(x_i(t) - x_k(t))] \\
&\hspace{0.5cm} - \sum_{x_i(t) \le x_k(t) \le x_j(t)} [\Psi(x_j(t) - x_k(t)) + \Psi(x_k(t) - x_i(t))] \\
&\hspace{0.5cm} + \sum_{x_i(t) \le x_j(t) \le x_k(t)} [\Psi(x_k(t) - x_j(t)) - \Psi(x_k(t) - x_i(t))].
\end{aligned}
\end{align}
Now if $x_k(t) \le x_i(t) \le x_j(t)$, it is obvious that $x_j(t) - x_k(t) \ge x_i(t) - x_k(t) \ge 0$. While on the other hand, if $x_i(t) \le x_j(t) \le x_k(t)$, we have $x_k(t) - x_i(t) \ge x_k(t) - x_j(t) \ge 0$. Therefore, we immediately obtain that 
\begin{align}\label{C-5}
\left\{
\begin{aligned}
& \Psi(x_j(t) - x_k(t)) - \Psi(x_i(t) - x_k(t)) \ge 0, \qquad \mbox{if}\ \ \ x_k(t) \le x_i(t) \le x_j(t), \\
& \Psi(x_k(t) - x_j(t)) - \Psi(x_k(t)- x_i(t)) \le 0, \qquad \mbox{if}\ \ \ x_i(t) \le x_j(t) \le x_k(t),\\
&\Psi(x_j(t) - x_k(t)) +\Psi(x_k(t) - x_i(t)) \ge 0,\qquad \mbox{if}\ \ \ x_i(t) \le x_k(t) \le x_j(t).
\end{aligned}
\right.
\end{align}
We apply the estimates \eqref{C-5} to yield the negative sign of \eqref{C-4}. Therefore, we combine \eqref{C-3}, \eqref{C-4} and \eqref{C-5} to obtain 
\[\dot{x}_j(t) - \dot{x}_i(t) \le \nu_j - \nu_i < 0.\]
We integrate above inequality on both sides from $0$ to $t$  and obtain that
\begin{equation}\label{C-6}
x_j(t) - x_i(t) \le x_j^0 - x_i^0 + (\nu_j - \nu_i)t.
\end{equation}
Let $t = \frac{x_j^0 - x_i^0}{\nu_i - \nu_j}$, we observe that for any $t \geq \frac{x_j^0 - x_i^0}{\nu_i - \nu_j}$, we have $x_j(t) - x_i(t) \le 0$, which obviously contradicts to \eqref{C-2}. Thus, we have shown that there exists at least one collision in finite time.\newline 
 
 We next prove that there will be exactly one collision if collision occurs. Due to the above analysis, we know that $t_* \in (0, +\infty)$ can be found such that
\[x_i(t) < x_j(t), \quad t \in [0, t_*) \text{\quad and \quad} x_i(t_*) = x_j(t_*).\]
In the extreme case where $x_i^0=x_j^0$, we set $t_*$ to be $0$. Then we apply the same analysis as in $(i)$ and obtain that
\[\left. \frac{d}{dt}(x_j - x_i) \right \vert_{t = t_*} = \nu_j - \nu_i < 0.\]
Hence, there exists $\delta > 0$ such that 
\begin{equation}\label{C-7}
x_j(t) - x_i(t) < 0 \text{\qquad for } t \in (t_*, t_* + \delta).
\end{equation}
Then, \eqref{C-7} and $\nu_j < \nu_i$ show that we can apply the analysis in $(i)$ to determine the dynamics after collision. According to the result of $(i)$, collision will not occur after $t_*$ and thus we can conclude that there is exactly one collision in the second case. 
\newline

\noindent $(iii)$~In the identical case, we will prove collisions occur in finite time by contradiction. Suppose there is no collision in any finite time, then we have  
\[x_i(t) < x_j(t), \qquad t \in (0, +\infty).\]
Due to the identical natural velocity condition $\nu_i = \nu_j$, we can apply the same criteria as in \eqref{C-3}, \eqref{C-4} and \eqref{C-5} to obtain
\begin{align}\label{C-8}
\begin{aligned}
\dot{x}_j(t) - \dot{x}_i(t)&= \frac{\kappa}{N} \sum_{k=1}^{N} [\Psi(x_k(t) - x_j(t)) - \Psi(x_k(t) - x_i(t))]\\
&\le - \frac{2\kappa}{N} \Psi(x_j(t) - x_i(t)) \le - \frac{2\kappa}{N(1-\beta)} (x_j(t) - x_i(t))^{1-\beta}.
\end{aligned}
\end{align}
Now we let $\left(y(t)\right)^{\frac{1}{\beta}} = x_j(t) - x_i(t)$. According to the noncollision assumption, we immediately obtain that
 \begin{equation}\label{C-9}
y(t) > 0, \qquad t\in [0,+\infty).
\end{equation} 
On the other hand, from the estimate \eqref{C-8}, we obtain that $\dot{y}(t) \le - \frac{2\kappa \beta}{N(1-\beta)}$. Then we integrate this inequality from zero to $t$ and oabtain
\[y(t) \le y(0) - \frac{2\kappa \beta}{N(1-\beta)}t, \quad t>0.\]
Then, we let $y(0) - \frac{2\kappa \beta}{N(1-\beta)}t = 0$ and construct an upper bound of collision time $t = \frac{y(0)N(1-\beta)}{2\kappa \beta} = \frac{(x_j^0 -x_i^0)^\beta N(1-\beta)}{2\kappa \beta}$. Therefore, we obtain that
\[y(t) \le 0, \qquad t \in \left[\frac{(x_j^0 -x_i^0)^\beta N(1-\beta)}{2\kappa \beta}, +\infty \right),\]
which is contradictory to \eqref{C-9}. Thus, $i$-th and $j$-th particles with same natural velocity will collide in finite time. Now we set the collision time to be $t_c$ and prove the two particles will stick together after collision. More precisely, we will show that 
\[x_i(t) = x_j(t), \qquad t\in [t_c, +\infty).\]
If not, $t_* \in (t_c, +\infty)$ can be found such that $x_i(t_*) \ne x_j(t_*)$. Assume $x_i(t_*) < x_j(t_*)$ without loss of generality. Then we define a set
\[S = \{t|x_i(t) = x_j(t), t_c \le t \le t_*\},\]
and define $T = \sup S$. Note that $S$ is not empty since $x_i(t_c) = x_j(t_c)$ and we have $x_i(T) = x_j(T)$ due to the continuity of the solution. Moreover we have $T < t_*$ and $x_i(t) < x_j(t)$ for  $t \in (T, t_*)$.
Using the same arguments as \eqref{C-8}, we obtain that
\[\dot{x}_j(t) - \dot{x}_i(t) = \frac{\kappa}{N} \sum_{j=1}^N [\Psi(x_k(t) - x_j(t)) - \Psi(x_k(t) - x_i(t))] \le 0, \qquad  t \in (T, t_*).\]
We integrate from $T$ to $t_*$ on both sides to get
\[\left. (x_j(t) - x_i(t)) \right \vert_{t = t_*} \le \left. (x_j(t) - x_i(t)) \right \vert_{t = T} = 0.\]
Therefore, we have $x_j(t_*) \le x_i(t_*)$, and this is contradictory to $x_i(t_*) < x_j(t_*)$.
Thus, we conclude the desired result.
\end{proof}
\vspace{0.2cm}
According to Lemma \ref{CL1}, the order of particles will be fixed according to the order of natural velocities after finite many collisions and then there would be no more collisions. Therefore, we can split the time line into two layers, one is collision layer which is up to a finite time and the other one is the large time layer in which collision will never occur. The non-collision and well order properties in the large time layer allow us to study the time asymptotical behavior including the formation of equilibrium state and the convergence rate. While for the uniqueness criteria, we have to understand clearly the structure of the solution in the collision layer, especially around the collision times. We will first make complete non-identical assumptions i.e. 
\[\nu_1 < \nu_2 < \cdots < \nu_N.\]
Then the other case can be studied similarly.   
\subsection{Time asymptotical property}
In this part, we will study the large time behavior of the solutions of \eqref{B1}. Actually, we will describe both the upper and lower bound of the distance between particles when time tends to infinity. Moreover, we will show the emergence of mono-cluster flocking and the exponential convergence rate.
\begin{lemma}\label{CL2}
Let $X(t)$ be a solution to \eqref{B1} and \eqref{B2} with initial data $X^0$. Assume that
\[x_1^0 < x_2^0 < \cdots < x_N^0 \text{\quad and \quad} \nu_1 < \nu_2 < \cdots < \nu_N,\]
then there exists a positive constant $C'_{m_1}$ depending on initial configurations and natural velocities such that
\[ \min_{i,j \in \mathcal{N}, i \ne j} |x_i(t) - x_j(t)| \ge C'_{m_1}, \qquad t \in [0, + \infty).\]
\begin{remark}
The time asymptotical lower bound of the distance between particles only depends on the natural velocities, and thus depends on the initial data of second order equation \eqref{A1}.
\end{remark}
\end{lemma}
\begin{proof}
Based on the assumption, it follows from Lemma \ref{CL1} $(i)$ that collision will never occur. Thus the order of the particles will be always preserved, i.e.
\[x_1(t) < x_2(t) < \cdots < x_N(t), \qquad t \in [0, +\infty).\]
Thus, the minimum of the distance between particles can be expressed below
\[\min_{1 \le i \le N-1} |x_{i+1}(t) - x_i(t)| = \min_{i,j \in \mathcal{N}, i \ne j} |x_i(t) - x_j(t)|, \qquad t \in [0,+\infty).\]
Then, the dynamics of the distance $x_{i+1}(t) - x_i(t)$ is governed by the differential equation below
\begin{align}\label{C10}
\begin{aligned}
&\dot{x}_{i+1}(t) - \dot{x}_i(t) \\
&= \nu_{i+1} + \frac{\kappa}{N} \sum_{j = 1}^N \Psi(x_j(t) - x_{i+1}(t))  - \nu_i - \frac{\kappa}{N} \sum_{j = 1}^N \Psi(x_j(t) - x_i(t)) \\
&= \nu_{i+1} - \nu_i - \frac{2\kappa}{N} \Psi(x_{i+1}(t) - x_i(t))  + \frac{\kappa}{N} \sum_{k > i+1} [\Psi(x_k(t) - x_{i+1}(t)) - \Psi(x_k(t) - x_i(t))] \\
&\hspace{0.5cm} + \frac{\kappa}{N} \sum_{k < i} [\Psi(x_k(t) - x_{i+1}(t)) - \Psi(x_k(t) - x_i(t))].
\end{aligned}
\end{align}
We are going to determine the sign of the last two terms. For this purpose, we apply the fact that $\Psi(x)$ is concave on $[0,+\infty)$ and $\Psi(0) = 0$ to have
\begin{align}\label{C11}
\begin{aligned}
&\Psi(x_k(t) - x_{i+1}(t)) - \Psi(x_k(t) - x_i(t)) \ge - \Psi(x_{i+1}(t) - x_i(t)), \qquad   k>i+1,\\
&\Psi(x_{i+1}(t) - x_k(t)) - \Psi(x_i(t) - x_k(t)) \le \Psi(x_{i+1}(t) - x_i(t)), \qquad  k < i.
\end{aligned}
\end{align}
Combining \eqref{C10} and \eqref{C11}, we immediately obtain that 
\begin{equation}\label{C12}
\dot{x}_{i+1}(t) - \dot{x}_i(t) \ge \nu_{i+1} - \nu_i - \kappa \Psi(x_{i+1}(t) - x_i(t)).
\end{equation}
If $x_{i+1}(t) - x_i(t) \le \Psi^{-1}(\frac{\nu_{i+1} - \nu_i}{\kappa})$, we exploit the monotonicity of $\Psi$ to deduce that $\dot{x}_{i+1}(t) - \dot{x}_i(t) \ge 0$, which means $x_{i+1}(t) - x_i(t)$ keeps increasing when $x_{i+1}(t) - x_i(t) \le \Psi^{-1}(\frac{\nu_{i+1} - \nu_i}{\kappa})$. Therefore, we set a positive constant $r = C_{m_{i,i+1}}$ satisfying the equation:
\[\nu_{i+1} - \nu_i - \kappa \Psi(r) = 0, \quad \text{i.e.}, \quad C_{m_{i,i+1}} = \Psi^{-1}(\frac{\nu_{i+1} - \nu_i}{\kappa}).\]
Now according to the differential inequality \eqref{C12} and above analysis, we obtain the following estimates of the relative distance $x_{i+1}(t) - x_i(t)$, 
\begin{equation*}
\begin{cases}
\displaystyle x_{i+1}(t) - x_i(t) \ge C_{m_{i,i+1}}, &\quad if \quad x_{i+1}^0 - x_i^0 \ge C_{m_{i,i+1}},\quad t \ge 0,\\
\displaystyle x_{i+1}(t) - x_i(t) \ge x_{i+1}^0 - x_i^0, &\quad if \quad x_{i+1}^0 - x_i^0 \le C_{m_{i,i+1}},\quad t\ge 0.
\end{cases}
\end{equation*}
Let $C'_{m_1} = \min_{1 \le i \le N-1} \left\{ \min \left\{x_{i+1}^0 - x_i^0, C_{m_{i,i+1}}\right\} \right\}$. Then we apply above estimates to obtain the uniform lower bound for any two adjacent particles, i.e.
\[x_{i+1}(t) - x_i(t) \ge \min_{1 \le i \le N-1} \left\{ \min \left\{x_{i+1}^0 - x_i^0, C_{m_{i,i+1}}\right\} \right\}, \qquad t\ge 0\, , i =1,2,\ldots,N-1.\]
This immediately implies the desired results $\min\limits_{i,j \in \mathcal{N}, i \ne j} |x_i(t) - x_j(t)| \ge C'_{m_1}$ for any $t\ge 0$.
\end{proof}
In the previous lemma, it does not make sense to consider the lower bound of particle distance in small time because of the existence of collisions. While when we study the upper bound of the particle distance, we need to take account of the initial layer to construct the upper bound in whole time, which allows us to prove the emergence of mono-cluster flocking and construct the convergence rate. 
\begin{lemma}
\label{CL3}
Let $X(t)$ be a solution to \eqref{B1} and \eqref{B2} with initial data $X^0$. More over, we assume the natural velocities are well ordered as below,
\[\nu_1 < \nu_2 < \cdots < \nu_N.\]
Then there exists a positive constant $C'_{M_1}$ depending on initial data and natural velocities such that
\[D_x(t) \le C'_{M_1}, \qquad t \in [0, +\infty).\]
\end{lemma}
\begin{proof}
Based on the Lemma \ref{CL1} and above analysis, there are only finite many collisions that may occur. Therefore, we can split the time line into a finite union as below
\[[0,+\infty)=\bigcup_{\alpha=1}^lI_{\alpha},\]
where the end points of $I_{\alpha}$ denote the collision time. Then in the interior of $I_{\alpha}$, the vector field in the system is analytic and thus the diameter $x_M-x_m$ is Lipschitz continuous with respect to $t$. 
By simple calculation and the concavity of $\Psi$, we derive that the dynamics of the diameter $D_x(t)$ in the interior of each $I_{\alpha}$ is governed by the differential equation below
\begin{align}\label{C13}
\begin{aligned}
&\dot{x}_M(t) - \dot{x}_m(t) \qquad, t\in I^0_{\alpha}\\
&\le D_{\nu} + \frac{\kappa}{N} \sum_{k=1}^N [\Psi(x_k(t) - x_M(t)) - \Psi(x_k(t) - x_m(t))] \\
&\le D_{\nu}- \frac{\kappa}{N} \sum_{k=1}^N \Psi(x_M(t) - x_m(t)) \\
&= D_{\nu} - \kappa \Psi(x_M(t) - x_m(t)),
\end{aligned}
\end{align}
where $D_{\nu}$ is the diameter of natural velocities. Then, we can analyze \eqref{C13} similarly as in the proof of Lemma \ref{CL2} and apply the continuous property of the solution to conclude that 
\[D_x(t) \le C'_{M_1}, \quad C'_{M_1} = \max \left\{ x_M^0 - x_m^0, C_{1,N} \right\}, \quad C_{1, N} = \Psi^{-1}(\frac{D_{\nu}}{\kappa}), \qquad t \ge 0.\]

\end{proof}
Now, we are ready to show the rate of convergence to the equilibrium for the solution to \eqref{B1} and \eqref{B2}. Actually, we will use the second order equation to estimate the convergence rate of velocity diameter first, and then we can construct the equilibrium state and show the convergence to the equilibrium state. As mentioned in Section \ref{sec:2}, in order to construct the equilibrium state, we have to assume the zero mean of natural velocities. Then we have the following lemma.
\begin{lemma}\label{CL4}
Let $X(t)$ be a solution to \eqref{B1} and \eqref{B2} with initial data $X^0$. More over, we assume the natural velocities and initial data satisfy the properties below,
\[\nu_1 < \nu_2 < \cdots < \nu_N,\quad \sum_{i=1}^N x_i^0 = 0,\quad \sum_{i=1}^N \nu_i = 0.\]
Then the mono-cluster flocking will emerge unconditionally and the convergence of the solution to the equilibrium state is exponentially fast. More precisely, there exists an equilibrium state $X^\infty=(x_1^\infty,\cdots,x_N^\infty)$ such that 
\[|X(t)-X^{\infty}|\leq C e^{-\kappa \psi(C'_{M_1})t},\]
where the constant $C'_{M_1}$ is defined in Lemma \ref{CL3} and the constant coefficient $C$ depends on initial data and natural velocities. 
\end{lemma}
\begin{proof}
Similar as in Lemma \ref{CL3}, there are only finite many collisions that may occur and we can split the time line into a finite union as below
\[[0,+\infty)=\bigcup_{\alpha=1}^lI_{\alpha},\]
where the end points of $I_{\alpha}$ denote the collision time. Then in the interior of $I_{\alpha}$, the vector field in the system is analytic and thus we can go back to the second order equation to study the dynamics of the velocity. More precisely, we take time derivative of the system \eqref{B1} and obtain the second order system \eqref{A1} as follows,
\begin{equation*}
\left\{\begin{aligned}
&\dot{x_i}=v_i,\quad t>0,\quad i=1,\cdots,N,\\
&\dot{v_i}=\frac{\kappa}{N}\sum_{j=1}^N\psi(\|x_i-x_j\|)(v_j-v_i),\\
&x_i(0)=x_i^0,\quad v_i(0)=\nu_i + \frac{\kappa}{N} \sum_{k=1}^{N} \Psi(x_k^0 - x_i^0),
\end{aligned}
\right.
\end{equation*}
where the solution of the above second order system is in the sense of Definition \ref{D2.1}. Then, the analyticity of the solution in the interior of $I_{\alpha}$ implies the diameter of velocity of the above second order system is Lipschitz continuous. Therefore, we may estimate the velocity diameter in the interior of $I_{\alpha}$ as below,
\begin{align}\label{C14}
\begin{aligned}
&\dot{v}_M(t) - \dot{v}_m(t) ,\qquad t\in I^0_{\alpha}\\
&= \frac{\kappa}{N} \sum_{j=1}^N \psi(x_j(t) - x_M(t))(v_j(t) - v_M(t)) - \frac{\kappa}{N} \sum_{j=1}^N \psi(x_j(t) - x_m(t))(v_j(t) - v_m(t)) \\
&= - \frac{\kappa}{N} \sum_{j=1}^N \psi\left(|x_j(t) - x_M(t)|\right)(v_M(t) - v_j(t))  - \frac{\kappa}{N} \sum_{j=1}^N \psi(|x_j(t) - x_m(t)|)(v_j(t) - v_m(t)).
\end{aligned}
\end{align}
According to Lemma \ref{CL3}, the diameter of position has a uniform upper bound $C'_{M_1}$ and thus we immediately have
\[\left| x_j(t) - x_M(t)\right| \le C'_{M_1}, \quad \left| x_j(t) - x_m(t)\right| \le C'_{M_1}, \qquad j \in \mathcal{N}.\]
Now, we combine \eqref{C14}, the upper bound of relative distance and the monotone decreasing property of $\psi$ on the positive real line to obtain that 
\begin{align*}
\begin{aligned}
&\dot{v}_M(t) - \dot{v}_m(t), \qquad t\in I^0_{\alpha}\\
&\le - \frac{\kappa}{N} \sum_{j=1}^N \psi(C'_{M_1}) (v_M(t) - v_j(t))   - \frac{\kappa}{N} \sum_{j=1}^N \psi(C'_{M_1}) (v_j(t) - v_m(t)) \\
&= - \frac{\kappa}{N} \psi(C'_{M_1}) \sum_{j=1}^N (v_M(t) - v_m(t)) \\
&= - \kappa \psi(C'_{M_1}) (v_M(t) - v_m(t))
\end{aligned}
\end{align*}
The above differential inequality provides the estimate of the velocity diameter in the interior of $I_{\alpha}$ as below
\[v_M(t) - v_m(t) \le (v_M(t_\alpha) - v_m(t_\alpha)) e^{-\kappa \psi(C'_{M_1})(t-t_\alpha)}, \qquad t \in I^0_{\alpha},\]
where $t_\alpha$ is the start point of the interval $I_{\alpha}$. Then, due to the continuity of the velocity, above estimates also hold for the end points of each $I_{\alpha}$. Therefore we can connect all $I_{\alpha}$ together and obtain
\[v_M(t) - v_m(t) \le (v_M^0 - v_m^0) e^{-\kappa \psi(C'_{M_1})t}, \qquad t \ge 0.\]
Then, we apply the zero mean of velocity, the conservation law of the mean velocity and the above exponential decay estimate of velocity diameter to have
\[|v_i(t)| \le v_M(t) - v_m(t) \le (v_M^0 - v_m^0) e^{-\kappa \psi(C'_{M_1})t}, \qquad t \ge 0.\]
Now, we turn back to the study on spatial variable. In fact, according to the second order system \eqref{A1} and the exponential decay estimate of the velocity, the time difference of the spatial variable can be estimated as follows,
\[|x_i(t_2)-x_i(t_1)|\leq \int_{t_1}^{t_2}|v_i(t)|dt\leq \int_{t_1}^{t_2}(v_M^0 - v_m^0) e^{-\kappa \psi(C'_{M_1})t}dt\leq C\left(e^{-\kappa \psi(C'_{M_1})t_1}-e^{-\kappa \psi(C'_{M_1})t_2}\right).\]  
Therefore, it is easy to apply Cauchy's criteria to prove that there exists a limit state $x_i^\infty$ such that $x_i(t)$ converges to $x_i^\infty$ when $t$ tends to infinity. Moreover, if we let $t_2=+\infty$ in the above estimate of time difference of $x_i(t)$, then we obtain that 
\[|x_i(t_1)-x_i^\infty|\leq Ce^{-\kappa \psi(C'_{M_1})t_1}.\] 
This finishes the proof of the lemma. 
\end{proof}

\subsection{Uniqueness}
In this part, we concern about the uniqueness of the solution to the system \eqref{B1}. According to the previous analysis, the most important point is to describe the particles' behavior around the collision times. Now let $X(t)$ be a solution of the system \eqref{B1} with initial data $X_0$, then similar as the discussion in previous sections, we can split the time interval into a finite union i.e.
\[[0,+\infty)=\bigcup_{\alpha=1}^lI_{\alpha}=[0,t_1)\cup[t_1,t_2)\cdots\cup[t_{l-1},+\infty),\]
where the end points $t_i$ $(i=1,\ldots,l-1)$ denote the collision time of the solution $X(t)$. The following lemma shows that the solution of the system \eqref{B1} is unique in $I_1$ i.e. before the first collision.
\begin{lemma}\label{CL5}
Assume the natural velocities of the system \eqref{B1} are different from each other:
\[\nu_k \ne \nu_l, \quad k \ne l, \quad k,l \in \mathcal{N},\] 
and there is another solution $\bar{X}(t)$ to the system \eqref{B1}, which has the same initial data $X_0$ as $X(t)$. Then we have 
\[\bar{X}(t)=X(t),\quad t\in I_1.\]
\end{lemma} 
\begin{proof}
If collision occurs initially, then $I_1=\{0\}$. As $\bar{X}$ and $X$ has the same initial data, we immediately conclude $\bar{X}=X$ for $t\in I_1$. Now, if there is not any collision initially, then $I_1$ is an interval with end points to be zero and the first collision time $t_1$. Therefore, the right hand-side vector field in the system \eqref{B1} is analytic in $[0,t_1)$ and thus the ODE theory guarantees the uniqueness of the solution in $I_1$. 
\end{proof}
Due to the continuity of $X$ and $\bar{X}$, we immediately conclude that $X(t_1)=\bar{X}(t_1)$. If we make $t_1$ to be the start point, then $X(t)$ and $\bar{X}(t)$ can be viewed as solutions to the system \eqref{B1} with same initial data $X(t_1)$. Therefore, in order to prove the uniqueness of the solution for any initial data, we only need to show the uniqueness of solution to the system \eqref{B1} with initial data containing collisions.

\begin{lemma}\label{CL6}
Assume the natural velocities of the system \eqref{B1} are different from each other. Moreover, we assume that the initial data contains some collisions i.e. $x_i^0 = x_j^0$ for some $i,j$. Then, the solution of system \eqref{B1} is unique.
\end{lemma}
\begin{proof}
Let $X(t)$ and $\bar{X}(t)$ are two solutions of the system \eqref{B1} and satisfy $X^0 = \bar{X}^0$ that contain at least one collision. In this case, the first collision time is $t_1=0$ and thus we have 
\[[0,+\infty)=\bigcup_{\alpha=2}^lI_{\alpha}=[t_1,t_2)\cup\cdots\cup[t_{l-1},+\infty),\]
where $t_i$ denotes the collision time of $X(t)$. Due to the continuity property of solutions and the assumption of different natural velocities, we can find a time $t_s$ such that $0<t_s<t_2$ and the particle orbits of $\bar{X}(t)$ do not collide in $(0,t_s)$. Moreover, as $t_s<t_2$, we know the particle orbits of $X(t)$ do not collide in $(0,t_s)$ according to the definition of collision time $t_2$.\newline 

\noindent$\bullet$ (Step 1) We will first prove $\bar{X}(t)=X(t)$ in $[0,t_s)$ by contradiction. Suppose not, then there exist $p \in \mathcal{N}$ and $t_* \in (0, t_s)$ such that 
\[x_p(t_*) \ne \bar{x}_p(t_*).\]
\noindent Without loss of generality, assume $x_p(t_*) - \bar{x}_p(t_*) = C > 0$. Since $x_k(t) - \bar{x}_k(t)$ is continuous and $x_k(0) - \bar{x}_k(0) = 0$, there exists $0 < \delta < t_*$ such that
\begin{equation}\label{C16}
|x_k(t) - \bar{x}_k(t)| \le \frac{C}{2}, \text{\quad for } t \in [0, \delta],\quad k=1,2,\ldots,N.
\end{equation}
Since there is not any collision in $[\delta, t_*]$ for $X(t)$, we know that $x_k(t)(k=1,2,\ldots,N)$ is analytic in $[\delta, t_*]$. The same arguments can be also applied to $\bar{X}(t)=(\bar{x}_1(t), \bar{x}_2(t), \ldots, \bar{x}_N(t))$. Therefore, $x_k(t) - \bar{x}_k(t)$ and $x_k(t) - \bar{x}_k(t) - (x_l(t) - \bar{x}_l(t))$ are both analytic in $[\delta, t_*]$. Then, according to the property of analytic functions, 
we may split $[\delta, t_*]$ into a finite union
\[[\delta, t_*] = \bigcup_{\gamma=1}^m J_{\gamma},\] 
such that both the sign and the order of $\{x_k-\bar{x}_k\}$ are preserved in each $J_{\gamma}$. More precisely, in each time interval $J_\gamma$, either $x_k(t) \ge \bar{x}_k(t)$ or $x_k(t) \le \bar{x}_k(t)$ holds, and either $x_l(t) - \bar{x}_l(t) \ge x_k(t) - \bar{x}_k(t)$ or $x_l(t) - \bar{x}_l(t) \le x_k((t) - \bar{x}_k(t)$ holds. Now, we study the dynamics of $D(t) = \max_{k \in \mathcal{N}} |x_k(t) - \bar{x}_k(t)|$ in each interval $J_\gamma$. We know that in each $J_\gamma$ there exists a fixed index $M_\gamma$ such that
\[D(t) = |x_{M_\gamma}(t) - \bar{x}_{M_\gamma}(t)|, \qquad t \in J_\gamma.\]
Suppose $x_{M_\gamma}(t) \ge \bar{x}_{M_\gamma}(t)$ in $J_\gamma$, then we have $D(t) = x_{M_\gamma}(t) - \bar{x}_{M_\gamma}(t)$. Then the dynamics of $D(t)$ in $J_\gamma$ is governed by the following equation
\begin{align}\label{C17}
\begin{aligned}
\dot{D}(t) &= \dot{x}_{M_\gamma}(t) - \dot{\bar{x}}_{M_\gamma}(t) \\
&= \nu_{M_\gamma} + \frac{\kappa}{N} \sum_{j=1}^N \Psi(x_j(t) - x_{M_\gamma}(t))  - \nu_{M_\gamma} - \frac{\kappa}{N} \sum_{j=1}^N \Psi(\bar{x}_j(t) - \bar{x}_{M_\gamma}(t)) \\
&= \frac{\kappa}{N} \sum_{l=1}^N [\Psi(x_l(t) - x_{M_\gamma}(t)) - \Psi(\bar{x}_l(t) - \bar{x}_{M_\gamma}(t))].
\end{aligned}
\end{align}
According to the definition of $D(t)$, we immediately have 
\[x_l(t)-\bar{x}_l(t)\leq |x_l(t)-\bar{x}_l(t)|\leq  \max_{k \in \mathcal{N}} |x_k(t) - \bar{x}_k(t)| = x_{M_\gamma}(t) - \bar{x}_{M_\gamma}(t), \quad l =1,\ldots,N, \  t \in J_{\gamma},\]
which is equivalent to say 
\begin{equation}\label{C18}
x_l(t)-x_{M_\gamma}(t)\leq \bar{x}_l(t)- \bar{x}_{M_\gamma}(t), \quad l =1,\dots,N, \ t \in J_{\gamma}.
\end{equation}
Now combing \eqref{C17}, \eqref{C18} and the monotonic increasing property of the interaction function $\Psi$, we obtain $\dot{D}(t)\leq 0$. Thus we can integrate $\dot{D}(t)$ from $\delta$ to $t_*$ to get $D(t_*) \le D(\delta)$. According to \eqref{C16}, we obtain 
\begin{equation}\label{C19}
D(\delta) = \max_{k \in \mathcal{N}} |x_k(\delta) - \bar{x}_k(\delta)| \le \frac{C}{2}.
\end{equation}
However, from the assumption $x_p(t_*) - \bar{x}_p(t_*) = C > 0$, we have 
\[D(t_*) = \max_{k \in \mathcal{N}} |x_k(t_*) - \bar{x}_k(t_*)| \ge |x_p(t_*) - \bar{x}_p(t_*)| = C,\]
which is a contradiction to \eqref{C19}. Thus we can conclude that 
\[X(t) = \bar{X}(t) \text{\quad in } [0,t_s).\]

\noindent$\bullet$ (Step 2) Now we prove the uniqueness of the solution in the interval $[t_s,t_2]$. By the continuity of $X$ and $\bar{X}$, we have $X(t_s)=\bar{X}(t_s)$. We know $X(t)$ contains no collision in $[t_s,t_2)$, thus we can apply the same criteria in Lemma \ref{CL5} to conclude 
\[X(t)=\bar{X}(t),\quad t\in [t_s,t_2).\]
Then by the continuity, we have $X(t_2)=\bar{X}(t_2)$, which means the second collision time and position of $\bar{X}$ are exactly the same as $X$. Therefore we conclude that 
\begin{equation}\label{C20}
X(t)=\bar{X}(t),\quad t\in [t_s,t_2].
\end{equation}

\noindent$\bullet$ (Step 3) Combining the analysis in (Step 1) and (Step 2), we prove the uniqueness of the solution to \eqref{B1} in $[0,t_2]$. Then we can start from $t_2$ and repeat the previous two steps to prove the uniqueness of the solution in $[t_i,t_{i+1}]$ for $i=2,3,\cdots,l-1$. Note that after the last collision time $t_{l-1}$, there will be no more collisions for $X(t)$ and thus $t_l=+\infty$. Finally, we combine all interval $I_\alpha$ together to claim the uniqueness of the solution to \eqref{B1} in the whole time and thus finish the proof. 

%
\end{proof}

\begin{remark}\label{R3.2}
If there exist i-th and j-th particles satisfying $\nu_i = \nu_j$, it is known from Lemma \ref{CL1} that i-th and j-th particles will collide once and after collision two particles will stick together. However, before $i$-th and $j$-th particles collide, we can apply Lemma \ref{CL6} to show that the system \eqref{A1}'s solution is unique. If $i$-th and $j$-th particles collide at $t_c$, we set $t_c$ as initial time. Then, we consider the uniqueness of the solution of the following system:
\begin{equation}\label{equal nu}
\begin{cases}
\displaystyle \dot{x}_k(t) = \nu_k + \frac{\kappa}{N} \sum_{l \ne i,j} \Psi(x_l(t) - x_k(t)) + \frac{2\kappa}{N} \Psi(x_i(t) - x_k(t)), \qquad k \ne i,j, \\
\displaystyle \dot{x}_i(t) = \nu_i + \frac{\kappa}{N} \sum_{l \ne j} \Psi(x_l(t) - x_i(t)), \\
x_j(t) = x_i(t).
\end{cases}
\end{equation}
In the sense that $N$ particles' dynamics are governed by system \eqref{equal nu}, we can repeat the process of analysis as previous to prove the uniqueness of the solution to the system \eqref{B1} and \eqref{B2}.
\end{remark}

\subsection{Conclusion} Now, we are ready to prove the main theorem in Section \ref{sec:3}. We will first provide the theorem for the first order system \eqref{B1} and \eqref{B2}. Then the results of the second order system \eqref{A1} will be provided as a corollary.  
\begin{theorem}\label{T3.1}
There exists a unique $C^1$ solution to the first order C-S model \eqref{B1} and \eqref{B2} with $0<\beta<1$. Moreover, assume that the natural velocities and initial data satisfy the properties below,
\[\sum_{i=1}^N x_i^0 = 0,\quad \sum_{i=1}^N \nu_i = 0.\]
Then, for any two particles such that $x_i^0<x_j^0$, the following assertions hold:
\begin{enumerate}
\item
If $\nu_i < \nu_j$, then $x_i$ and $x_j$ will never collide:
\[|\{t_* \in (0,+\infty): x_i(t_*) = x_j(t_*)\}| = 0. \]
\item
If $\nu_i > \nu_j$, then $x_i$ and $x_j$ will collide exactly once:
\[|\{t_* \in (0, +\infty) : x_i(t_*) = x_j(t_*)\}| = 1.\]
\item
If $\nu_i = \nu_j$, then $x_i$ and $x_j$ will collide in finite time. Moreover, assume that $x_i$ and $x_j$ collide at the instant $t_c$, then $x_i$ and $x_j$ will stick together after $t_c$:
\[x_i(t) = x_j(t), \qquad t \in [t_c, +\infty).\]
\item 
There exist positive constants $C'_{m_1}$ and $C'_{M_1}$ such that
\begin{equation*}
\left\{\begin{aligned}
& \lim_{t\rightarrow+\infty}|x_i(t)-x_j(t)|\geq C'_{m_1},\\
& |x_i(t)-x_j(t)|\leq C'_{M_1},\quad 1\leq i\neq j\leq N,\quad t\geq 0. 
\end{aligned}
\right.
\end{equation*}
\item
Unconditional flocking occurs and there exists an equilibrium state $X^\infty=(x_1^\infty,\cdots,x_N^\infty)$ such that 
\[|X(t)-X^{\infty}|\leq C e^{-\kappa \psi(C'_{M_1})t}.\]
\end{enumerate}
\end{theorem}
\begin{proof}
The uniqueness follows Lemma \ref{CL5}, Lemma \ref{CL6} and Remark \ref{R3.2}. The assertions $(1)-(3)$ are proved in Lemma \ref{CL1}. The assertion $(4)$ is derived from Lemma \ref{CL2} and Lemma \ref{CL3}. Finally, the assertion $(5)$ follows from Lemma \ref{CL4}. 
\end{proof}
\begin{figure}[t]
\centering
\mbox{
\subfigure[Collision and flocking]{\includegraphics[width=0.46\textwidth]{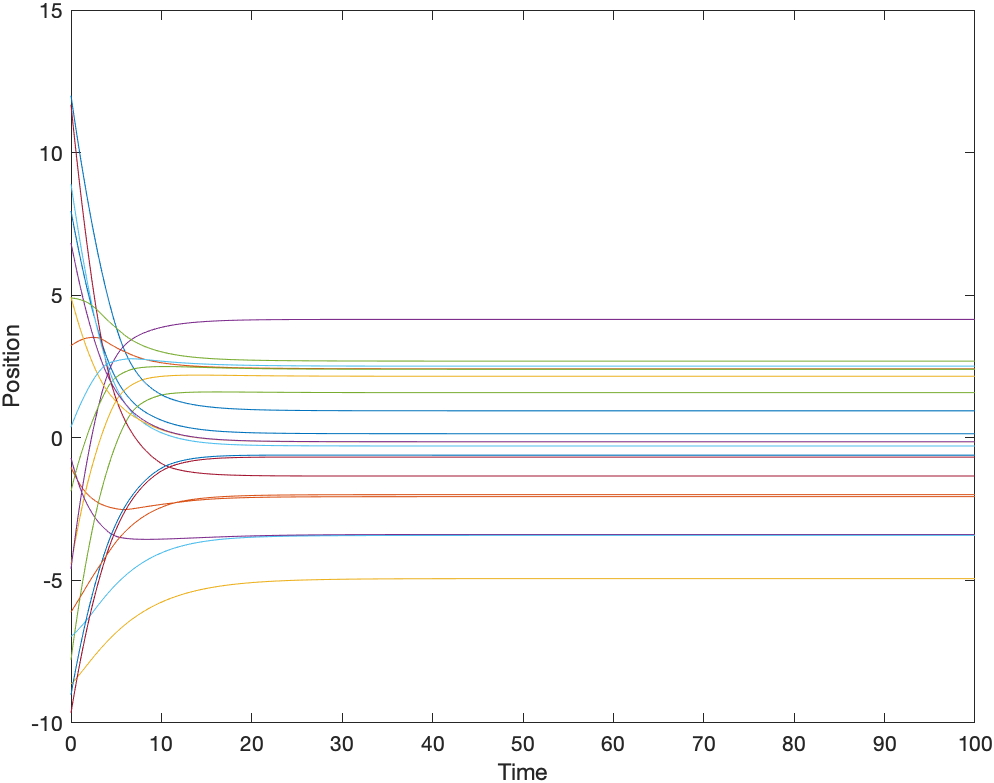}
\label{Fig1:sin}}
}
\centering
\mbox{
\subfigure[Partial sticking]{\includegraphics[width=0.46\textwidth]{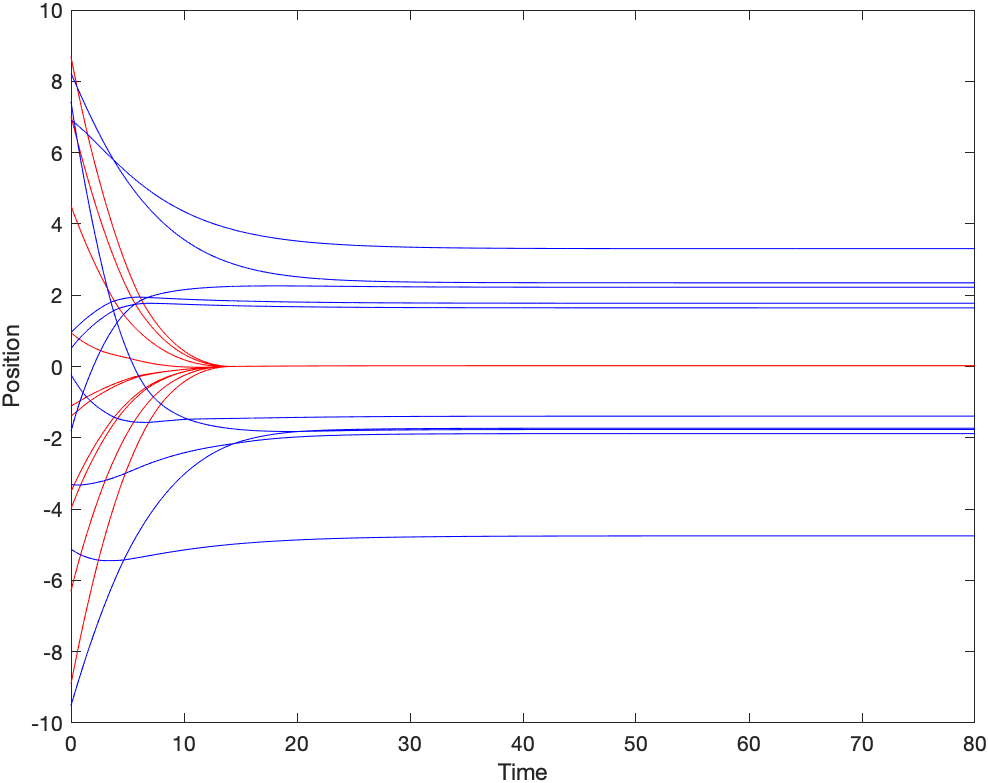}
\label{Fig1:partial}}
\hspace{0.1cm}
\subfigure[Full sticking]{\includegraphics[width=0.46\textwidth]{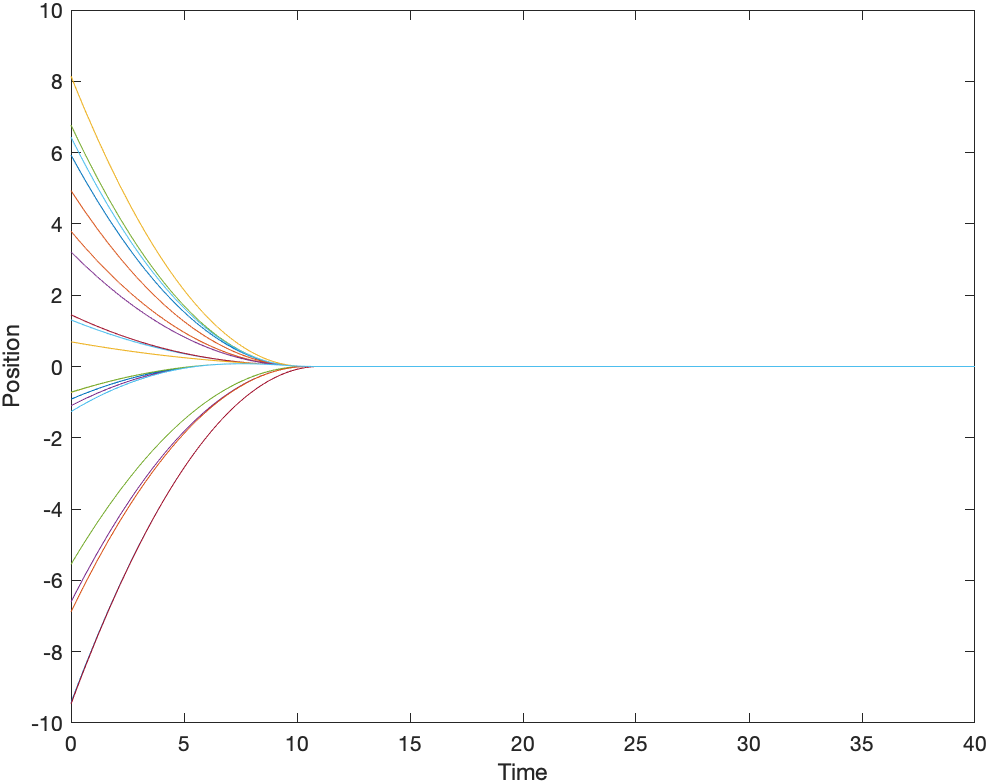}
\label{Fig1:stick}}
}
\centering\caption{Asymptotical behavior when $\beta<1$}
\label{Fig1}
\end{figure}
In Figure \ref{Fig1:sin}, we can see clearly the collision behavior in the initial layer and the asymptotical flocking emergence. In Figure \ref{Fig1:partial}, the red curves shows the trajectories of particles with same natural velocities, and thus they stick together after collisions. Figure \ref{Fig1:stick} shows the full sticking and finite time flocking emergence when all particles are of identical natural velocities.

\begin{corollary}\label{C3.1}
For generic initial data $(X^0, V^0)$, there exists a unique solution $(X(t),V(t))\in (C^1([0,T])^N, C([0,T])^N)$ to the one dimensional second order C-S model \eqref{A1} with $0<\beta<1$. Moreover, assume that the initial data satisfy the properties below,
\[\sum_{i=1}^N x_i^0 = 0,\quad \sum_{i=1}^N v_i = 0.\]
Then all the assertions in Theorem \ref{T3.1} still hold. 
\end{corollary}
\begin{proof}
The proof directly follows from the first order reduction in Section \ref{sec:2.1} and Lemma \ref{L2.1}.\newline
\end{proof}
\vspace{0.3cm}
\section{Higher order singularity}\label{sec:4}
\setcounter{equation}{0}
\vspace{0.3cm}
In this section, we discuss the Cucker-Smale system with singular communication weight $\psi(r) =  \frac{1}{|r|^\beta}$ and $\beta > 1$. In this case, we will first construct an uniform-in-time lower bound of the distance between any two particles. Next, as $\beta>1$, the interaction is short range due to the integrability of $\psi$ at infinity. Therefore, the emergence of mono-cluster flocking depends on the initial configuration and multi-cluster formation may occur. Similar as in \cite{H-K-P-Z}, we will construct the critical value $\kappa_c$ so that mono-cluster formation will emerge if and only if $\kappa>\kappa_c$ and multi-cluster formation will emerge if and only if $\kappa\leq \kappa_c$.  
\subsection{Uniform lower bound}\label{sec:4.1}
In this subsection, we will show the uniform-in-time lower bound of the distance between particles. In \cite{C-C-M-P}, the authors have already proved the non-existence of collisions when $\beta>1$, therefore the global existence and uniqueness of the solution are guaranteed by the ODE theory. Moreover, the order of particles will never change and thus, without loss of generality, we may assume the order of particles as below
\[x_1(t) < x_2(t) < \cdots < x_N(t), \qquad  \forall t \ge 0.\]
We denote $D_{ij}=|x_i-x_j|$ and thus there are $ \frac{N(N-1)}{2}$ quantities $D_{ij}$. Therefore, we may order $D_{ij}$ in any fixed $t$ and have 
\[D_1(t) \le D_2(t) \le \ldots \le D_\frac{N(N-1)}{2}(t),\]
where $D_{l}(t) = |x_j(t) - x_i(t)|$ for some $i, j \in \mathcal{N} = \{1,2, \ldots, N\}, l = 1,2,\ldots, \frac{N(N-1)}{2}.$
According to the analyticity of the solution $x_i$, we have $D_i(t)$ is Lipschitz continuous with respect to $t$. Moreover, we can separate the time line to be an at most countable union of sets with the properties below 
\begin{equation}\label{D1}
\left\{\begin{aligned}
&[0, +\infty) = \cup_\alpha \bar{I}_\alpha,\quad I_{\alpha}=(t_{\alpha-1},t_{\alpha}),\quad t_0=0,\quad \alpha=1,2\cdots,\\
&D_1(t) < D_2(t) < \ldots < D_\frac{N(N-1)}{2}(t),\quad t\in I_\alpha.
\end{aligned}
\right.
\end{equation}
 Now we state our first theorem in this section which is corresponding to the uniform-in-time lower bound. This is a more precise description of the avoidance of collision.
\begin{theorem}\label{TD1}
Let $X$ be a solution to system \eqref{B5} with initial configuration $X^0$ and natural velocities $\nu_i$ satisfying the following properties
\[\sum_{i=1}^N x_i^0 = 0,\quad \sum_{i=1}^N \nu_i = 0,\quad x_i^0\neq x_j^0,\quad 1\leq i\neq j\leq N.\]
Then for $\beta>1$, there exists a positive constant $C_m$ such that the minimum distance of any two particles is uniformly bounded from below by $C_m$ i.e.
\[D_1(t)=\min_{i,j}|x_i(t) - x_j(t)|\geq C_m,\quad t\geq 0.\]
\end{theorem}
\begin{remark}
When $\beta\geq 1$, the communication function $\Phi$ tends to infinity around origin. Therefore, the method in Lemma \ref{CL2} cannot be applied to yield the lower bound.
\end{remark}
\noindent We will apply the inductive method to prove Theorem \ref{TD1}. Now, as collision does not exist, we immediately conclude that $D_\frac{N(N-1)}{2}(t)=x_N(t) - x_1(t)=D_x(t)$. In the following lemma, we will construct the uniform-in-time lower bound for the diameter $D_x(t)$, which is the first step of the induction.
\begin{lemma}\label{LD1}
Let $X$ be a solution to system \eqref{B5} with initial configuration $X^0$ and natural velocities $\nu_i$ satisfying the following properties
\[\sum_{i=1}^N x_i^0 = 0,\quad \sum_{i=1}^N \nu_i = 0,\quad x_i^0\neq x_j^0,\quad 1\leq i\neq j\leq N.\]
When $\beta>1$, the diameter of position $D_x(t)$ has uniform-in-time lower bound. More precisely, there exists a positive constant $C_L$ such that 
\[D_x(t) \ge C_L, \quad C_L=\min\{1,\ D_x(0),\ \Phi^{-1}(- \frac{N D_{\nu}}{2\kappa})\} \qquad t \ge 0,\]
where $\Phi^{-1}$ is the inverse function of $\Phi(x)$ for $x>0$.
\end{lemma}
\begin{proof}
We first study the dynamics of the diameter in $I_1$. If $D_x(t)\geq 1$ for all $t\in I_1$, then, by the definition of $C_L$,  we have $D_x(t)\geq C_L$ for $t\in I_1$. On the other hand, if $D_x(t)<1$ for some $t_0\in I_1$, we will prove $D_x(t_0)\geq C_L$ by contradiction. Suppose not, then we have $D_x(t_0)<C_L$. We immediately obtain that $t_0>0$ because of the fact $D_x(t_0)< C_L\leq D_x(0)$. Therefore, due to the continuity of the $D_x(t)$, we can find at least one time $t_1$ such that
\[0\leq t_1<t_0,\quad D_x(t_1)= C_L.\]
Next, the existence of $t_1$ allows us to define $t^*=\sup\{s\ |\ 0\leq s<t_0,\quad D_x(s)= C_L\}$. Then, in the interval $I_{t_0}=[t^*,t_0]$, we have the following properties
\begin{equation}\label{D2}
\left\{\begin{aligned}
&D_x(t^*)=C_L,\\
& D_x(t)< C_L\leq 1,\quad t^*<t\leq t_0.
\end{aligned}
\right.
\end{equation}
According to above analysis, the dynamics of the diameter $D_x(t)$ is governed by the differential equation below:  
\begin{align}\label{D3}
\begin{aligned}
&\frac{d}{dt}(D_x(t))=\dot{x}_N(t) - \dot{x}_1(t) \\
&= \nu_N + \frac{\kappa}{N} \underset{j \ne N}{\sum_{j =1}^N} \Phi(x_j(t) - x_N(t)) -\nu_1 - \frac{\kappa}{N} \underset{j \ne 1}{\sum_{j =1}^N} \Phi(x_j(t) - x_1(t)) \\
&= \nu_N - \nu_1 - \frac{2\kappa}{N} \Phi(x_N(t) - x_1(t)) - \frac{\kappa}{N} \sum_{j=2}^{N-1} [\Phi(x_N(t) -x_j(t)) + \Phi(x_j(t) - x_1(t))].
\end{aligned}
\end{align}
Next, according to \eqref{D2}, we know the diameter $x_N(t) - x_1(t) \leq 1$ in the interval $I_{t_0}$. Therefore we obtain
\[0 < x_N(t) - x_j(t) \le 1, \quad 0 < x_j(t) - x_1(t) \le 1, \qquad j =2,3,\ldots,N-1.\]
Then, we combine above inequalities and the definition of $\Phi$ in \eqref{B3} to obtain
\begin{equation}\label{D4}
\Phi(x_N(t) - x_j(t)) \le 0, \quad \Phi(x_j(t) - x_1(t)) \le 0, \qquad j =2,3,\ldots,N-1.
\end{equation}
Now, we combine \eqref{D3} and \eqref{D4} to obtain the differential inequality of the diameter as below
\begin{equation}\label{D5}
\dot{D}_X(t) \ge - D_{\nu} - \frac{2\kappa}{N} \Phi(x_N(t) - x_1(t))=- D_{\nu} - \frac{2\kappa}{N} \Phi(D_x(t)),\quad t\in I_{t_0}.
\end{equation}
According to \eqref{D2} and the definition of $C_L$, we know the diameter $D_x(t)\leq C_L\leq \Phi^{-1}(- \frac{N D_{\nu}}{2\kappa})$ in the time interval $I_{t_0}$. Then, as $\Phi$ in \eqref{B3} is monotonic increasing on the positive real line and tends to negative infinity around origin, we can obtain 
\[\dot{D}_X(t) \ge- D_{\nu} - \frac{2\kappa}{N} \Phi(D_x(t))\geq 0,\quad t\in I_{t_0}.\]
 Integrate above differential inequality on both sides and we obtain $D_x(t_0)\geq D_x(t^*)=C_L$, which obviously contradicts to \eqref{D2}. Thus we obtain $D_x(t_0)\geq C_L$. Due to the arbitrary choice of $t_0$, we conclude that $D_x(t)\geq C_L$ in the whole interval $I_1$. Now suppose $D_x(t)\geq C_L$ in $I_\alpha$, then we have $D_x(t_\alpha)\geq C_L$, where $t_\alpha$ is the end point of $I_\alpha$. Moreover, by the definition of $I_{\alpha}$, we know $t_{\alpha}$ is the start point of $I_{\alpha+1}$, and  thus we can apply the same argument as above to conclude 
 \[D_x(t)\geq \min\{1,\ D_x(t_\alpha),\ \Phi^{-1}(- \frac{N D_{\nu}}{2\kappa})\}\geq  \min\{1,\ C_L,\ \Phi^{-1}(- \frac{N D_{\nu}}{2\kappa})\} = C_L,\quad t\in I_{\alpha+1},\]
 where the last equality holds due to the definition of $C_L.$
Therefore, we apply the induction principle to conclude that $D_x(t)\geq C_L$ for all $t\geq 0$.
\end{proof}
\vspace{0.2cm}
Next, we will inductively prove the existence of uniform-in-time lower bound for any $D_p$.

\begin{lemma}\label{LD2}
Let $X$ be a solution to system \eqref{B5} with initial configuration $X^0$ and natural velocities $\nu_i$ satisfying the following properties
\[\sum_{i=1}^N x_i^0 = 0,\quad \sum_{i=1}^N \nu_i = 0,\quad x_i^0\neq x_j^0,\quad 1\leq i\neq j\leq N.\] 
For $\beta>1$ and a given integer $p$ such that $0<p<\frac{N(N-1)}{2}$, we suppose there exists a positive constant $C_1 > 0$ such that $D_q \ge C_1$ for every $q > p$ and any $t\geq 0$. Then there exists a positive constant $C_2 > 0$ such that 
\[C_2 = \min \Big\{ \Phi^{-1} \Big(-\frac{N}{2\kappa}[2 \kappa \max \{|\Phi(C_1)|, |\Phi(+\infty)|\} + D_{\nu}]\Big), 1, D_p(0),\frac{C_1}{2}\Big\},\quad D_p(t) \ge C_2,\quad t\geq 0,\]
where $\Phi^{-1}$ is the inverse function of $\Phi(x)$ for $x>0$.
 \end{lemma}
\begin{proof}
Let's first study the situation in the interval $I_1$. Then, by the definition of $I_1$, there exist indexes $i,j \in \mathcal{N}$ such that
\[ D_p(t) = |x_i(t) - x_j(t)|,\quad t\in I_1.\]
In Lemma \ref{LD1}, we already show the existence of the lower bound of the diameter of spatial variable. Therefore, without loss of generality, we may assume $x_i(t) < x_j(t), t\in I_1$. Then, according to \eqref{D1}, we can construct a set
\[\mathcal{C}_{ij} = \{n \in \mathcal{N} \, \vert \, \exists \  m \in \mathcal{N}, s.t., |x_n(t) - x_m(t)| > D_p(t), t \in I_1 \}.\]

\noindent$\bullet$ (Case 1): $i \notin \mathcal{C}_{ij}$ and $j \in \mathcal{C}_{ij}$. In this case, we apply the property that $i \notin \mathcal{C}_{ij}$ to obtain for any $m \in \mathcal{N}$ that,
\begin{equation}\label{D6}
|x_i(t) - x_m(t)| \le D_p(t) = |x_i(t) -x_j(t)|, \quad t \in I_1.
\end{equation}
While due to $j \in \mathcal{C}_{ij}$, we obtain that there exists an integer $m_j \in \mathcal{N}$ such that 
\begin{equation}\label{D7} 
|x_j(t) - x_{m_j}(t)| > D_p(t) = |x_i(t) - x_j(t)|, \quad t \in I_1.
\end{equation}
Therefore, according to \eqref{D1}, we can find one integer $q_{j,m_j} > p$ such that $D_{q_{j,m_j}}(t) = |x_j(t) - x_{m_j}(t)| $. Then by the condition in the statement of Lemma \ref{LD2}, we use \eqref{D7} to obtain that $D_{q_{j,m_j}}(t)\geq C_1$. Thus, we can apply the triangle inequality, \eqref{D6} and \eqref{D7} to have 
\[C_1 \le |x_j(t) - x_{m_j}(t)| \le |x_j(t) - x_i(t)| + |x_i(t) - x_{m_j}(t)| \le 2D_p(t),\quad t \in I_1,\]
which immediately implies that $D_p(t) \ge \frac{C_1}{2}\geq C_2$ for $t \in I_1$.\newline

\noindent$\bullet$ (Case 2): $i \in \mathcal{C}_{ij}$ and $j \notin \mathcal{C}_{ij}$. In this case, we may use the same arguments as in Case 1 to obtain that $D_p(t) \ge \frac{C_1}{2}\geq C_2$ for $t \in I_1$.\newline
	
\noindent$\bullet$ (Case 3): $i \notin \mathcal{C}_{ij}$ and $j \notin \mathcal{C}_{ij}$. In this case, according to the definition of $\mathcal{C}_{ij}$, we have for any $m \in \mathcal{N}$ that, 
\[|x_i(t) - x_m(t)| \le D_p(t) = |x_i(t) - x_j(t)|,\quad |x_j(t) - x_m(t)| \le D_p(t) = |x_i(t) - x_j(t)|,\quad t\in I_1.\]
Then, it is obvious that $D_p(t)$ has to be the diameter $D_x(t)=D_{\frac{N(N-1)}{N}}(t)$. However, we already assume in the present lemma that $p < \frac{N(N-1)}{2}$. Therefore, this case does not exist. \newline

\noindent$\bullet$ (Case 4): $i \in \mathcal{C}_{ij}$ and $j \in \mathcal{C}_{ij}$. In this case, we use the definition of $\mathcal{C}_{ij}$ to obtain that there exist $m_i, m_j \in \mathcal{N}$ such that
\[ |x_i(t) - x_{m_i}(t)| > D_p(t) = |x_i(t) - x_j(t)|, \quad |x_j(t) - x_{m_j}(t)| > D_p(t) = |x_i(t) - x_j(t)|, \qquad t \in I_1.\]
Then we consider two sub-cases.\newline 	
	
\noindent$\diamond$ $(i)$ If there exists some integer $m \in \mathcal{N}$ such that 
\[ |x_i(t) - x_m(t)| > D_p(t), \quad |x_j(t) - x_m(t)| < D_p(t),\]
then, due to $|x_i(t) - x_m(t)| > D_p(t)$, we can find $q_{i,m} > p$ such that $D_{q_{i,m}} = |x_i(t) - x_m(t)| > D_p(t)$ and thus $D_{q_{i,m}}\ge C_1$. Therefore, we apply the triangle inequality to have
\[C_1 \le |x_i(t) - x_m(t)| \le |x_i(t) - x_j(t)| + |x_j(t) - x_m(t)| < 2D_p(t),\]
which immediately implies that $D_p(t) \ge \frac{C_1}{2}\geq C_2$ for $ t \in I_1$. For the other situation that there exists an integer $m \in \mathcal{N}$ such that
\[|x_i(t) - x_m(t)| < D_p(t), \quad |x_j(t) - x_m(t)| > D_p(t),\]
we can use the same arguments to obtain the same results.\newline

\noindent$\diamond$ $(ii)$ Now we only need to study the case that for any $m \in \mathcal{N}$, one of the following holds,
\begin{equation}\label{D8}
\left\{
\begin{aligned}
&|x_i(t) - x_m(t)| > D_p(t) \ \mbox{and}\  |x_j(t) - x_m(t)| > D_p(t),\quad t\in I_1,\\
&|x_i(t) - x_m(t)| < D_p(t) \ \mbox{and}\   |x_j(t) - x_m(t)| < D_p(t),\quad t\in I_1.
\end{aligned}
\right.
\end{equation}
We will prove $D_p(t)\geq C_2$ by contradiction in this case. Suppose there exists a $t_0\in I_1$ such that $D_p(t)< C_2$. Then we apply the same argument in Lemma \ref{LD1} to imply that there exists an interval $I_{t_0}=[t^*,t_0]$, which has the following properties
\begin{equation}\label{D9}
\left\{\begin{aligned}
&D_p(t^*)=C_2,\\
& D_p(t)< C_2\leq 1,\quad t^*<t\leq t_0.
\end{aligned}
\right.
\end{equation}
On the other hand, we follow \eqref{D8} to separate $\mathcal{N}$ into two sets respectively as below,
\begin{equation}\label{D10}
\left\{
\begin{aligned}
&A = \{ m \in \mathcal{N} \,|\, |x_i(t) - x_m(t)| > D_p(t) \text{ and } |x_j(t) - x_m(t)| > D_p(t), t \in I_1\},\\
&B = \{ m \in \mathcal{N} \,|\, |x_i(t) - x_m(t)| < D_p(t) \text{ and } |x_j(t) - x_m(t)| < D_p(t), t \in I_1\}.
\end{aligned}
\right.
\end{equation}
 Then, the dynamics of $D_p(t)$, i.e. the distance between $x_i$ and $x_j$, is governed by the following differential equation,
\begin{align}\label{D11}
\begin{aligned}
\dot{D}_p &=\dot{x}_j(t) - \dot{x}_i(t),\qquad \qquad t\in I_{t_0}  \\
&= \nu_j + \frac{\kappa}{N} \sum_{k \ne j} \Phi(x_k(t) - x_j(t))  - \nu_i - \frac{\kappa}{N} \sum_{k \ne i} \Phi(x_k(t) - x_i(t)) \\
&= \nu_j - \nu_i - \frac{2\kappa}{N} \Phi(x_j(t) - x_i(t))  + \frac{\kappa}{N} \sum_{k \ne i,j}[\Phi(x_k(t) - x_j(t)) - \Phi(x_k(t) - x_i(t))] \\
& = -\frac{\kappa}{N} \sum_{k \in B} [\Phi(x_j(t) - x_k(t)) + \Phi(x_k(t) - x_i(t))]\\
&\hspace{0.5cm} + \frac{\kappa}{N} \sum_{k \in A} [\Phi(x_k(t) - x_j(t)) - \Phi(x_k(t) - x_i(t))]\\
&\hspace{0.5cm}+\nu_j -\nu_i - \frac{2\kappa}{N} \Phi(x_j(t) - x_i(t))\\
& =\sum_{i=1}^3\mathcal{I}_i.	
\end{aligned}
\end{align}
\noindent$\star$ ($\mathcal{I}_1$): For $\mathcal{I}_1=-\frac{\kappa}{N} \sum_{k \in B} [\Phi(x_j(t) - x_k(t)) + \Phi(x_k(t) - x_i(t))]$, we apply the definition of the set $B$ and the assumption $x_i<x_j$ in $I_1$ to obtain that 
\[x_i(t)<x_k(t)<x_j(t),\quad k\in B,\quad t\in I_{t_0}.\]
Therefore, we immediately have 
\begin{equation}\label{D12}
0 < x_j(t) -x_k(t) < D_p(t), \quad 0< x_k(t) - x_i(t) < D_p(t),\quad t\in I_{t_0}.
\end{equation}
Combing \eqref{D9}, \eqref{D12} and the structure of the communication $\Phi$ in \eqref{B3}, we obtain that
\begin{equation}\label{D13}
\left\{
\begin{aligned}
&0<x_j(t) -x_k(t)\leq 1,\quad 0< x_k(t) - x_i(t) \leq 1,\quad\  t\in I_{t_0},\quad k\in B,\\
& \Phi(x_j(t) - x_k(t)) \le 0, \quad \Phi(x_k(t) - x_i(t)) \le 0,\qquad t\in I_{t_0},\quad k\in B.
\end{aligned}
\right.
\end{equation}
Now we apply \eqref{D13} to obtain the estimate of $\mathcal{I}_1$ as below
\begin{equation}\label{D14}
\mathcal{I}_1=-\frac{\kappa}{N} \sum_{k \in B} [\Phi(x_j(t) - x_k(t)) + \Phi(x_k(t) - x_i(t))]\geq 0.\quad t\in I_{t_0}.
\end{equation}
\noindent$\star$ ($\mathcal{I}_2$): For $\mathcal{I}_2=\frac{\kappa}{N} \sum_{k \in A} [\Phi(x_k(t) - x_j(t)) - \Phi(x_k(t) - x_i(t))]$, we apply the definition of the set $A$ to imply that 
\begin{equation}\label{D15}
|x_k(t)-x_j(t)|\geq C_1,\quad |x_k(t)-x_i(t)|\geq C_1.
\end{equation}
Moreover, the structure of $\Phi$ in \eqref{B3} shows that $\Phi(s)$ is monotonic increasing for $s>0$ and has a finite positive limit $\Phi(+\infty)$ at positive infinity. Therefore, we combine \eqref{D15} and the structure of $\Phi$ to obtain 
\begin{equation}\label{D16}
\left\{
\begin{aligned}
&|\Phi(x_k(t) -x_j(t))| \le \max \{|\Phi(C_1)|, |\Phi(+\infty)|\},\quad  t\in I_{t_0},\\
&|\Phi(x_k(t) -x_i(t))| \le \max \{|\Phi(C_1)|, |\Phi(+\infty)|\},\quad  t\in I_{t_0}.
\end{aligned}
\right.
\end{equation} 
Then, we apply \eqref{D16} to obtain the estimate of $\mathcal{I}_2$ as below
\begin{equation}\label{D17}
\mathcal{I}_2=\frac{\kappa}{N} \sum_{k \in A} [\Phi(x_k(t) - x_j(t)) - \Phi(x_k(t) - x_i(t))]\geq -2\kappa \max \{|\Phi(C_1)|, |\Phi(+\infty)|\},\quad  t\in I_{t_0}.
\end{equation}

\noindent Now, we combine the estimates \eqref{D14}, \eqref{D17} and the equation \eqref{D11} to obtain the differential inequality of $D_p(t)$ as below,
\begin{equation}\label{D18}
\dot{D}_p(t) \ge - D_{\nu} - \frac{2\kappa}{N} \Phi(D_p(t)) - 2 \kappa \max \{|\Phi(C_1)|, |\Phi(+\infty)|\},\quad t\in I_{t_0}.
\end{equation}
Next, according to the definition of $C_2$, we have $C_2\leq \Phi^{-1} \Big(-\frac{N}{2\kappa}[2 \kappa \max \{|\Phi(C_1)|, |\Phi(+\infty)|\} + D_{\nu}]\Big)$, which is equivalent to
\begin{equation}\label{D19}
- D_{\nu} - \frac{2\kappa}{N} \Phi(C_2) - 2 \kappa \max \{|\Phi(C_1)|, |\Phi(+\infty)|\}\geq 0,\quad t\in I_{t_0}.
\end{equation} 
Combing \eqref{D9}, \eqref{D18}, \eqref{D19} and the monotonic property of the communication function $\Phi$, we obtain 
\[\dot{D}_p\geq 0, \quad t\in I_{t_0}.\]
 Therefore, we have $D_p(t_0)\geq D_p(t^*)=C_2$ which obviously contradicts to \eqref{D9}. Therefore, we obtain that $D_p(t)\geq C_2$ for all $t\in I_{1}$ in Case 4. Combining all the analysis from (Case 1) to (Case 4), we can conclude that $D_p(t)\geq C_2$ for all $t\in I_{1}$. 
 
Next, suppose $D_p(t)\geq C_2$ in $I_\alpha$, then we have $D_p(t_\alpha)\geq C_2$, where $t_\alpha$ is the end point of $I_\alpha$. Moreover, by the definition of $I_{\alpha}$, we know $t_{\alpha}$ is the start point of $I_{\alpha+1}$, and  thus we can apply the same argument as above to conclude for any $t\in I_{\alpha+1}$ that,
 \begin{align*}
 \begin{aligned}
D_p(t)&\geq \min \Big\{ \Phi^{-1} \Big(-\frac{N}{2\kappa}[2 \kappa \max \{|\Phi(C_1)|, |\Phi(+\infty)|\} + D_{\nu}]\Big), 1, D_p(t_\alpha),\frac{C_1}{2}\Big\}\\
 &\geq \min \Big\{ \Phi^{-1} \Big(-\frac{N}{2\kappa}[2 \kappa \max \{|\Phi(C_1)|, |\Phi(+\infty)|\} + D_{\nu}]\Big), 1, C_2,\frac{C_1}{2}\Big\}\\
 &= C_2,
 \end{aligned}
 \end{align*}
 where the last equality holds due to the definition of $C_2.$
Therefore, we apply the induction principle to conclude that $D_p(t)\geq C_2$ for all $t\geq 0$.\newline
\end{proof}
\noindent $\mathbf{Proof\ of\  Theorem\  \ref{TD1}}$: According to the Lemma \ref{LD1} and Lemma \ref{LD2} and the finite number of particles, we can inductively construct the uniform lower bound for the distance between particles. More precisely, we can follow the process in Lemma \ref{LD2} to construct a positive constant $C_m$ such that 
\[D_1(t)=\min_{i,j \in \mathcal{N}}|x_i(t)-x_j(t)|\geq C_m,\quad t\geq 0.\]
\qed

\begin{remark}\label{R4.2}
By simple calculation, we can derive the order of the lower bound with respect to particle number $N$. In fact, we set $\Phi^\infty := \lim_{x \to +\infty} \Phi(x) = \frac{1}{\beta - 1} < +\infty$. Then, if there exist constants $x>0$ and $C < 0$ such that $\Phi(x) = C$, we can obtain
\begin{align}\label{B4}
\left\{
\begin{aligned}
& x = \Phi^{-1}(C) = e^C,\quad \beta = 1;\\
& x = \Phi^{-1}(C) = \left[\frac{1}{1 - C(\beta - 1)}\right]^{\frac{1}{\beta -1}},\quad \beta > 1.
\end{aligned}
\right.
\end{align}
Then we can substitute the equality \eqref{B4} into the induction process and obtain the estimates for the lower bound of particle distance as below,
\[C_m \geq C \exp(-\frac{N^2\log N}{\beta-1}),\quad \beta>1.\]
\end{remark}
\vspace{0.05cm}
\subsection{Critical coupling strength for mono-cluster formation}
In this subsection, we will construct the critical coupling strength $\kappa_c$ for mono-cluster formation of the Cucker-Smale model. According to the collision avoidance results in the previous section, we can immediately conclude the order preservation of the spatial variable for all time, i.e. 
\begin{equation}\label{order}
x_1(t)<x_2(t)<\cdots<x_N(t),\quad t\geq 0.
\end{equation}
The order preservation allows us to apply similar arguments in \cite{H-K-P-Z} to yield the asymptotic non-oscillatory behavior of relative distance between particles governed by the Cucker-Smale system \eqref{B5}. In fact, based on the analysis in \cite{H-K-P-Z}, we have the following lemma.
\begin{lemma} \label{non-oscillatory}
Let $X$ be a solution to system \eqref{B5} with initial configuration $X^0$ and natural velocities $\nu_i$ satisfying the following properties
\[\sum_{i=1}^N x_i^0 = 0,\quad \sum_{i=1}^N \nu_i = 0,\quad x_i^0\neq x_j^0,\quad 1\leq i\neq j\leq N.\]
 Then for every $i,j = 1, \ldots,N$, we have
\[\limsup_{t \to +\infty} |x_i(t) - x_j(t)| = \liminf_{t \to \infty} |x_i(t) - x_j(t)|.\]
More precisely, the limit of the relative distance for any two particles exists. Therefore, we immediately obtain that either $ \lim\limits_{t \to \infty} |x_i(t) - x_j(t)| < \infty$ or $ \lim\limits_{t \to \infty} |x_i(t) - x_j(t)| = \infty$.
\end{lemma}
\begin{proof}
The proof is almost similar as in \cite{H-K-P-Z}. But since we consider singular case and thus the potential function $\Phi$ is very different from regular case, we will keep the details of the proof for convenience and put it in Appendix \ref{p-non-oscillatory}.
\end{proof}

Using the same arguments as in \cite{H-K-P-Z}, we provide the critical coupling strength $\kappa_c$ for the emergence of mono-cluster flocking of C-S particles:
\begin{equation}\label{f-kappa_c}
\kappa_c = \max_{1 \le l \le N-1} \left\{- \frac{\frac{1}{l} \sum_{i=1}^{l} \nu_i}{\frac{N-l}{N} \Phi^\infty}\right\}.
\end{equation}

\noindent In the next lemma, we will show that the sufficient and necessary condition of the emergence of mono-cluster flocking for system \eqref{B5} is $\kappa>\kappa_c$.
\begin{lemma}\label{critical}
Let $\kappa$ be the positive coupling strength and suppose the initial configuration $X^0$ and natural velocities $\nu_i$ satisfy the following properties
\[\sum_{i=1}^N x_i^0 = 0,\quad \sum_{i=1}^N \nu_i = 0,\quad x_i^0\neq x_j^0,\quad 1\leq i\neq j\leq N.\]
Then the condition $\kappa > \kappa_c$ is a sufficient and necessary condition for the emergence of mono-cluster flocking to the C-S system \eqref{B5}. More precisely,
\begin{enumerate}[(1)]
\item
Assume $\kappa > \kappa_c$ and let $X$ be a solution to system \eqref{B5} with initial data $X^0$. Then, the mono-cluster flocking emerges:
\[\exists \  x_{ij}^\infty := \lim_{t \to \infty} |x_i(t) - x_j(t)|, \quad 1 \le i,j \le N.\]
\item
Assume that $X^\infty = (x_1^\infty, \ldots, x_N^\infty)$ is an emergent mono-cluster flocking state to system $\eqref{B5}$ with some initial configuration $X^0$. Then, $\kappa > \kappa_c$ holds.
\end{enumerate}
\end{lemma}
\begin{proof}
The proof is almost similar as in \cite{H-K-P-Z}. But since we consider singular case and thus the potential function $\Phi$ is very different from the regular case, we will keep the details of the proof for convenience and put it in Appendix \ref{p-critical}.\newline
\end{proof}

\subsection{Sufficient and necessary condition for multi-cluster formation}
In this part, we will provide a sufficient and necessary condition for multi-cluster formation emergence. Actually, we will describe how to capture the asymptotic clusters only based on natural velocities and coupling strength. Similar as in \cite{H-K-P-Z}, we will provide a clustering algorithm to divide the entire ensemble into several sub-ensembles for the given natural velocities $\nu_i$ and a fixed coupling stength $\kappa$. Here, we denote $(a, b] = \{a+1 \le \cdots \le b\}$. For a given sub-ensemble $(a, b]$, we introduce the local average velocity, local average position, and local average velocity and position fluctuations as follows:
\begin{align}
\begin{aligned}\label{local vp}
&\bar{\nu}^{(a, b]} := \frac{1}{b-a} \sum_{j=a+1}^b \nu_i, \quad \bar{x}^{(a, b]} := \frac{1}{b-a} \sum_{j = a+1}^b x_j, \\
&\hat{\nu}_j^{(a,b]} := \nu_j - \bar{\nu}^{(a, b]}, \quad \hat{x}_j^{(a,b]} := x_j - \bar{x}^{(a, b]}, \text{\quad for } a < j \le b.
\end{aligned}
\end{align}
We first fix the natural velocity $\nu$ and suppress the $\nu$-dependence in clustering number $N_c(\nu, \kappa)$, namely, we set
\[N_c(\kappa) := N_c(\nu, \kappa).\]
Note that the natural velocities are not monotonic.
We will divide the entire ensemble into several sub-ensembles according to the following algorithm:
\begin{equation} \label{f-algorithm}
\begin{cases}
\displaystyle  n_0 = 0, \\
\displaystyle \mathcal{I}_l := \{n_{l-1} +1\} \\
\hspace{0.5cm} \cup \left\{m : \frac{1}{k - n_{l-1}}  \sum_{j = n_{l-1} + 1}^k \hat{\nu}_j^{(n_{l-1}, m]} + \kappa \frac{m-k}{N} \Phi^\infty > 0, \text{ for any } n_{l-1} < k \le m-1 \right\}, \\
\displaystyle  n_l = \max \mathcal{I}_l. 
\end{cases}
\end{equation}
Now for a fixed natural velocity $\nu$ and a given coupling stength $\kappa$, by the clustering algorithm \eqref{f-algorithm} we will have a partition of particles below
\begin{equation*}
\left\{\begin{aligned}
&\{1,\ldots,N\} =: \bigsqcup_{j=1}^{N_c(\kappa)} \mathcal{I}_j,\\
&n_0 := 0, \quad n_{N_c(\kappa)} := N, \quad \mathcal{I}_l = \{n_{l-1} + 1, \ldots, n_l\},\quad  l =1,\ldots, N_c(\kappa).
\end{aligned}
\right.
\end{equation*}
For each sub-ensemble $\mathcal{I}_l (1 \le l \le N_c(\kappa))$, according to the clustering algorithm \eqref{f-algorithm}, we have the corresponding conditions
\[\frac{1}{k - n_{l-1}}  \sum_{j = n_{l-1} + 1}^k \hat{\nu}_j^{(n_{l-1}, n_l]} + \kappa \frac{n_l -k}{N} \Phi^\infty > 0, \quad \text{for all } \ n_{l-1} < k \le n_l -1.\]

\begin{figure}[t]
\centering
\mbox{
\subfigure[$\kappa=1.6$]{\includegraphics[width=0.48\textwidth]{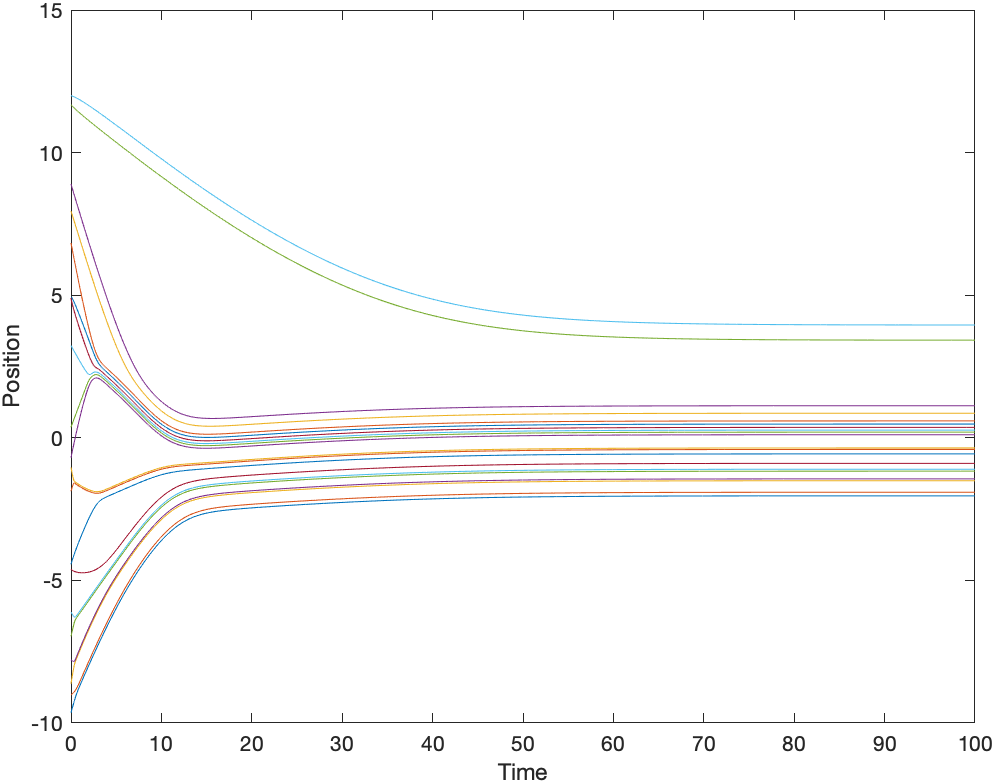}
\label{Fig2:16}}
\hspace{0.1cm}
\subfigure[$\kappa=0.4$]{\includegraphics[width=0.48\textwidth]{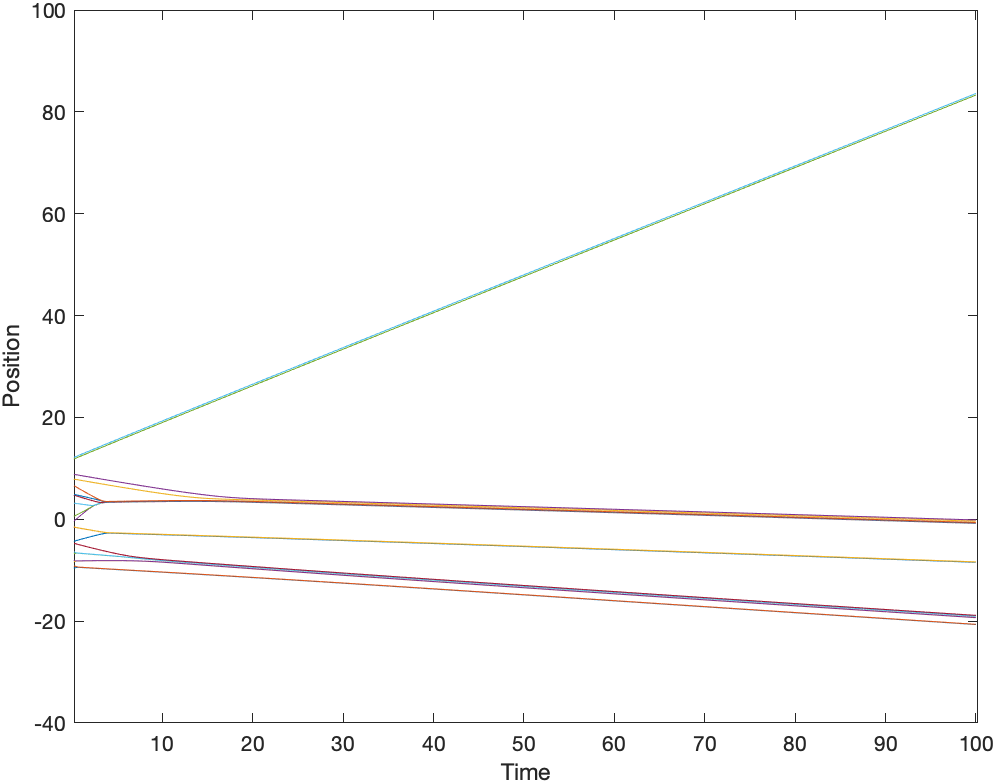}
\label{Fig2:04}}
}
\centering
\mbox{
\subfigure[$\kappa=0.01$]{\includegraphics[width=0.48\textwidth]{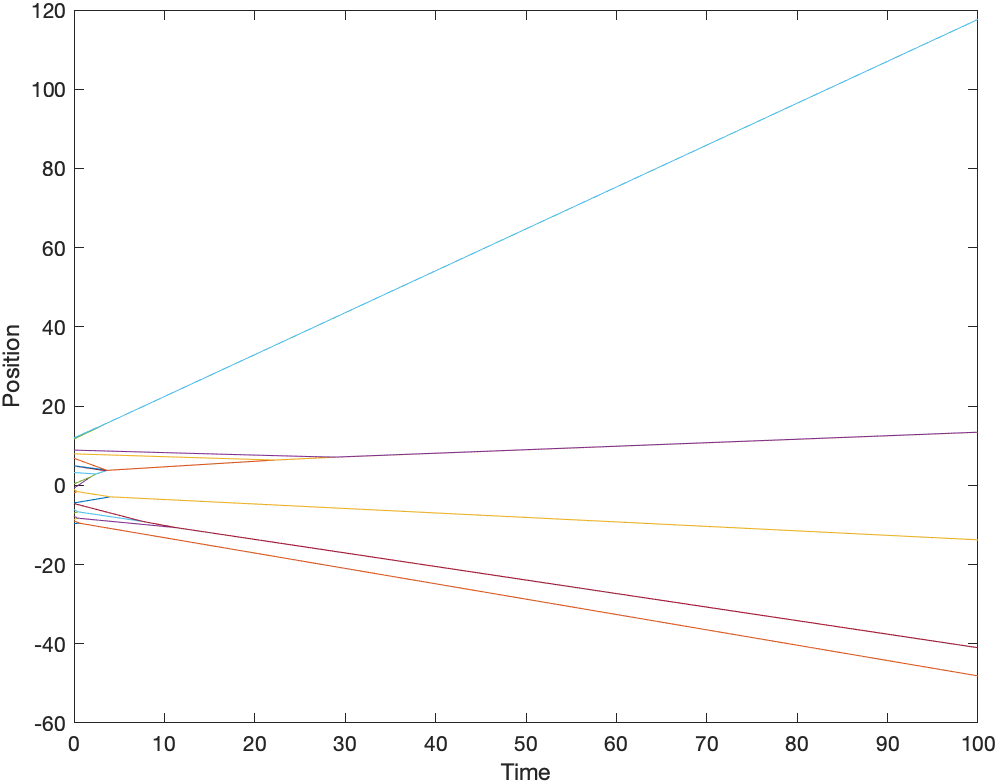}
\label{Fig2:001}}
\hspace{0.1cm}
\subfigure[Number of clusters]{\includegraphics[width=0.48\textwidth]{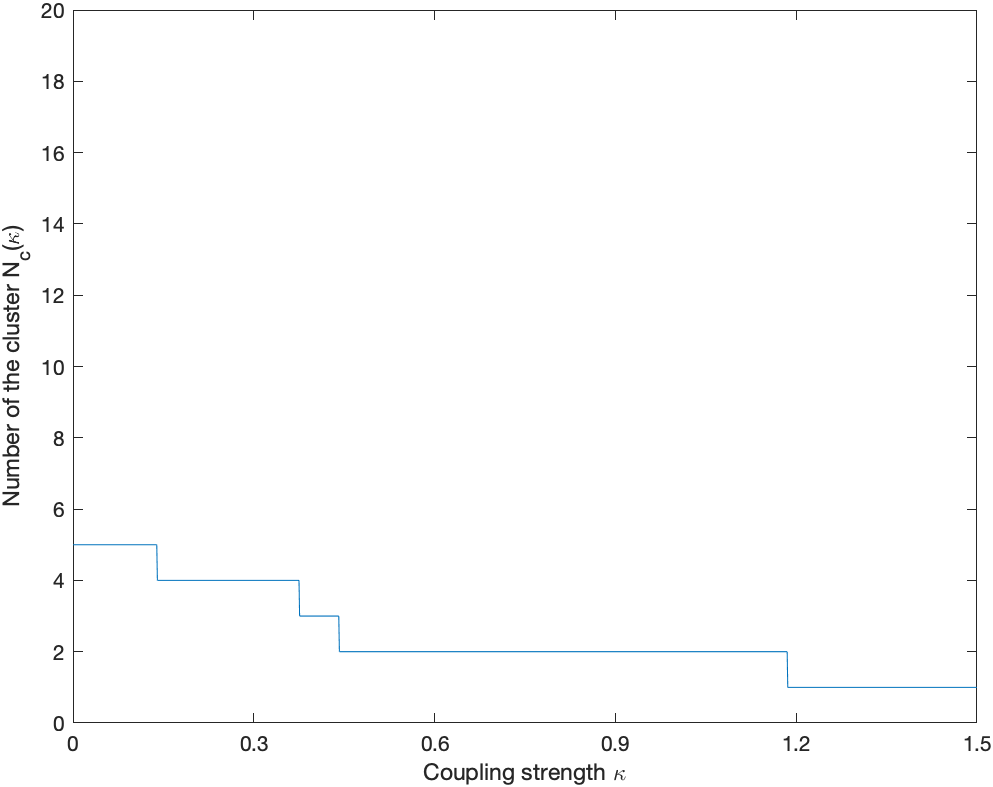}
\label{Fig2:num}}
}
\caption{Clustering dynamics of the first-order model}
\label{Fig2:cluster}
\end{figure}

\noindent Figure \ref{Fig2:16}, \ref{Fig2:04} and \ref{Fig2:001} shows the asymptotical behavior and clustering formation of the first order C-S model with $\kappa=1.6$, $0.4$ and $0.01$ respectively. We emphasize that there are no collisions in Figure  \ref{Fig2:04} and \ref{Fig2:001} although it looks like due to the scale of the figures. Figure \ref{Fig2:num} provides the number of the clusters for different coupling strength $\kappa$ based on the algorithm \eqref{f-algorithm}, which is compatible with the numerical results in Figure \ref{Fig2:16}, \ref{Fig2:04} and \ref{Fig2:001}. We now present the main result on the complete clustering predictability to system \eqref{B5} according to the constructed clustering algorithm \eqref{f-algorithm}. Then the sufficient and necessary condition for $p$-cluster formation will be provided for any positive integer $1\leq p\leq N$.
\begin{theorem}\label{T4.2}
Let $X$ be a solution to system \eqref{B5} with initial configuration $X^0$ and natural velocities $\nu_i$ satisfying the following properties
\[\sum_{i=1}^N x_i^0 = 0,\quad \sum_{i=1}^N \nu_i = 0,\quad x_i^0\neq x_j^0,\quad 1\leq i\neq j\leq N.\]
 Moreover, for a given coupling strength $\kappa$, let $\{\mathcal{I}_1, \ldots, \mathcal{I}_{N_c(\kappa)}\}$ be a partition of the whole ensemble constructed by the clustering algorithm \eqref{f-algorithm}. Then, we have the following assertions:
\begin{enumerate}[(i)]
\item
The sub-ensemble $\mathcal{I}_i = \{n_{i-1} +1, \ldots, n_i\}$ forms a     
maximal cluster-flocking asymptotically: for $i \ne j$,
\begin{align*}
\begin{aligned}
&\exists \  \lim_{t \to \infty} |x_k(t) - x_l(t)| < \infty \quad \forall \ k,l \in \mathcal{I}_i \quad \text{and} \\
&\sup_{0 \le t < \infty} \underset{l \in \mathcal{I}_j}{\min_{k \in \mathcal{I}_i}} |x_k(t) - x_l(t)| = \infty \quad \text{for } i \ne j.
\end{aligned} 
\end{align*}
\item
The asymptotic group velocity $v_i^\infty$ is given by the following explicit formula:
\[v_i^\infty := \bar{\nu}^{(n_{i-1}, n_i]} + \frac{\kappa (N - n_i - n_{i-1})}{N} \Phi^\infty, \quad 1 \le i \le N_c(\kappa).\]
\item
The $p$-cluster formation emergence occurs if and only if $p = N_c(\kappa)$ and the clusters are exactly $ \mathcal{I}_i$ where $i = 1, \ldots, p$.
\end{enumerate}
\end{theorem}
\begin{proof}
The proof is almost similar as in \cite{H-K-P-Z}. But since we consider singular case and thus the potential function $\Phi$ is very different from regular case, we will keep the details of the proof for convenience and put it in Appendix \ref{p-f-predictability}.\newline
\end{proof}

\vspace{0.05cm}
\subsection{Multi-cluster formation in second order C-S model}
In this subsection, we discuss the multi-cluster flocking for the second-order C-S model with communication rate $\psi(s) = \frac{1}{|s|^\beta}$ where $\beta > 1$. Due to the equivalent relation between the first-order system \eqref{B5} and second-order system \eqref{A1}, the results on multi-cluster flocking for second-order system are almost the same as what are obtained for first-order system. The reason why we emphasize the second-order model \eqref{A1} is that when we start from the model \eqref{A1} to construct natural velocities $\nu_i$ via the relation $\eqref{B5}_2$, the natural velocities $\nu_i$ depends on the coupling strength $\kappa$ as well. Therefore, for the second-order model \eqref{A1}, as we changes $\kappa$ to construct the critical value of coupling strength, the natural velocities will also change accordingly. Thus, the results for second-order model \eqref{A1} are slightly different from the first-order model \eqref{B5}. Now, we present the sufficient and necessary condition of the emergence of multi-cluster formation for second-order model \eqref{A1} with $\beta > 1$.

\begin{theorem}\label{T4.3}
Let $X$ be a solution to system \eqref{A1} with $\beta >1$ and
 initial configuration $(X^0,V^0)$ satisfying the following properties
\[\sum_{i=1}^N x_i^0 = 0,\quad \sum_{i=1}^N v_i = 0,\quad x_i^0\neq x_j^0,\quad 1\leq i\neq j\leq N.\]
Then the following assertions hold:
\begin{enumerate}[(i)]
\item 
Asymptotical $p$-cluster formation emerges for the second-order model \eqref{A1} with $\beta >1$ if and only if asymptotical $p$-cluster formation emerges for the model \eqref{B5}.
\item
Moreover, the asymptotic $p$-cluster formation $\bigsqcup_{i=1}^{p} \mathcal{I}_i$ with $|\mathcal{I}_i| = n_i - n_{i-1}$ can also be constructed by the same algorithm \eqref{f-algorithm} and the group velocity $v_i^\infty$ for $\mathcal{I}_i$ can be expressed by the following formula:
\begin{equation*}
\begin{aligned}
v_i^\infty = \frac{1}{n_i - n_{i-1}} \sum_{k = n_{i-1} + 1}^{n_i} v_i^0  - \frac{\kappa}{(n_i - n_{i-1})N} \sum_{k = n_{i-1} + 1}^{n_i} \underset{j \ne i}{\sum_{j=1}^N} \Phi(x_j^0 - x_i^0) +\frac{\kappa(N - n_i - n_{i-1})}{N} \Phi^\infty.
\end{aligned}
\end{equation*}
\end{enumerate}
\end{theorem}
\begin{proof}
The proof of Theorem \ref{T4.3} can be directly obtained by the equivalence relation between the model \eqref{A1} and \eqref{B5}.  And by substituting $\nu_i$ in \eqref{B5} into the asymptotic group velocity formula in Theorem \ref{T4.2} we get the asymptotic group velocity $v_i^\infty$ expressed by above formula.\newline
\end{proof}

Next, we will construct the critical coupling strength for mono-cluster emergence in the second order system \eqref{A1} with $\beta>1$. Actually the natural velocities $\nu_i$ given by \eqref{B5} can be rewritten as 
\[\frac{\nu_i(X^0, V^0, \kappa)}{\kappa} = \frac{v_i^0}{\kappa} - \frac{1}{N} \underset{k \ne i}{\sum_{k=1}^N} \Phi(x_k^0 - x_i^0).\]
It's obvious that when $\kappa$ is  large enough, the part of the relative initial position is dominant, whereas the initial velocity is dominant part when $\kappa$ is small enough. This leads to the different value of natural velocities for different $\kappa$, and thus the cluster formation from algorithm \eqref{f-algorithm} changes along with $\kappa$. Fortunately, according to \eqref{B5}, the natural velocity $\nu_i$ linearly depends on the coupling strength $\kappa$. Therefore, the critical coupling strength $\kappa_c$ for emergence of mono-cluster in the second order system \eqref{A1} can be determined in a similar way as before.

\begin{corollary}\label{C4.1}
Let $X$ be a solution to system \eqref{A1} with $\beta >1$ and
 initial configuration $(X^0,V^0)$ satisfying the following properties
\[\sum_{i=1}^N x_i^0 = 0,\quad \sum_{i=1}^N v_i = 0,\quad x_i^0\neq x_j^0,\quad 1\leq i\neq j\leq N.\]
Then we can find a positive constant $\kappa_c$ such that
\begin{equation*}
\begin{cases}
\displaystyle \text{mono-cluster formation}, \quad \kappa \in (\kappa_c, +\infty), \\
\displaystyle \text{multi-cluster formation}, \quad \kappa \in (0, \kappa_c].
\end{cases}
\end{equation*}
\end{corollary}
\begin{proof}
We follow the first order reduction in Section \ref{sec:2} to construct the first order model and corresponding natural velocities $\nu_i$. Then the natural velocities $\{ \nu_i \}$ and coupling strength $\kappa$ satisfy the following properties for large $\kappa$,
\begin{equation}\label{E-14}
\sum_{i =1}^N \nu_i = 0, \quad \frac{- \sum_{i = 1}^l \nu_i}{\kappa} < \frac{(N-l)l}{N} \Phi^\infty, \quad 1 \le l \le N-1.
\end{equation}
Therefore, according to the algorithm \eqref{f-algorithm}, mono-cluster flocking emerges for large enough $\kappa$. Then we define the critical value of coupling strength as follows
\[\kappa_c := \inf \{\kappa \ :\  \eqref{E-14} \ \text{holds}\}. \]
If $\kappa\leq \kappa_c$, then we directly apply the definition of $\kappa_c$ and Theorem \ref{T4.2} to obtain that mono-cluster formation will not occur. On the other hand, for any $\kappa > \kappa_c$, we will show \eqref{E-14} always holds and then we can combine Theorem \ref{T4.2} to conclude the emergence of mono-cluster formation. In fact, from the definitioin of $\kappa_c$, for any positive constant $\varepsilon$, we can find a $\kappa_c(\varepsilon)$ satisfying \eqref{E-14} such that
\[\kappa_c \le \kappa_c(\varepsilon) < \kappa_c + \varepsilon.\] 
Then at $\kappa_c(\varepsilon)$, \eqref{E-14} holds with respect to the natural velocities at $\kappa_c(\varepsilon)$. That is, for all $1 \le l \le N-1$,
\[ \frac{-\sum_{i =1}^l (v_i^0 - \frac{\kappa_c(\varepsilon)}{N} \sum_{k \ne i}\Phi(x_k^0 - x_i^0))}{\kappa_c(\varepsilon)} = \frac{- \sum_{i = 1}^l \nu_i}{\kappa_c(\varepsilon)} < \frac{(N-l)l}{N} \Phi^\infty.\]
We rewrite the left-hand side of above formula as below,
\begin{equation*}
\begin{aligned}
\frac{- \sum_{i = 1}^l \nu_i}{\kappa_c(\varepsilon)} = \frac{- \sum_{i=1}^l v_i^0}{\kappa_c(\varepsilon)} + \frac{1}{N} \sum_{i=1}^l \sum_{k \ne i} \Phi(x_k^0 - x_i^0) = \frac{- \sum_{i=1}^l v_i^0}{\kappa_c(\varepsilon)} + \frac{1}{N} \sum_{i=1}^l \sum_{k = l+1}^N \Phi(x_k^0 - x_i^0).
\end{aligned}
\end{equation*}
Then, we consider two cases, respectively.\newline

\noindent $\diamond$ (Case 1):~If $- \sum_{i=1}^l v_i^0 \le 0$, then for any $\kappa \ge \kappa_c(\varepsilon)$, we obviously have that the natural velocities at $\kappa$ satisfy
\[\frac{-\sum_{i =1}^l (v_i^0 - \frac{\kappa}{N} \sum_{k \ne i}\Phi(x_k^0 - x_i^0))}{\kappa} \le \frac{1}{N} \sum_{i=1}^l \sum_{k = l+1}^N \Phi(x_k^0 - x_i^0) < \frac{(N-l)l}{N} \Phi^\infty.\]

\noindent $\diamond$ (Case 2):~If $- \sum_{i=1}^l v_i^0 > 0$, then for any $\kappa \ge \kappa_c(\varepsilon)$, we exploit the monotone decreasing property of $\frac{- \sum_{i=1}^l v_i^0}{\kappa}$ with respect to $\kappa$ to obtain that
\[ \frac{-\sum_{i =1}^l (v_i^0 - \frac{\kappa}{N} \sum_{k \ne i}\Phi(x_k^0 - x_i^0))}{\kappa} \le \frac{- \sum_{i=1}^l v_i^0}{\kappa_c(\varepsilon)} + \frac{1}{N} \sum_{i=1}^l \sum_{k = l+1}^N \Phi(x_k^0 - x_i^0) < \frac{(N-l)l}{N} \Phi^\infty.\]
Combining results of two cases, we obtain that for any $\kappa \ge \kappa_c(\varepsilon)$, \eqref{E-14} holds with respect to the natural velocities at $\kappa$. Then by the arbitrary choice of $\varepsilon$, we obtain that \eqref{E-14} holds for any $\kappa > \kappa_c$, thus mono-cluster flocking occurs.

\end{proof}

\begin{figure}[b]
\centering
\includegraphics[width=0.5\textwidth]{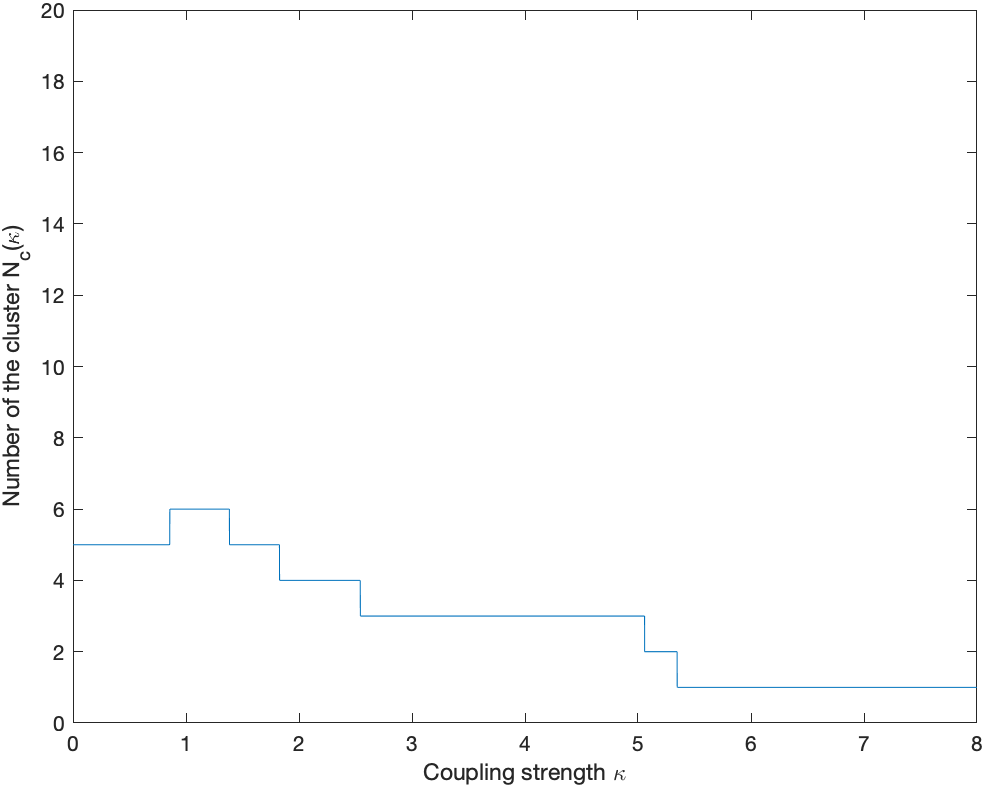}
\caption{Number of clusters for second order model}
\label{Fig3}
\end{figure}
Figure \ref{Fig3} shows the number of clusters of second order C-S model for different $\kappa$. Different from the first order model, the number of clusters may increase and decrease when $\kappa$ increases from zero to positive infinity.

\vspace{0.3cm}
\section{Difference between regular and singular C-S model}\label{sec:5}
\setcounter{equation}{0}
\vspace{0.3cm}
In this section, we will discuss the differences between C-S model with regular communication and singular communication. In fact, from the analysis in Section \ref{sec:3} and Section \ref{sec:4}, we already know some differences between the regular case and singular case. When $\beta<1$, we have analytic solution for any initial data and flocking occurs only when $t$ tends to infinity if the communication function is regular, while we only have the weak solution in the sense of Definition \ref{D2.1} and finite time flocking (sticking) may occur if the communication function is singular. On the other hand, when $\beta>1$, collision may occur in regular case and we show the collision avoidance in singular case. In the following, we will continue to show two differences between regular and singular C-S model. 

\subsection{Critical order of singularity: $\beta=1$}In this section, we discuss the dynamics of system \eqref{A1} in the case that $\beta =1$. In this case,  \cite{C-C-M-P} provided a proof of collision avoidance. Thus the solution to \eqref{A1} is analytic and we may apply the result of Proposition 4.1 in \cite{H-L} to obtain that there exists a positive constant $C_M$ such that the diameter of positions satisfies the following estimate:
\[D_x(t) = \max_{1 \le i, j \le N} |x_i(t) - x_j(t)| \le C_M, \quad t \ge 0.\]
This upper bound of the diameter of spatial variable immediately leads to the unconditional flocking. However, different from the regular case, there will be no collision for all time if the communication function is singular. Therefore, we apply the same arguments as in Lemma \ref{LD1} and \ref{LD2}, we can also show the uniform lower bound of the distance between particles in the critical case $\beta = 1$. More precisely, we have the following theorem for singular C-S model with $\beta=1$.

\begin{theorem}\label{T5.1}
Let $X$ be a solution to the second order equations \eqref{A1} with singular communication function and $\beta=1$, and suppose that the initial configuration $(X^0,V^0)$ satisfy the following properties
\[\sum_{i=1}^N x_i^0 = 0,\quad \sum_{i=1}^N v_i = 0,\quad x_i^0\neq x_j^0,\quad 1\leq i\neq j\leq N.\]
Then the following assertions hold:
\begin{enumerate}
\item 
There exist positive constants $C'_{m_1}$ and $C'_{M_1}$ such that
\begin{equation*}
 C'_{m_1}\leq  |x_i(t)-x_j(t)|\leq C'_{M_1},\quad 1\leq i\neq j\leq N,\quad t\geq 0. 
\end{equation*}
\item
Unconditional flocking occurs and there exists an equilibrium state $X^\infty=(x_1^\infty,\cdots,x_N^\infty)$ such that 
\[|X(t)-X^{\infty}|\leq C e^{-\kappa \psi(C'_{M_1})t},\quad |V(t)|\leq C' e^{-\kappa \psi(C'_{M_1})t}.\]
\end{enumerate}
\end{theorem}
\vspace{0.2cm}
\subsection{Cluster formation for small coupling strength when $\beta>1$}
We then consider the situation when $\beta>1$ and the coupling strength $\kappa$ is sufficiently small. It is known in \cite{H-K-P-Z, H-P-Z} that, for regular case, $N$ cluster formation will occur when $\kappa$ is sufficiently small provided all natural velocities are different from each other. However, this is not true if the communication function is singular. Actually, we will provide an new algorithm to describe the cluster formation in this case.\newline 

Due to the equivalent relation between the system \eqref{B5} and the second-order system \eqref{A1}, for a given $\kappa$, we substitute $\nu_i$ constructed by the formula $\eqref{B5}_2$ into the clustering algorithm \eqref{f-algorithm} to obtain the parallel clustering algorithm for the second-order system:
\begin{equation}\label{second-algorithm}
\begin{cases}
\displaystyle n_0 = 0, \\
\displaystyle \mathcal{I}_l := \{n_{l-1} + 1\} \\ 
\hspace{0.5cm} \cup \left\{ m : \ \frac{1}{k - n_{l-1}} \sum_{j = n_{l-1} + 1}^k \widehat{v_j^0}^{(n_{l-1}, m]}  + \kappa A > 0, \quad \text{for all} \ n_{l-1} < k \le m-1\right\}, \\
 n_l = \max \mathcal{I}_l,
\end{cases}
\end{equation}
where we use the notations as in \eqref{local vp} for the inital velocity $v_i^0$ and for notational simplicity we set $A$ as below,
\[ A = - \frac{1}{(k - n_{l-1})N} \sum_{j= n_{l-1} +1}^k \underset{k \ne j}{\sum_{k=1}^N} \Phi(x_k^0 - x_j^0) + \frac{1}{N(m - n_{l-1})} \sum_{i = n_{l-1}+1}^m \underset{k \ne i}{\sum_{k=1}^N} \Phi(x_k^0 - x_i^0) + \frac{m-k}{N} \Phi^\infty.\]
On the other hand, we introduce another algorithm which depends on the initial velocities $\{v_i^0\}$ but is independent of the coupling strength $\kappa$:
\begin{equation}\label{small-algorithm}
\begin{cases}
\displaystyle n_0 = 0, \\
\displaystyle \mathcal{I}_l := \{n_{l-1} + 1\} \cup \left\{ m : \ \frac{1}{k - n_{l-1}} \sum_{j = n_{l-1} + 1}^k \widehat{v_j^0}^{(n_{l-1}, m]} > 0, \quad \text{for all} \ n_{l-1} < k \le m-1 \right\}, \\
\displaystyle n_l = \max \mathcal{I}_l.
\end{cases}
\end{equation}
We will show in the next lemma that the algorithm \eqref{second-algorithm} and \eqref{small-algorithm} are equivalent when $\kappa$ is sufficiently small. Therefore, the cluster formation will follow the algorithm  \eqref{small-algorithm} and only depends on initial velocities for sufficiently small $\kappa$. 

\begin{lemma}\label{no change}
Let $X$ be a solution to the second order equations \eqref{A1} with singular communication function and $\beta>1$, and suppose that the initial configuration $(X^0,V^0)$ satisfy the following properties
\begin{equation}\label{E3}
\sum_{i=1}^N x_i^0 = 0,\quad \sum_{i=1}^N v_i = 0,\quad x_i^0\neq x_j^0,\quad 1\leq i\neq j\leq N.
\end{equation}
For any given $\kappa$, assume that asymptotic $p_\kappa$-cluster formation $\bigsqcup_{l=1}^{p_\kappa} \mathcal{J}_l$ with $|\mathcal{J}_l| = \bar{n}_l - \bar{n}_{l-1}$ emerges from the algorithm \eqref{second-algorithm}. Moreover, the cluster formation constructed from the algorithm \eqref{small-algorithm}  is denoted by $\bigsqcup_{l=1}^q \mathcal{I}_l$ with $|\mathcal{I}_l| = n_l - n_{l-1}$. Then, for almost all initial data satisfying \eqref{E3}, we conclude for sufficiently small $\kappa$ that
\begin{equation}\label{E-15}
p_\kappa = q, \quad \text{and} \quad \mathcal{J}_l = \mathcal{I}_l, \ l = 1,\ldots,q.
\end{equation}
\end{lemma}
\begin{proof}
We prove that \eqref{E-15} holds by contradiction. Suppose \eqref{E-15} does not hold for $\kappa$ sufficiently small. Then there exists $l \ge 1$ such that
\begin{equation}\label{E-16}
\mathcal{I}_l \ne \mathcal{J}_l, \quad \text{and} \quad  \mathcal{I}_j = \mathcal{J}_j, \ 1 \le j \le l-1.
\end{equation}
Due to the relation \eqref{E-16}, we have $n_{l-1} = \bar{n}_{l-1}$.
Thus, the group $\mathcal{J}_l$ can be rewritten as $\mathcal{J}_l = (n_{l-1}, \bar{n}_l]$ where $\bar{n}_l \ne n_l $. We claim that $n_l  <  \bar{n}_l$. Indeed, from algorithm \eqref{small-algorithm}, for the group $\mathcal{I}_l$ we have the following relations
\[\frac{1}{k - n_{l-1}} \sum_{j = n_{l-1} + 1}^k\widehat{v_j^0}^{(n_{l-1}, n_l]} > 0, \quad \text{for all} \  n_{l-1} < k \le n_l -1.\]
When $\kappa$ is sufficiently small, combing the above inequality and $\kappa > 0$, we have
\begin{equation}\label{E-17}
\frac{1}{k - n_{l-1}} \sum_{j = n_{l-1} + 1}^k\widehat{v_j^0}^{(n_{l-1}, n_l]} + \kappa A_1 \ge 0, \quad \text{for all} \ n_{l-1} < \kappa \le n_l -1,
\end{equation}
where
\[ A_1 = - \frac{1}{(k - n_{l-1})N} \sum_{j= n_{l-1} +1}^k \underset{k \ne j}{\sum_{k=1}^N} \Phi(x_k^0 - x_j^0) + \frac{1}{N(n_l - n_{l-1})} \sum_{i = n_{l-1}+1}^{n_l} \underset{k \ne i}{\sum_{k=1}^N} \Phi(x_k^0 - x_i^0) + \frac{n_l-k}{N} \Phi^\infty.\]
Then it is clear from the algorithm $\eqref{second-algorithm}_3$ that for the group $\mathcal{J}_l $ we have  $n_l \le \bar{n}_l$. Since $n_l \ne \bar{n}_l$, we obtain $n_l < \bar{n}_l$. However, from the definition of $n_l$ for group $\mathcal{I}_l$ in \eqref{small-algorithm} and the fact that $n_l < \bar{n}_l$, it is known that there exists $n_{l-1} < k_0 \le \bar{n}_l -1$ such that
\[\frac{1}{k_0 - n_{l-1}} \sum_{j = n_{l-1} + 1}^{k_0}\widehat{v_j^0}^{(n_{l-1}, \bar{n}_l]} \le 0.\]
Suppose the equality holds in above formula. Then the equality generates a lower dimensional hyperplane, which is of zero measure in $N$ dimensional phase space. Thus for almost all initial data, we have the following inequality
\[\frac{1}{k_0 - n_{l-1}} \sum_{j = n_{l-1} + 1}^{k_0}\widehat{v_j^0}^{(n_{l-1}, \bar{n}_l]} < 0.\]
Therefore, for $\kappa > 0$ enough small, we get
\begin{equation}\label{E-18}
\frac{1}{k_0 - n_{l-1}} \sum_{j = n_{l-1} + 1}^{k_0}\widehat{v_j^0}^{(n_{l-1}, \bar{n}_l]} + \kappa A_2(k_0) < 0,
\end{equation}
where $A_2(k)$ depending on the integer $k$ is expressed by the formula as follows
\[ A_2(k) = - \frac{1}{(k - n_{l-1})N} \sum_{j= n_{l-1} +1}^k \underset{k \ne j}{\sum_{k=1}^N} \Phi(x_k^0 - x_j^0) + \frac{1}{N(\bar{n}_l - n_{l-1})} \sum_{i = n_{l-1}+1}^{\bar{n}_l} \underset{k \ne i}{\sum_{k=1}^N} \Phi(x_k^0 - x_i^0) + \frac{\bar{n}_l-k}{N} \Phi^\infty.\]
However, for the group $\mathcal{J}_l$, it can be seen from the algorithm \eqref{second-algorithm} that the coupling strength $\kappa$ satisfies the following rules
\begin{equation}\label{E-19}
\frac{1}{k - n_{l-1}} \sum_{j = n_{l-1} + 1}^k\widehat{v_j^0}^{(n_{l-1}, \bar{n}_l]} + \kappa A_2(k) \ge 0, \quad \text{for all} \ n_{l-1} < k \le \bar{n}_l -1.
\end{equation}
Since the index $k_0$ satisfies $n_l < k_0 \le \bar{n}_{l} - 1$, \eqref{E-18} obviously contradicts to \eqref{E-19} for $\kappa$ sufficiently small. Thus, we derive the desired result.
\end{proof}
\begin{figure}[p]
\centering
\mbox{
\subfigure[Regular with $\kappa=0.01$]{\includegraphics[width=0.48\textwidth]{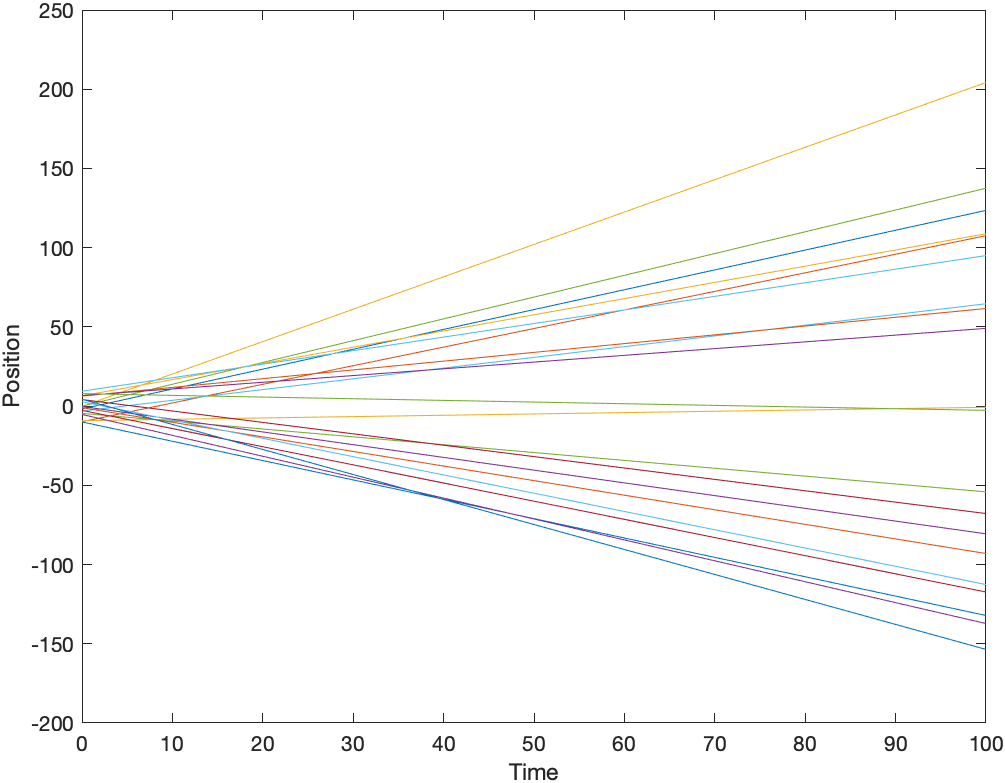}
\label{Fig2:p5}}
\hspace{0.1cm}
\subfigure[Singular with $\kappa=0.01$]{\includegraphics[width=0.48\textwidth]{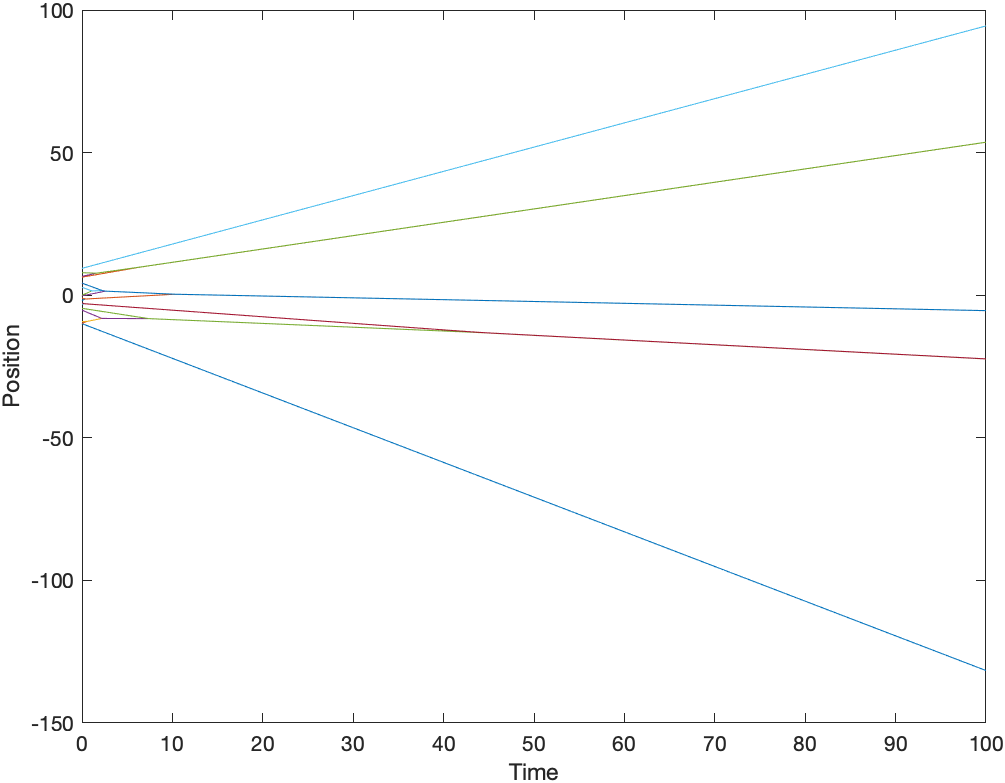}
\label{Fig2:v5}}
}
\centering
\mbox{
\subfigure[Regular with $\kappa=1.2$]{\includegraphics[width=0.48\textwidth]{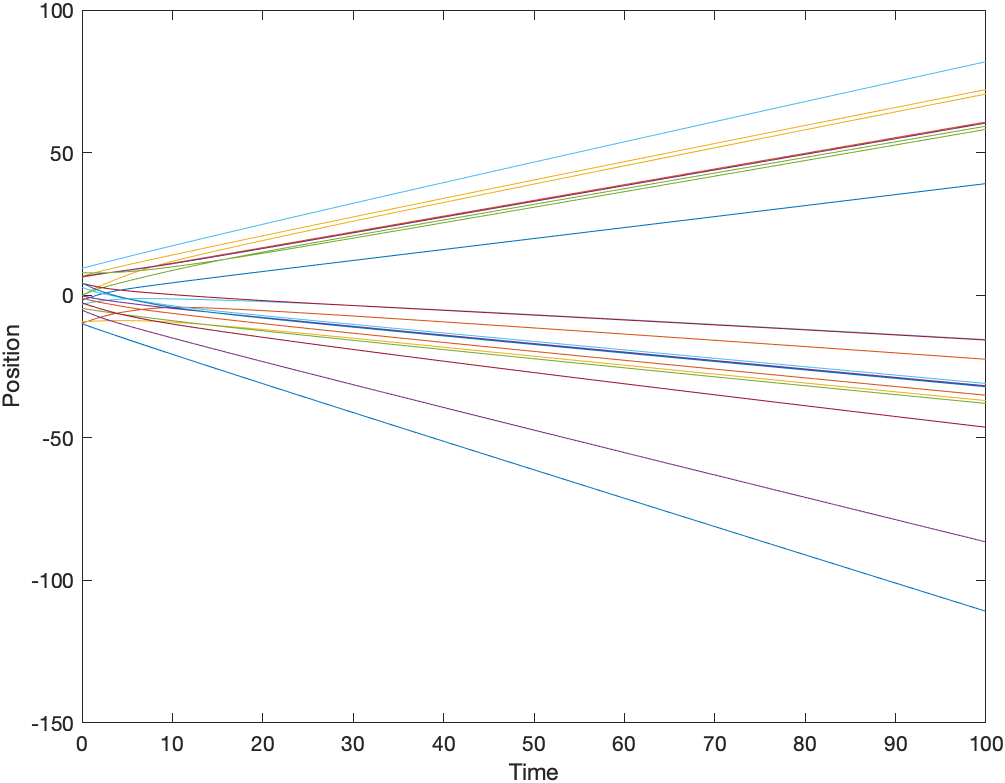}
\label{Fig2:p10}}
\hspace{0.1cm}
\subfigure[Singular with $\kappa=1.2$]{\includegraphics[width=0.48\textwidth]{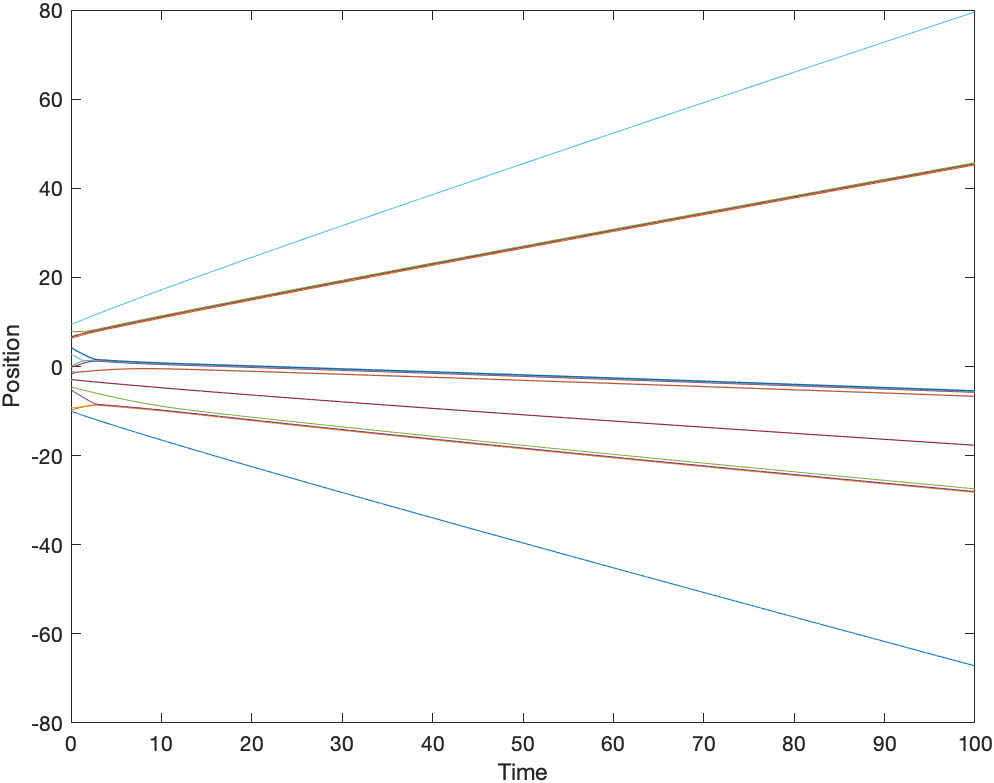}
\label{Fig2:v10}}
}
\centering
\mbox{
\subfigure[Number of cluster for regular]{\includegraphics[width=0.48\textwidth]{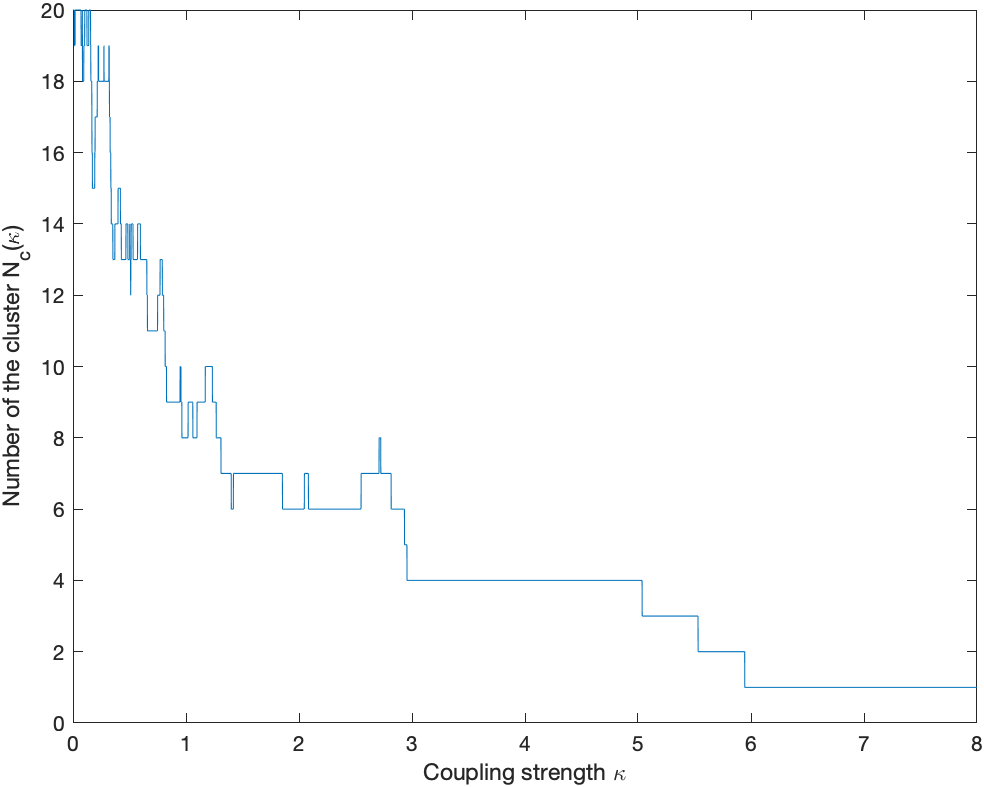}
\label{Fig2:p40}}
\hspace{0.1cm}
\subfigure[Number of cluster for singular]{\includegraphics[width=0.48\textwidth]{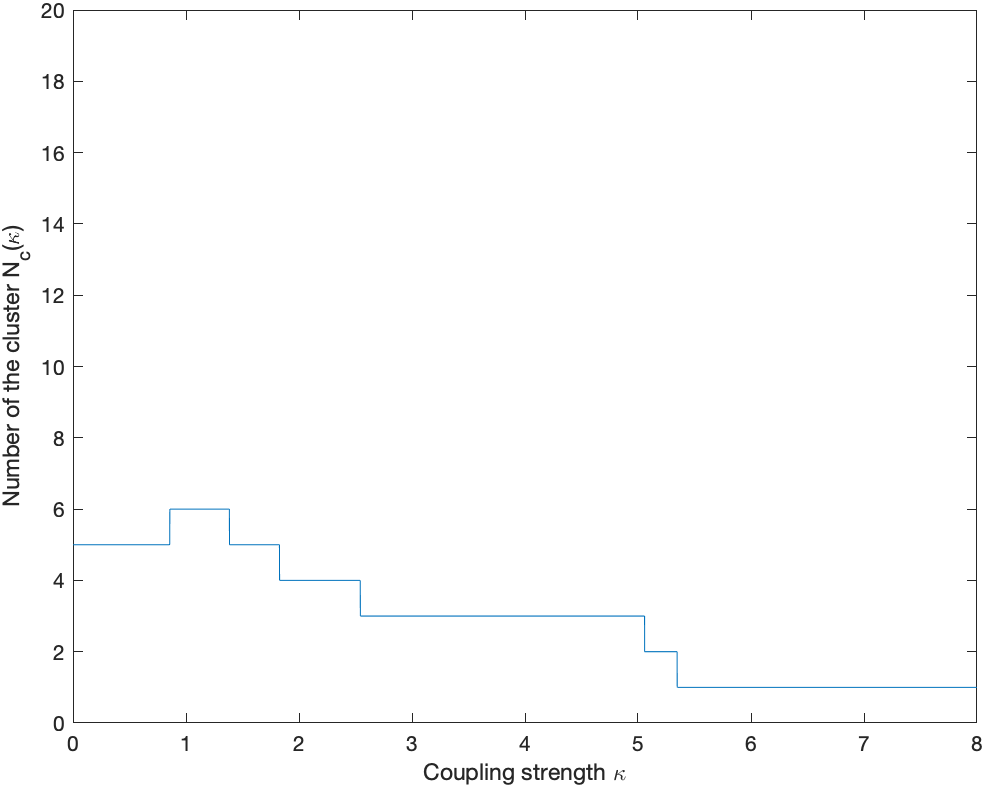}
\label{Fig2:v40}}
}
\caption{Comparison between regular and singular models}
\label{Fig:pv}
\end{figure}

\begin{remark}\label{R-no change}
Lemma \ref{no change} shows us the differences between the regular case and singular case.  In fact,  for second-order system \eqref{A1}, we know from Lemma \ref{no change} that, when $\kappa$ is sufficiently small, the asymptotic cluster formation is not going to change and can be determined by the algorithm \eqref{small-algorithm}. Note that how small the coupling strength $\kappa$ should depends on the initial position and velocity. Then if we let $v_N^0<\cdots<v_1^0$, it is obviously from the algorithm \eqref{small-algorithm} that mono-cluster flocking will occur if $\kappa>0$. This phenomenon is quite different from the regular case, in which $N$ cluster formation will always emerge for small $\kappa$ for almost all initial data. These differences are verified in Figure \ref{Fig:pv}, which shows the simulation of both regular and singular models with twenty oscillators.

\end{remark}

\vspace{0.3cm}

\section{Summary}\label{sec:6}
\setcounter{equation}{0}
\vspace{0.3cm}
In the present paper, we study the C-S model with singular communication weight on the real line. For $\beta<1$, we show the uniqueness of the weak solution in the sense of Definition \ref{D2.1} and provide the exponentially fast unconditional flocking. Moreover, we give an explicit analysis on the collision behavior depending on the natural velocities and conclude the sticking phenomenon occurs if and only if two natural velocities are identical. For $\beta>1$, we show the uniform-in-time lower bound of the relative distance between particles and provide the sufficient and necessary condition for the emergence of multi-cluster formation. For $\beta=1$, we show that unconditional flocking will occur and construct the uniform-in-time lower bound of the relative distance between particles. Therefore, we provide a complete classification of the asymptotical behavior of the C-S model with singular communication weight on the real line. So far, we do not know how to deal with the higher dimensional model, and this will be our future work.

\vspace{0.3cm}
\begin{appendix}
\section{Proof of Lemma \ref{non-oscillatory}}\label{p-non-oscillatory}
\setcounter{equation}{0}
\vspace{0.3cm}
Below, we will show that the configuration $\{x_i\}$ will form either a mono-cluster or multi-clusters asymptotically and will not oscillate when time tends to infinity. More precisely, we will verify that for every $i,j$, we have
\begin{equation*}
\begin{cases}
\displaystyle  \lim_{t \to +\infty} |x_i(t) - x_j(t)| < + \infty, \quad \text{if} \  x_i \ \text{and}\  x_j  \ \text{are in the same cluster} ; \\
\displaystyle \lim_{t \to +\infty} |x_i(t) - x_j(t)| = +\infty, \quad \text{if} \ x_i \ \text{and} \ x_j \ \text{are in the different cluster}.
\end{cases}
\end{equation*}
First, as we already prove the non-collision property for $\beta>1$. Therefore the the system \eqref{B5} admits an analytic solution and thus we can apply the arguments of \cite{H-Ts} to show that the velocity will not oscillate, i.e. 
\[\exists \lim_{t \to +\infty} v_i(t) = v_i^\infty \quad \text{for any} \ i =1,\ldots, N.\]
 Then we can apply the same idea from \cite{H-K-P-Z, H-P-Z} to show the uniqueness of equilibrium for the sub-system of \eqref{B5}.
 \begin{lemma}\label{equilibrium}
 For any given $1 \le m \le N$ and $\{\nu_i\}_{i = m}^{N}$, let $Y^\infty = (y_m^\infty, \ldots, y_N^\infty) \in \mathbb{R}^{N-m+1}$ be a solution to the following equilibrium system:
 \begin{equation}\label{F-1}
 \nu_i + \frac{\kappa}{N} \underset{k \ne i}{\sum_{k =m}^N} \Phi(y_k^\infty - y_i^\infty) = 0, \quad i = m, \ldots, N,
 \end{equation} 
 where $\sum_ {k =m}^N y_k^\infty = 0$ and $\Phi$ is defined as in \eqref{B3}. Then $Y^\infty$ is unique up to permutation.
 \end{lemma}
\begin{proof}
We prove by contradiction. Suppose $Y^\infty = (y_m^\infty, \ldots, y_N^\infty) \in \mathbb{R}^{N-m+1}$ and $\bar{Y}^\infty = (\bar{y}_m^\infty, \ldots, \bar{y}_N^\infty) \in \mathbb{R}^{N-m+1}$ are two different solutions of \eqref{F-1}. Then, we have
\begin{equation}\label{F-2}
\begin{aligned}
&\nu_i + \frac{\kappa}{N} \underset{k \ne i}{\sum_{k =m}^N} \Phi(y_k^\infty - y_i^\infty) = 0, \quad i = m, \ldots, N, \quad \\
&\nu_i + \frac{\kappa}{N} \underset{k \ne i}{\sum_{k =m}^N} \Phi(\bar{y}_k^\infty - \bar{y}_i^\infty) = 0, \quad i = m, \ldots, N.
\end{aligned}
\end{equation}
In view of permutation, without loss of generality, we may assume that the partial configurations $Y^\infty$ and $\bar{Y}^\infty$ are sorted as follows:
\begin{equation}\label{F-3}
y_m^\infty < y_{m+1}^\infty < \cdots < y_N^\infty, \qquad \bar{y}_m^\infty < \bar{y}_{m+1}^\infty < \cdots < \bar{y}_N^\infty.
\end{equation}
We introduce three index sets:
\[\mathcal{A} := \{i \ :\ \bar{y}_i^\infty > y_i^\infty\}, \qquad \mathcal{B} := \{i \ : \  \bar{y}_i^\infty <  y_i^\infty\}, \qquad \mathcal{C} := \{i \ : \  \bar{y}_i^\infty =  y_i^\infty\}.\]
Since $Y^\infty$ and $\bar{Y}^\infty$ are different, at least one of the sets $\mathcal{A}$ and $\mathcal{B}$ is non-empty. Without loss of generality, suppose the set $\mathcal{A}$ is non-empty. Then we claim that $\mathcal{B}$ is an empty set.  If not the case, $\mathcal{B} \ne \emptyset$. It is known from $\eqref{F-2}_2$ that for all $i \in \mathcal{B}$,
\[\nu_i = - \frac{\kappa}{N} \underset{k \ne i}{\sum_{k =m}^N} \Phi(\bar{y}_k^\infty - \bar{y}_i^\infty).\]
We take sum on above equation for all $i \in \mathcal{B}$ to obtain that
\begin{equation}\label{F-4}
\sum_{i \in \mathcal{B}} \nu_i = - \frac{\kappa}{N} \sum_{i \in \mathcal{B}} \underset{k \ne i}{\sum_{k =m}^N} \Phi(\bar{y}_k^\infty - \bar{y}_i^\infty) =- \frac{\kappa}{N} \sum_{i \in \mathcal{B}} \sum_{k \not\in \mathcal{B}} \Phi(\bar{y}_k^\infty - \bar{y}_i^\infty),
\end{equation}
where we apply the cancellation that $\sum_{i \in \mathcal{B}} \underset{k \ne i} {\sum_{k \in \mathcal{B}} } \Phi(\bar{y}_k^\infty - \bar{y}_i^\infty) = 0$. Since $i \in \mathcal{B}$ and $k \not\in \mathcal{B}$, we have $\bar{y}_i^\infty < y_i^\infty$ and $ \bar{y}_k^\infty \ge y_k^\infty$. Then we combine these inequalities and \eqref{F-3} to obtain that
\[\text{either} \quad 0 < \bar{y}_i^\infty - \bar{y}_k^\infty < y_i^\infty - y_k^\infty \quad \text{or} \quad \bar{y}_i^\infty - \bar{y}_k^\infty < y_i^\infty - y_k^\infty < 0.\]
Then according to the monoticity of $\Phi$, for $i \in \mathcal{B}$ and $k \not\in \mathcal{B}$ we have $\Phi(\bar{y}_i^\infty - \bar{y}_k^\infty) < \Phi(y_i^\infty - y_k^\infty )$. Thus, it follows from \eqref{F-4} that
\[\sum_{i \in \mathcal{B}} \nu_i = - \frac{\kappa}{N} \sum_{i \in \mathcal{B}} \sum_{k \not\in \mathcal{B}} \Phi(\bar{y}_k^\infty - \bar{y}_i^\infty) < - \frac{\kappa}{N} \sum_{i \in \mathcal{B}} \sum_{k \not\in \mathcal{B}} \Phi(y_i^\infty - y_k^\infty ) = \sum_{i \in \mathcal{B}} \nu_i.\]
This is clearly contradictory and thus we conclude $\mathcal{B} = \emptyset$. We next prove $\mathcal{A}$ is empty. If not, we have
\[ 0 = \sum_{i = m}^N \bar{y}_i^\infty > \sum_{i = m}^N  y_i^\infty = 0,\]
which is again a contradiction. Thus, we have $\mathcal{A}$ and $\mathcal{B}$ are empty and conclude that 
\[\bar{y}_i^\infty = y_i^\infty, \quad i = m,\ldots, N.\]
\end{proof}

Next, we will prove either mono-cluster or multi-clusters will emerge asymptotically and oscillation will never happen. For this, we first show the dichotomy of the asymptotic behavior of $|x_i - x_N|$.
\begin{lemma}\label{|x_i - x_N|}
Let $X$ be a solution of system \eqref{B5}. Then for each $i =1, \ldots, N$, we have the following dichotomy:
\[\text{either} \quad  \lim_{t \to \infty} |x_i(t) - x_N(t)| < \infty \quad \text{or} \quad
\lim_{t \to \infty} |x_i(t) - x_N(t)| = \infty.\]
\end{lemma}
\begin{proof}
If not the case, there exists at least one $j\in \mathcal{N}$ satisfying
\begin{equation}\label{F-5}
\limsup_{t \to +\infty} |x_j(t) - x_N(t)| \ne \liminf_{t \to +\infty} |x_j(t) - x_N(t)|.
\end{equation}
We construct two sets:
\begin{equation}\label{F-6}
\begin{cases}
\displaystyle E_1 := \{i \ : \ \liminf_{t \to +\infty}  |x_N(t) - x_i(t)| = +\infty\}, \\ 
\displaystyle E_2 := \{i \ : \ \liminf_{t \to +\infty} |x_N(t) - x_i(t)| < + \infty\}.
\end{cases}
\end{equation}
Obviously, $E_1 \cup E_2 = \{1,2,\ldots, N\}$. For $i \in E_1$, we have $\lim\limits_{t \to +\infty} |x_N(t) - x_i(t)| = + \infty$.
Therefore, for $j$ satisfying the relation \eqref{F-5}, we obtain $j \in E_2$. Moreover, the limit of the velocity always exists according to \cite{H-Ts}. Therefore if $\lim\limits_{t \to +\infty} v_i(t) \ne \lim\limits_{t \to +\infty} v_N(t)$, we have $i \in E_1$. Thus, for $i \in E_2$, we have 
\begin{equation}\label{F-7}
\lim_{t \to +\infty} v_i(t) = \lim_{t \to +\infty} v_N(t).
\end{equation}
If $E_2 = \emptyset$, we derive the desired result. Otherwise, we set $m := \min E_2$. Then, according to the well order of spatial variable and the definition of the sets $E_1$ and $E_2$, we have 
\[E_1 = \{1, \ldots, m-1\}, \qquad E_2 = \{m, \ldots, N\}.\]
From the definition of $E_2$, there exists $C_m$ such that
\begin{equation}\label{F-8}
\liminf_{t \to +\infty} |x_N(t) - x_m(t)| = C_m < +\infty.
\end{equation}
Next,  depending on the value of $\limsup_{t \to +\infty} |x_N(t) - x_m(t)|$, we will split the analysis into two cases. 
\newline

\noindent $\diamond$ Case $A$:~Suppose that $\limsup\limits_{t \to +\infty} |x_N(t) - x_m(t)| = C_m$. Combining this assumption and  \eqref{F-8}, we obtain that
\[\lim_{t \to +\infty} |x_N(t) - x_m(t)| = C_m.\]
Now for $i \in E_2$, i.e. $m < i < N$, we have $x_m < x_i < x_N$. Then we claim that 
\[\limsup_{t \to +\infty} |x_N(t) - x_i(t)| = \liminf_{t \to +\infty} |x_N(t) - x_i(t)|.\]
If not, there exists some $i_0 \in E_2$ and $x_m < x_{i_0} < x_N$ such that
\[\liminf_{t \to +\infty} |x_N(t) - x_{i_0}(t)| < \limsup_{t \to +\infty} |x_N(t) - x_{i_0}(t)| \le \limsup_{t \to +\infty} |x_N(t) - x_m(t) | = C_m. \]
Then, the uniform boundedness provides convergence subsequences. Therefore, we can find two time sequences $\{t_l^1\}_{l=1}^\infty$ and $\{t_l^2\}_{l=1}^\infty$ such that
\begin{equation}\label{F-9}
\begin{aligned}
&\lim_{l \to +\infty} t_l^1 = +\infty, \quad \lim_{l \to +\infty} |x_N(t_l^1) - x_{i_0}(t_l^1)| = \liminf_{t \to +\infty} |x_N(t) - x_{i_0}(t)|, \\
&\lim_{l \to +\infty} t_l^2 = +\infty, \quad \lim_{l \to +\infty} |x_N(t_l^2) - x_{i_0}(t_l^2)| = \limsup_{t \to +\infty} |x_N(t) - x_{i_0}(t)|.
\end{aligned}
\end{equation}
Since the sequence $\{|x_N(t) - x_m(t)|\}$  is convergent when $t \to \infty$, there exists a constant $M > 0$ such that for all $i, j \in E_2$
\[|x_i(t) - x_j(t)| \le |x_N(t) - x_m(t)| \le M, \quad t > 0.\] 
In particular, we have the estimates for $t_l^1$ and $t_l^2$, i.e.
\[|x_i(t_l^1) - x_j(t_l^1)| \le M, \quad |x_i(t_l^2) - x_j(t_l^2)| \le M, \quad i,j \in E_2.\]
Then, we can find two subsequences $\{t_{l_p}^1\}_{p=1}^\infty$ and $\{t_{l_p}^2\}_{p=1}^\infty$ of sequences $\{t_l^1\}_{l=1}^\infty$ and $\{t_l^2\}_{l=1}^\infty$, respectively such that
\begin{equation}\label{F-10}
\begin{cases}
\displaystyle \lim \limits_{p \to +\infty} |x_i(t^1_{l_p}) - x_j(t^1_{l_p})| < +\infty, \quad \lim \limits_{p \to +\infty} |x_i(t^2_{l_p}) - x_j(t^2_{l_p})| < +\infty \quad i, j \in E_2, \\
\displaystyle \lim \limits_{p \to +\infty} |x_i(t^1_{l_p}) - x_j(t^1_{l_p})| = +\infty, \quad \lim \limits_{p \to +\infty} |x_i(t^2_{l_p}) - x_j(t^2_{l_p})| = +\infty \quad i \in E_1, \  j \in E_2.
\end{cases}
\end{equation}
From system \eqref{B5}, the dynamics of the particles in $E_2$ will be governed by the following differential equations,
\[\dot{x}_i = \nu_i + \frac{\kappa}{N} \sum_{k=1}^{m-1} \Phi(x_k - x_i) + \frac{\kappa}{N} \underset{k \ne i}{\sum_{k =m}^{N}} \Phi(x_k -x_i) , \quad t > 0, \  i \in E_2.\]
Now, we let $p $ tends to infinity in two sequences $\{t_{l_p}^1\}_{p=1}^\infty$ and $\{t_{l_p}^2\}_{p=1}^\infty$ respectively. Then we  obtain from above differential equations that, for $i \in E_2$,
\begin{equation*}
\begin{cases}
\displaystyle  \lim_{t_{l_p}^1 \to + \infty} \dot{x}_i = \nu_i + \lim_{t_{l_p}^1 \to + \infty} \frac{\kappa}{N} \sum_{k=1}^{m-1} \Phi(x_k - x_i) + \lim_{t_{l_p}^1 \to + \infty} \frac{\kappa}{N} \underset{k \ne i}{\sum_{k =m}^{N}} \Phi(x_k -x_i), \\
\displaystyle  \lim \limits_{t_{l_p}^2 \to + \infty} \dot{x}_i = \nu_i + \lim \limits_{t_{l_p}^2 \to + \infty} \frac{\kappa}{N} \sum\limits_{k=1}^{m-1} \Phi(x_k - x_i) + \lim \limits_{t_{l_p}^2 \to + \infty} \frac{\kappa}{N} \underset{k \ne i}{\sum\limits_{k =m}^{N}} \Phi(x_k -x_i).
\end{cases}
\end{equation*}
Then according to \eqref{F-7}, \eqref{F-10}, the definition of $m$ and the above limits, we immediately obtain the following equalities,
\begin{equation}\label{F-11}
\begin{cases}
\displaystyle  \lim\limits_{t \to +\infty} v_N(t)  = \nu_i -  \frac{\kappa}{N} \sum\limits_{k=1}^{m-1} \Phi^\infty + \frac{\kappa}{N} \underset{k \ne i}{\sum\limits_{k =m}^{N}} \Phi(\lim \limits_{t_{l_p}^1 \to + \infty} (x_k -x_i)), \quad i \in E_2,\\
\displaystyle \lim\limits_{t \to +\infty} v_N(t) = \nu_i -  \frac{\kappa}{N} \sum\limits_{k=1}^{m-1} \Phi^\infty + \frac{\kappa}{N} \underset{k \ne i}{\sum\limits_{k =m}^{N}} \Phi(\lim \limits_{t_{l_p}^2 \to + \infty}  (x_k -x_i)), \quad i \in E_2.
\end{cases}
\end{equation}
Now, we set $y_i := x_i - x_N$. According to \eqref{F-10}, we conclude the limits of $y_i$ exists for the subsequences $ \{t_{l_p}^1\}_{p=1}^\infty$ and $ \{t_{l_p}^2\}_{p=1}^\infty$. Therefore,  
we  may denote $\lim\limits_{ p \to +\infty} y_i(t_{l_p}^1) = y^1_i(+\infty)$, $\lim\limits_{ p \to +\infty} y_i(t_{l_p}^2) = y^2_i(+\infty)$ and $\lim\limits_{t \to +\infty} v_N(t) = v_N^\infty$. Then, \eqref{F-11} can be rewritten as below,
\begin{equation}\label{F-12}
\begin{cases}
\displaystyle 0 = \nu_i -  v_N^\infty - \frac{\kappa}{N} \sum\limits_{k=1}^{m-1} \Phi^\infty + \frac{\kappa}{N} \underset{k \ne i}{\sum\limits_{k =m}^{N}} \Phi(y_k^1(+\infty) - y_i^1(+\infty)), \quad i \in E_2,\\
\displaystyle 0 = \nu_i -  v_N^\infty -  \frac{\kappa}{N} \sum\limits_{k=1}^{m-1} \Phi^\infty + \frac{\kappa}{N} \underset{k \ne i}{\sum\limits_{k =m}^{N}} \Phi(y_k^2(+\infty) - y_i^2(+\infty)), \quad i \in E_2.
\end{cases}
\end{equation}
Next, we shift $y_k^1(+\infty)$ and $y_k^2(+\infty)$ to define $\bar{y}_k^1(+\infty)$ and $\bar{y}_k^2(+\infty)$ for each $k \in E_2$ as below,
\[\bar{y}_k^1(+\infty) = y_k^1(+\infty) - \frac{1}{N- m+1} \sum_{i = m}^{N} y_i^1(+\infty), \quad \bar{y}_k^2(+\infty) = y_k^2(+\infty) - \frac{1}{N- m+1} \sum_{i = m}^{N} y_i^2(+\infty),\]
which both satisfy the zero sum condition $\sum_{i = m}^N \bar{y}_i^1(+\infty) = \sum_{i = m}^N \bar{y}_i^2(+\infty) = 0$.
Then, \eqref{F-12} can be rewritten as 
\begin{equation*}
\begin{cases}
\displaystyle 0 = \nu_i -  v_N^\infty - \frac{\kappa}{N} \sum\limits_{k=1}^{m-1} \Phi^\infty + \frac{\kappa}{N} \underset{k \ne i}{\sum\limits_{k =m}^{N}} \Phi(\bar{y}_k^1(+\infty) - \bar{y}_i^1(+\infty)), \quad i \in E_2,\\
\displaystyle 0 = \nu_i -  v_N^\infty -  \frac{\kappa}{N} \sum\limits_{k=1}^{m-1} \Phi^\infty + \frac{\kappa}{N} \underset{k \ne i}{\sum\limits_{k =m}^{N}} \Phi(\bar{y}_k^2(+\infty) - \bar{y}_i^2(+\infty)), \quad i \in E_2.
\end{cases}
\end{equation*}
From \eqref{F-9}, we have
\begin{equation*}
\begin{cases}
\displaystyle y_{i_0}^1(+\infty) = \lim_{t_{l_p}^1 \to +\infty}(x_{i_0} - x_N) = - \lim_{l \to +\infty} |x_N(t_l^1) - x_{i_0}(t_l^1)| = - \liminf_{t \to +\infty} |x_N(t) - x_{i_0}(t)|, \\
\displaystyle y_{i_0}^2(+\infty) = \lim_{t_{l_p}^2 \to +\infty}(x_{i_0} - x_N) = - \lim_{l \to +\infty} |x_N(t_l^2) - x_{i_0}(t_l^2)| = - \limsup_{t \to +\infty} |x_N(t) - x_{i_0}(t)|.
\end{cases}
\end{equation*}
It yields that $y_{i_0}^1(+\infty) \ne y_{i_0}^2(+\infty)$, hence, $\bar{y}_{i_0}^1(+\infty) \ne \bar{y}_{i_0}^2(+\infty)$. Therefore, we construct  two different solutions $\{\bar{y}_k^1(+\infty)\}_{k=m}^N$ and $\{\bar{y}_k^2(+\infty)\}_{k=m}^N$ to the same equation \eqref{F-12}, which is a contradiction to Lemma \ref{equilibrium}.
\newline

\noindent $\diamond$ Case $B$:~On the other hand, suppose that 
\[\limsup_{t \to +\infty} |x_N(t) - x_m(t)| > C_m.\]
That is,  the limit of $|x_N(t) - x_m(t)|$ when $t $ tends to infinity does not exist. Then we can find two time sequences $\{t_l^1\}_{l=1}^\infty$ and $\{t_l^2\}_{l=1}^\infty$ such that
\begin{equation*}
\begin{cases}
\displaystyle \lim_{l \to +\infty} t^1_l = +\infty, \quad \lim_{l \to +\infty} t^2_l = +\infty, \quad \lim_{l \to +\infty} |x_N(t^1_l)- x_m(t^1_l)| = \liminf_{t \to +\infty} |x_N(t) - x_m(t)|, \\
\displaystyle \lim_{l \to +\infty} |x_N(t^2_l)- x_m(t^2_l)|  = \min \left\{ \limsup_{t \to +\infty} |x_N(t) - x_m(t)| ,
\liminf_{t \to +\infty} |x_N(t) - x_m(t)| + 1\right\}.
\end{cases}
\end{equation*}
Then, using the same arguments as in Case A, we can find two subsequences $\{t_{l_p}^1\}_{p=1}^\infty$ and $\{t_{l_p}^2\}_{p=1}^\infty$ respectively such that for $i, j \in E_2$,
\begin{equation*}
\begin{cases}
\displaystyle \exists \ \lim_{ p \to +\infty} |x_i(t_{l_p}^1)- x_j(t_{l_p}^1)| < +\infty, \quad \exists \ \lim_{ p \to +\infty} |x_i(t_{l_p}^2)- x_j(t_{l_p}^2)| < +\infty, \quad i, j \in E_2,\\
\displaystyle \lim_{ p \to +\infty} |x_i(t_{l_p}^1)- x_j(t_{l_p}^1)|  = +\infty, \quad \lim_{ p \to +\infty} |x_i(t_{l_p}^2)- x_j(t_{l_p}^2)| = +\infty, \quad i\in E_1,\ j \in E_2.
\end{cases}
\end{equation*}
Then, by same analysis as in Case A, we can obtain two different solutions of the corresponding constrained equilibrium system, since at least $\bar{y}_m^1(+\infty) \ne \bar{y}_m^2(+\infty)$. This is a contradiction according to Lemma \ref{equilibrium}.

Now we combine the analysis in Case A and Case B to obtain the desired result.\newline
\end{proof}

\begin{corollary}\label{|x_i - x_j|}
Let $X$ be a solution of system \eqref{B5} with initial data $X^0$ and assume the components of $X^0$ are sorted as below
\[x_1^0 < x_2^0 < \ldots < x_N^0.\]
Then we can seperate $N$ particles as follows:
\begin{equation}\label{F-13}
\begin{cases}
\displaystyle F_1 = \{\ i\ : \ \exists \lim_{t \to +\infty} |x_N(t) - x_i(t)| < +\infty \}, \\
\displaystyle F_2 = \{\ i : \ \lim_{t \to +\infty} |x_N(t) - x_i(t)| = +\infty\}.
\end{cases}
\end{equation}
Moreover, we have
\begin{equation*}
\begin{cases}
\displaystyle \lim_{t \to +\infty} |x_i(t) - x_j(t)| < +\infty,  \quad i, j \in F_1,\\
\displaystyle \lim_{t \to +\infty} |x_i(t) - x_j(t) | = +\infty, \quad i \in F_1, \ j \in F_2.
\end{cases}
\end{equation*}
\end{corollary}
\begin{proof}
Under the assumption that $x_1^0 < \ldots < x_N^0$, it is known that components of $X$ are ordered, that is, $x_1(t) < x_2(t) < \cdots < x_N(t), t \ge 0.$
Then it follows from Lemma \ref{|x_i - x_N|} that \eqref{F-13} holds. Combing the definition of $F_1$ and the continuity of $X$, for $i, j \in F_1$,  we have
\[\exists \ \lim_{t \to +\infty} (x_N(t) - x_i(t)) < +\infty \quad \text{and} \quad \exists \ \lim_{t \to +\infty} (x_N(t) - x_j(t)) < +\infty.\]
Then, it is easy to see that the limit of $x_i - x_j$ exists for $i,j \in F_1$ since
\[\lim_{t \to +\infty} (x_i(t) - x_j(t)) = - \lim_{t \to +\infty} (x_N(t) - x_i(t)) + \lim_{t \to +\infty} (x_N(t) - x_j(t)), \quad i ,j \in F_1.\]
That is, $\exists \ \lim_{t \to +\infty} |x_i(t) - x_j(t)| < +\infty,  \ i, j \in F_1.$
Next, for $i \in F_1, j \in F_2$, by triangle inequality we have
\[|x_i(t) - x_j(t)| \ge |x_N(t) - x_j(t)| - |x_N(t) - x_i(t)|.\]
We take inferior limit on both sides of above inequality to obtain that
\[\liminf_{t \to +\infty} |x_i(t) - x_j(t)| \ge \lim_{t \to +\infty} |x_j(t) - x_N(t)| - \lim_{t \to +\infty} |x_N(t) - x_i(t)| = +\infty,\]
due to the fact that $\lim_{t \to +\infty} |x_j(t) - x_N(t)| = +\infty.$
Thus,  for $i \in F_1, j \in F_2$, we have \[\lim_{t \to +\infty} |x_i(t) - x_j(t)| = +\infty, \quad i \in F_1, \ j \in F_2.\]
\end{proof}

\noindent $\mathbf{Proof\ of\ Lemma\ \ref{non-oscillatory}.}$
From Lemma \ref{|x_i - x_N|} and Corollary \ref{|x_i - x_j|}, it is known that $m \in \mathcal{N}$ can be found such that
\[F_1 = \{m, \ldots, N\}, \qquad F_2 = \{1, \ldots, m-1\},\]
where $F_1$ and $F_2$ are defined in \eqref{F-13}.
Thus, we can repeat the same arguments in Lemma \ref{|x_i - x_N|} and Corollary \ref{|x_i - x_j|} inductively to seperate $\{x_1, x_2, \ldots, x_N\}$ into finite subsets:
\begin{equation*}
\begin{cases}
\displaystyle \{x_1, x_2,\ldots, x_N\} = \bigcup_{j=1}^q\mathcal{G}_j, \quad \mathcal{G}_j = \{x^j_1, \ldots, x^j_{n_j}\}, \quad j =1, \ldots, q, \\
\displaystyle x^1_1 = x_1, \ x^q_{n_q} = x_N, \ x^j_1 = x_{1+\sum_{i=1}^{j-1}n_i }, \\
\displaystyle \lim_{t \to +\infty} |x^j_k- x^j_l| < +\infty, \quad \lim_{t \to +\infty} |\dot{x}^j_k - \dot{x}^j_l| = 0, \\
\displaystyle \lim_{t \to +\infty} |x^i_k - x^j_l| = +\infty, \quad i \ne j. 
\end{cases}
\end{equation*}
\qed
\vspace{0.3cm}
\section{Proof of Lemma \ref{critical}}\label{p-critical}
\setcounter{equation}{0}
\vspace{0.5cm}
\noindent $\mathbf{Proof\ of\ Lemma\ \ref{critical}.}$
(1)~Suppose $\kappa > \kappa_c$, we claim that mono-cluster flocking will emerge asymptotically:
\[x_{ij}^\infty := \lim_{t \to +\infty} |x_i(t) - x_j(t)| < +\infty, \quad 1 \le i, j \le N.\]
If not, there exists some $l_0 \ (1 \le l_0 \le N-1)$ such that $\lim\limits_{t \to +\infty} (x_{l_0+1}(t) - x_{l_0}(t) ) = +\infty$. Combing the zero sum condition $\sum_{i=1}^N x_i(t) = 0$ and the order relation $x_1(t) < x_2(t) < \cdots < x_N(t), \ t> 0,$ we have 
\[\sum_{i=1}^{l_0} x_i(t) \le 0, \quad \sum_{i=l_0 + 1}^{N} x_i(t) \ge 0, \quad \lim_{ t\to +\infty} \sum_{i=1}^{l_0} x_i(t) = -\infty, \quad \lim_{ t\to +\infty} \sum_{i=l_0 + 1}^{N} x_i(t) = +\infty.\]
We next consider the derivative of $\sum_{i=1}^{l_0} x_i(t)$,
\begin{align}\label{F-14}
\begin{aligned}
\frac{d}{dt} \left(\sum_{i=1}^{l_0} x_i(t)\right)&= \sum_{i=1}^{l_0} [\nu_i + \frac{\kappa}{N} \underset{k \ne i}{\sum_{k=1}^N} \Phi(x_k-x_i)] = \sum_{i=1}^{l_0} \nu_i + \frac{\kappa}{N} \sum_{i=1}^{l_0} \sum_{k=l_0+1}^N \Phi(x_k-x_i) \\
& \ge \sum_{i=1}^{l_0} \nu_i + \frac{\kappa}{N} l_0 (N-l_0) \Phi(x_{l_0+1} - x_{l_0}).
\end{aligned}
\end{align}
It follows from $\kappa > \kappa_c $ and the definition of $\kappa_c$ that $\Phi^\infty  > - \frac{\frac{1}{l} \sum_{i=1}^{l} \nu_i}{\frac{N-l}{N} \kappa}$ for all $1 \le l \le N-1$. Then we set
\[C_0 = \max_{1 \le l \le N-1} \Phi^{-1} \left( - \frac{\frac{1}{l} \sum_{i=1}^{l} \nu_i}{\frac{N-l}{N} \kappa}\right),  \quad 1 \le l \le N-1,\]
where $\Phi^{-1} \left( - \frac{\frac{1}{l} \sum_{i=1}^{l} \nu_i}{\frac{N-l}{N} \kappa}\right) > 0$ due to the inverse $\Phi^{-1}$ is defined for $\Phi(x)$ when $x>0$. According to the definition of $l_0$,  for the positive constant $C_0$, there exists $t_0>0$ such that
\[x_{l_0+1}(t) - x_{l_0}(t) \ge C_0, \quad \text{for} \ t \ge t_0. \]
Then due to the monotone increasing property of $\Phi$ on $(0, + \infty)$ and the definition of $C_0$, we immediately obtain  that
\[\Phi(x_{l_0}(t) - x_{l_0}(t)) \ge \Phi(C_0) \ge - \frac{\frac{1}{l_0} \sum_{i=1}^{l_0} \nu_i}{\frac{N-l_0}{N} \kappa}, \quad t \ge t_0.\]
Thus, combing above inequality and \eqref{F-14}, we obtain that
\[\frac{d}{dt} (\sum_{i=1}^{l_0} x_i(t)) \ge 0, \qquad t \ge t_0.\]
We integrate on both sides of above formula to obtain that $\sum_{i=1}^{l_0} x_i(t) \ge \sum_{i=1}^{l_0} x_i(t_0)$, which contradicts to the fact that $\lim\limits_{ t\to +\infty} \sum_{i=1}^{l_0} x_i(t) = -\infty$.
\newline

\noindent $(2)$~Suppose that mono-cluster flocking emerges, i.e., there exists an  asymptotic relative equilibrium state $X^\infty$ such that $x_i^\infty :=\lim\limits_{t \to +\infty} x_i(t)$. Then, we will prove  $\kappa > \kappa_c$ by contradiction. Supppose $\kappa \le \kappa_c$. Then there exists $ 1 \le r_0 \le N-1$ such that
\begin{equation}\label{F-15}
\kappa \le - \frac{\frac{1}{r_0} \sum_{i=1}^{r_0} \nu_i}{\frac{N-r_0}{N} \Phi^\infty} \quad \text{i.e.} \quad \frac{\kappa}{N} r_0(N-r_0) \Phi^\infty + \sum_{i=1}^{r_0} \nu_i \le 0.
\end{equation}
Next we apply \eqref{F-15} and the fact that $\sum_{i=1}^{r_0} \underset{k \ne i } {\sum_{k =1}^{r_0}} \Phi(x_k - x_i) = 0$ to obtain the estimate of the derivative of $\sum_{i=1}^{r_0} x_i$ as below,
\begin{align}\label{F-16}
\begin{aligned}
\frac{d}{dt} \left(\sum_{i=1}^{r_0} x_i(t)\right) = \sum_{i=1}^{r_0} \nu_i +  \frac{\kappa}{N} \sum_{i=1}^{r_0} \sum_{k =r_0+1}^{N} \Phi(x_k - x_i) \le \sum_{i=1}^{r_0} \nu_i + \frac{\kappa}{N} r_0(N-r_0) \Phi^\infty \le 0, \quad t \ge 0 .
\end{aligned}
\end{align}
 Thus, it yields from \eqref{F-16} that
$\lim_{t \to +\infty} \sum_{i=1}^{r_0} x_i(t) = -\infty,$
which obviously contradicts to 
\[\lim_{t \to +\infty} \sum_{i=1}^{r_0} x_i(t) = \sum_{i=1}^{r_0} x_i^\infty > -\infty.\]
\newline
We combine all of above analysis to obtain the desired result.\newline
\qed
\vspace{0.3cm}
\section{Proof of Theorem \ref{T4.2}}\label{p-f-predictability}
\setcounter{equation}{0}
\vspace{0.3cm}
\noindent $\mathbf{Proof\ of\ Theorem \ \ref{T4.2}.}$
$(i)$~ We split the proof into two steps by induction, Step $A$ and Step $B$. In the initial step, we claim that the set $\mathcal{I}_1$ form a maximal cluster in the sense that it forms an asymptotic cluster. In the inductive step, we assume that for any $k \in \{1,2, \ldots, N_c(\kappa) -1\}$ and each $1 \le j \le k$, the set $\mathcal{I}_j$ forms a maximal cluster. Then we prove that the set $\mathcal{I}_{k+1}$ form a maximal cluster.
 
 Since $N_c(\kappa) = 1$ has been treated in Lemma \ref{critical}, we consider the case that $N_c(\kappa) \ge 2$.
 \newline
 
 \noindent $\bullet$ Step $A$ (Initial step):~In this step, we divide the step into two sub-steps, Step $A.1$ and Step $A.2$. In Step $A.1$, we claim that the set $\mathcal{I}_1$ form a cluster. From Definition \ref{D2.2}, it is known that
 \[\sup_{0 \le t < +\infty} \max_{k,l \in \mathcal{I}_1}  |x_k(t) - x_l(t)| < +\infty.    \]
 Then from Lemma \ref{non-oscillatory}, we obtain that the limit of the relative distance of particles in $\mathcal{I}_1$ exists, i.e.
 \[\exists \ \lim_{t \to +\infty} |x_k(t) - x_l(t)| < +\infty, \quad k,l \in \mathcal{I}_1.\]
 In Step $A.2$, we will prove that the set $\mathcal{I}_1$ form a maximal cluster.
\newline

\noindent $\diamond$ Step $A.1$:~If $n_1 = 1$, we are done. It is assumed that $n_1 \ge 2$. Suppose that $\mathcal{I}_1$ constructed by the clustering algorithm \eqref{f-algorithm} can not form a cluster. Then there exists $1 \le r \le n_1 -1$ such that
\begin{equation}\label{F-17}
\lim_{t \to +\infty} (x_{r+1}(t) -x_r(t)) = +\infty.
\end{equation}
 It follows from the algorithm \eqref{f-algorithm} that
 \begin{equation}\label{F-18}
 \frac{1}{k} \sum_{j=1}^k \hat{\nu}_j^{(0,n_1]} + \kappa \frac{n_1 - k}{N} \Phi^\infty > 0, \quad 0 < k \le n_1 -1.
 \end{equation}
 Then, the dynamics of the center of  the first $n_1$ particles is governed by the following differential inequality,
 \begin{align}\label{F-20}
 \begin{aligned}
\frac{d}{dt} (\bar{x}^{(0,n_1]}(t)) &= \frac{1}{n_1} \sum_{i=1}^{n_1} [\nu_i + \frac{\kappa}{N} \underset{k \ne i}{\sum_{k=1}^N} \Phi(x_k - x_i)]\\
& = \bar{\nu}^{(0, n_1]} + \frac{1}{n_1}\frac{\kappa}{N} \sum_{i=1}^{n_1} \sum_{k=n_1+1}^{N} \Phi(x_k - x_i)\le \bar{\nu}^{(0, n_1]} + \kappa \frac{N-n_1}{N} \Phi^\infty := v_1^\infty.
 \end{aligned}
 \end{align}
 That is, the center of mass of first $n_1$ particles stay less than or equal to the linear growth with slope $v_1^\infty$:
 \[\bar{x}^{(0,n_1]}(t) \le \bar{x}^{(0,n_1]}(0) + v_1^\infty t.\]
 We next consider the dynamics of the quantity $\bar{x}^{(0,r]}$, which is governed by the following differential equation,
 \begin{align}\label{F-21}
 \begin{aligned}
 \frac{d}{dt} \bar{x}^{(0,r]} &= \frac{1}{r} \sum_{i=1}^r [\nu_i + \frac{\kappa}{N} \underset{k \ne i}{\sum_{k=1}^N} \Phi(x_k - x_i)] \\
 &= \frac{1}{r} \sum_{i=1}^r [\hat{\nu}_i^{(0, n_1]} + \bar{\nu}^{(0, n_1]} + \frac{\kappa}{N} \sum_{k=r+1}^{n_1} \Phi(x_k - x_i) + \frac{\kappa}{N} \sum_{k=n_1+1}^{N} \Phi(x_k - x_i)] \\
 &= \left( \frac{1}{r} \sum_{i=1}^r \hat{\nu}_i^{(0, n_1]} + \frac{1}{r} \frac{\kappa}{N} \sum_{i=1}^r \sum_{k=r+1}^{n_1} \Phi(x_k - x_i) \right) + \left( \bar{\nu}^{(0, n_1]} + \frac{1}{r} \frac{\kappa}{N} \sum_{i=1}^r \sum_{k=n_1+1}^{N} \Phi(x_k - x_i) \right) \\
 &=: \mathcal{K}_{11} + \mathcal{K}_{12}. 
 \end{aligned}
 \end{align}

 \noindent $\star$ (Estimate of $\mathcal{K}_{11}$):~Since $1 \le r \le n_1-1$, from \eqref{F-17} and \eqref{F-18}, we have
 \[\lim_{t \to +\infty} \left( \frac{1}{r} \sum_{i=1}^r \hat{\nu}_i^{(0, n_1]} + \frac{1}{r} \frac{\kappa}{N} \sum_{i=1}^r \sum_{k=r+1}^{n_1} \Phi(x_k - x_i) \right) = \frac{1}{r} \sum_{i=1}^r \hat{\nu}_i^{(0, n_1]} + \frac{\kappa}{N} (n_1 - r) \Phi^\infty > 0.\]
 Thus, for a given small $\varepsilon > 0$, there exists $t_\varepsilon > 0$ such that
 \begin{equation}\label{F-22}
 \frac{1}{r} \sum_{i=1}^r \hat{\nu}_i^{(0, n_1]} + \frac{1}{r} \frac{\kappa}{N} \sum_{i=1}^r \sum_{k=r+1}^{n_1} \Phi(x_k - x_i) > \varepsilon, \quad t \ge t_{\varepsilon}.
 \end{equation}
 
 \noindent $\star$ (Estimate of $\mathcal{K}_{12}$) :~From \eqref{F-17} and \eqref{F-20}, it is obvious to see that
 \[\lim_{t \to +\infty} \left( \bar{\nu}^{(0, n_1]} + \frac{1}{r} \frac{\kappa}{N} \sum_{i=1}^r \sum_{k=n_1+1}^{N} \Phi(x_k - x_i) \right) = \bar{\nu}^{(0, n_1]} + \frac{\kappa}{N}(N - n_1) \Phi^\infty = v_1^\infty.\]
 Then for the above given $\varepsilon > 0$, there exists $\tilde{t}_\varepsilon$ such that
 \begin{equation}\label{F-23}
  \bar{\nu}^{(0, n_1]} + \frac{1}{r} \frac{\kappa}{N} \sum_{i=1}^r \sum_{k=n_1+1}^{N} \Phi(x_k - x_i) > v_1^\infty - \frac{\varepsilon}{2}, \quad t \ge \tilde{t}_\varepsilon.
  \end{equation}
Thus, we combine \eqref{F-21}, \eqref{F-22}, and\eqref{F-23} to obtain that
\[\frac{d}{dt} \bar{x}^{(0,r]} \ge \varepsilon + v_1^\infty - \frac{\varepsilon}{2} = v_1^\infty + \frac{\varepsilon}{2}> v_1^\infty \ge \frac{d}{dt} \bar{x}^{(0,n_1]}, \quad t \ge \max \{t_\varepsilon, \tilde{t}_\varepsilon\}.\]
Then a finite time $t_c > \max \{t_\varepsilon, \tilde{t}_\varepsilon\} > 0$ can be found such that
\begin{equation}\label{F-24}
\bar{x}^{(0,r]}(t) > \bar{x}^{(0,n_1]}(t), \quad \text{for} \  t \ge t_c.
\end{equation}
However, the order relation of particles yields that
\[\bar{x}^{(0,r]}(t) < \bar{x}^{(0,n_1]}(t) , \quad t \ge 0,\]
which obviously contradicts to \eqref{F-24}. Thus, the particles in $\mathcal{I}_1$ will not depart from each other.
\newline

\noindent $\diamond$ Step $A.2$:~We next show that particles in $\mathcal{I}_1$ departs from the other particles. If not, suppose that there is a larger set $\tilde{\mathcal{I}}_1$ such that the particles in $\tilde{\mathcal{I}}_1$ make a maximal cluster and $\mathcal{I}_1 \subsetneq \tilde{\mathcal{I}}_1$. Due to the order relation in $X$ and $N_c(\kappa) \ge 2$, there exists $n_1 < n_* < N$ such that
\[\tilde{\mathcal{I}}_1 := \{1, \ldots, n_*\}.\]
Since $\tilde{\mathcal{I}}_1$ is a maximal cluster, from Definition \ref{D2.2}, we have
\begin{equation}\label{F-25}
\lim_{t \to +\infty} (x_{n_* + 1}(t) - x_{n_*}(t)) = +\infty. 
\end{equation}
Then, the dynamics of center of mass for the group $\tilde{\mathcal{I}}_1$ satisfies the following differential inequality,
\begin{align}\label{F-26}
\begin{aligned}
\frac{d}{dt} \bar{x}^{(0,n_*]}&= \frac{1}{n_*} \sum_{i = 1}^{n_*} [\nu_i + \frac{\kappa}{N} \underset{k \ne i}{\sum_{k=1}^N} \Phi(x_k - x_i)] = \frac{1}{n_*} \sum_{i = 1}^{n_*} [\nu_i + \frac{\kappa}{N} (\underset{k \ne i}{\sum_{k=1}^{n_*}} \Phi(x_k - x_i) + \sum_{k = n_* + 1}^N \Phi(x_k- x_i))] \\
&= \bar{\nu}^{(0, n_*]} + \frac{1}{n_*} \frac{\kappa}{N} \sum_{i = 1}^{n_*} \sum_{k = n_* + 1}^N \Phi(x_k - x_i) \le \bar{\nu}^{(0, n_*]} + \frac{\kappa(N- n_*)}{N} \Phi^\infty := \tilde{v}^\infty_1.
\end{aligned}
\end{align}
Thus, for $1 \le l < n_*$, it follows from the order relation in $X$ and \eqref{F-26} that
\[\bar{x}^{(0,l]}(t) < \bar{x}^{(0, n_*]}(t) \le \bar{x}^{(0, n_*]}(0) + \tilde{v}^\infty_1 t, \quad 1 \le l < n_*, \ t \ge 0.\]
Then we study the dynamics of the quantity $\bar{x}^{(0,l]} - \tilde{v}^\infty_1 t$ through the differential equation below,
\begin{equation}\label{F-27}
\begin{aligned}
&\frac{d}{dt} (\bar{x}^{(0, l]} - \tilde{v}^\infty_1 t) \\
 &= \frac{1}{l} \sum_{i=1}^l [\nu_i + \frac{\kappa}{N} (\underset{k \ne i}{\sum_{k=1}^{l}} \Phi(x_k - x_i) + \sum_{k= l+1}^{n_*} \Phi(x_k -x_i) + \sum_{k =n_* + 1}^{N} \Phi(x_k- x_i))]  - \tilde{v}^\infty_1 \\
&= \frac{1}{l} \sum_{i=1}^l (\hat{\nu}_i^{(0, n_*]} + \bar{\nu}^{(0, n_*]}) + \frac{1}{l} \frac{\kappa}{N} \sum_{i=1}^l \sum_{k= l+1}^{n_*} \Phi(x_k -x_i) + \frac{1}{l} \frac{\kappa}{N} \sum_{i=1}^l \sum_{k =n_* + 1}^{N} \Phi(x_k- x_i) - \tilde{v}^\infty_1 \\
&= \left( \frac{1}{l} \sum_{i=1}^l \hat{\nu}_i^{(0, n_*]} + \frac{1}{l} \frac{\kappa}{N} \sum_{i=1}^l \sum_{k= l+1}^{n_*} \Phi(x_k -x_i) \right) + \left( \frac{1}{l} \frac{\kappa}{N} \sum_{i=1}^l \sum_{k =n_* + 1}^{N} \Phi(x_k- x_i) - \frac{\kappa(N-n_*)}{N} \Phi^\infty \right)\\
&=: \mathcal{K}_{21} + \mathcal{K}_{22}.
\end{aligned}
\end{equation}
Since $\tilde{\mathcal{I}}_1$ is a maximal cluster, it follows from Definition \ref{D2.2} that there exists a sufficiently small constant $\varepsilon > 0$ such that
\begin{equation}\label{F-28}
\lim_{t \to +\infty} \frac{1}{l} \frac{\kappa}{N} \sum_{i=1}^l \sum_{k= l+1}^{n_*} \Phi(x_k -x_i) + \varepsilon < \frac{\kappa(n_* - l)}{N} \Phi^\infty .
\end{equation}
Since the velocity of the particles in the same cluster will align, from \eqref{F-26}, we note that all particles in group $\tilde{\mathcal{I}}_1$ will align asymptotically at velocity $\tilde{v}^\infty_1$. Therefore, for the above given small $\varepsilon > 0$, there exists a time $t_*^1$ such that
\begin{equation}\label{F-29}
\left| \frac{d}{dt} (\bar{x}^{(0, l]} - \tilde{v}^\infty_1 t) \right| \le \frac{\varepsilon}{2} , \quad t \ge t_*^1.
\end{equation}
Since the particles in different groups will segregate, we have
\[\lim_{t \to +\infty} |x_i(t) - x_k(t)| = +\infty, \quad 1 \le i \le l, \ n_* + 1 \le k \le N.\]
Thus, we get
\[\lim_{t \to +\infty} \frac{1}{l} \frac{\kappa}{N} \sum_{i=1}^l \sum_{k =n_* + 1}^{N} \Phi(x_k- x_i) = \frac{\kappa(N-n_*)}{N} \Phi^\infty.\]
Hence, for the above given small $\varepsilon > 0$, there exists a time $t_*^2$ such that
\begin{equation}\label{F-30}
\left|\mathcal{K}_{22} \right| = \left| \frac{1}{l} \frac{\kappa}{N} \sum_{i=1}^l \sum_{k =n_* + 1}^{N} \Phi(x_k- x_i) - \frac{\kappa(N-n_*)}{N} \Phi^\infty \right| \le \frac{\varepsilon}{2}, \quad t \ge t_*^2.
\end{equation}
Then, by triangle inequality, we combine \eqref{F-29} and \eqref{F-30} to obtain that
\begin{equation}\label{G-1}
 |\mathcal{K}_{21}| \le \left| \frac{d}{dt} (\bar{x}^{(0, l]} - \tilde{v}^\infty_1 t) \right| + \left|\mathcal{K}_{22} \right| \le \varepsilon, \quad t \ge \max \{t_*^1, t_*^2\}.
 \end{equation}
 Thus, we combine \eqref{F-28} and \eqref{G-1} to obtain that for $1 \le l \le n_* - 1$,
 \[\frac{1}{l} \sum_{i=1}^l \hat{\nu}_i^{(0, n_*]} \ge \lim_{
 t \to +\infty} - \frac{1}{l} \frac{\kappa}{N} \sum_{i=1}^l \sum_{k= l+1}^{n_*} \Phi(x_k -x_i) - \varepsilon > - \frac{\kappa(n_* - l)}{N} \Phi^\infty.\]
 That is, 
 \[\frac{1}{l} \sum_{i=1}^l \hat{\nu}_i^{(0, n_*]} + \frac{\kappa(n_* - l)}{N} \Phi^\infty >0, \quad \text{for} \  1 \le l \le n_* -1.\]
However, it follows from the algorithm $\eqref{f-algorithm}_3$ that $n_* \le n_1$ which contradicts to $n_* > n_1$. Therefore, $\mathcal{I}_1$ is a maximal cluster, and the other particles in $\{n_1 + 1, \ldots, N\} $ will depart from the group $\mathcal{I}_1$.
\newline

\noindent $\bullet$ Step $B$: In this step, we will show that each group $\mathcal{I}_j$ is a maximal cluster by induction. In Step $A$, we have proved that the group $\mathcal{I}_1$ is a maximal cluster. Now suppose that for $1 \le i \le N_c(\kappa) -1$, sub-ensembles $\mathcal{I}_j(1 \le j \le i)$ are maximal clusters, then we prove that $\mathcal{I}_{i+1}$ is a maximal cluster. We show this by contradiction. Suppose the particles in $\mathcal{I}_{i+1}$ can not form a cluster. Then, there exists $n_i < r \le n_{i+1} - 1$ such that
\begin{equation}\label{G-2}
\lim_{t \to +\infty} (x_{r+1}(t) - x_r(t)) = +\infty.
\end{equation}
We set
\begin{equation}\label{G-3}
v_{i+1}^\infty := \bar{\nu}^{(n_i, n_{i+1}]} + \frac{\kappa(N- n_{i+1} - n_i) \Phi^\infty}{N}.
\end{equation}
Then the center of mass of particles in $\mathcal{I}_{i+1}$ satisfies
\begin{align*}
\begin{aligned}
&\frac{d}{dt} \bar{x}^{(n_i, n_{i+1}]}\\
&= \frac{1}{n_{i+1} -n_i} \sum_{k = n_i + 1}^{n_{i+1}} [ \nu_k + \frac{\kappa}{N} \underset{j \ne k}{\sum_{j=1}^N} \Phi(x_j - x_k)] \\
&= \bar{\nu}^{(n_i, n_{i+1}]} + \frac{\kappa}{N}  \frac{1}{n_{i+1} -n_i} \sum_{k = n_i + 1}^{n_{i+1}} \sum_{j = n_{i+1} +1}^N \Phi(x_j - x_k) + \frac{\kappa}{N} \frac{1}{n_{i+1} -n_i} \sum_{k = n_i + 1}^{n_{i+1}}  \sum_{j = 1}^{n_i} \Phi(x_j - x_k) \\
&= \bar{\nu}^{(n_i, n_{i+1}]} + \mathcal{R}_{11} + \mathcal{R}_{12}.
\end{aligned}
\end{align*}
Due to the well-ordering of particles, we have
\begin{equation}\label{G-4}
\mathcal{R}_{11} = \frac{1}{n_{i+1} -n_i} \sum_{k = n_i + 1}^{n_{i+1}}  \sum_{j = n_{i+1}+1}^{N} \Phi(x_j - x_k) \le \frac{\kappa(N- n_{i+1})}{N} \Phi^\infty .
\end{equation}
According the assumption of induction, it is known that all the particles with index $j \le n_i$ will depart from $\mathcal{I}_{i+1}$.
Then, we obtain that
\[\lim_{t \to +\infty} \mathcal{R}_{12}=\lim_{t \to +\infty} \frac{\kappa}{N} \frac{1}{n_{i+1} -n_i} \sum_{k = n_i + 1}^{n_{i+1}}  \sum_{j = 1}^{n_i} \Phi(x_j - x_k) = -\frac{\kappa n_i}{N} \Phi^\infty.\]
Thus, for a given small constant $\varepsilon > 0$, there exists a time $t_0$ such that
\begin{equation}\label{G-5}
\mathcal{R}_{12} = \frac{\kappa}{N} \frac{1}{n_{i+1} -n_i} \sum_{k = n_i + 1}^{n_{i+1}}  \sum_{j = 1}^{n_i} \Phi(x_j - x_k) \le -\frac{\kappa n_i}{N} \Phi^\infty + \varepsilon, \quad t \ge t_0.
\end{equation}
We combine \eqref{G-4} and \eqref{G-5} to obtain that
\begin{equation}\label{G-6}
\frac{d}{dt} \bar{x}^{(n_i, n_{i+1}]} \le \bar{\nu}^{(n_i, n_{i+1}]} + \frac{\kappa(N- n_{i+1} - n_i)}{N} \Phi^\infty + \varepsilon = v_{i+1}^\infty + \varepsilon, \quad \text{for} \ t \ge t_0.
\end{equation}
Thus, the center of mass of particles in $\mathcal{I}_{i+1}$ will stay less than or equal to the linear growth with a slope $v_{i+1}^\infty + \varepsilon$,
\[\bar{x}^{(n_i, n_{i+1}]}(t) \le \bar{x}^{(n_i, n_{i+1}]}(0) + (v_{i+1}^\infty + \varepsilon) (t- t_0), \quad t \ge t_0.\]
We next take the time derivative on $\bar{x}^{(n_i, r]} - (v_{i+1}^\infty + \varepsilon)t$, 
\begin{align*}
\begin{aligned}
&\frac{d}{dt} (\bar{x}^{(n_i, r]} - (v_{i+1}^\infty + \varepsilon)t) \\
&= \frac{1}{r-n_i} \sum_{k= n_i+1}^r [\nu_k + \frac{\kappa}{N} \underset{l \ne k}{\sum_{l=1}^N} \Phi(x_l -x_k)] - (v_{i+1}^\infty + \varepsilon) \\
&= \frac{1}{r-n_i} \sum_{k= n_i+1}^r ( \hat{\nu}_k^{(n_i, n_{i+1}]} + \bar{\nu}^{(n_i, n_{i+1}]})  - (v_{i+1}^\infty + \varepsilon)  + \frac{1}{r-n_i} \frac{\kappa}{N} \sum_{k= n_i+1}^r \sum_{l = r+1}^{n_{i+1}} \Phi(x_l-x_k) \\
&\hspace{0.5cm} + \frac{1}{r-n_i} \frac{\kappa}{N} \sum_{k= n_i+1}^r \left( \sum_{l=1}^{n_i} \Phi(x_l - x_k) + \sum_{l = n_{i+1} + 1}^N \Phi(x_l - x_k)\right) \\
&= \bar{\nu}^{(n_i, n_{i+1}]} - v_{i+1}^\infty  + \frac{1}{r-n_i} \sum_{k= n_i+1}^r \hat{\nu}_k^{(n_i, n_{i+1}]} + \frac{1}{r-n_i} \frac{\kappa}{N} \sum_{k= n_i+1}^r \sum_{l = r+1}^{n_{i+1}} \Phi(x_l-x_k) \\
&\hspace{0.5cm} + \frac{1}{r-n_i} \frac{\kappa}{N} \sum_{k= n_i+1}^r \left( \sum_{l=1}^{n_i} \Phi(x_l - x_k) + \sum_{l = n_{i+1} + 1}^N \Phi(x_l - x_k)\right) - \varepsilon \\
&= \mathcal{R}_{21} + \mathcal{R}_{22} + \mathcal{R}_{23} - \varepsilon.
\end{aligned}
\end{align*}

\noindent $\star$ (Estimate of $\mathcal{R}_{21}$):~
From \eqref{G-3}, it is easy to see that
\begin{equation}\label{G-7}
\mathcal{R}_{21} = \bar{\nu}^{(n_i, n_{i+1}]} - v_{i+1}^\infty = - \frac{\kappa(N- n_{i+1} - n_i)}{N} \Phi^\infty.
\end{equation}

\noindent $\star$ (Estimate of $\mathcal{R}_{22}$):~ It follows from \eqref{G-2} and the algorithm \eqref{f-algorithm} that
\begin{align*}
\begin{aligned}
&\lim_{t \to +\infty} \mathcal{R}_{22}\\
 &= \lim_{t \to +\infty`} \left( \frac{1}{r-n_i} \sum_{k= n_i+1}^r \hat{\nu}_k^{(n_i, n_{i+1}]}  + \frac{1}{r-n_i} \frac{\kappa}{N} \sum_{k= n_i+1}^r \sum_{l = r+1}^{n_{i+1}} \Phi(x_l-x_k) \right) \\
&= \frac{1}{r-n_i} \sum_{k= n_i+1}^r \hat{\nu}_k^{(n_i, n_{i+1}]} + \frac{\kappa(n_{i+1} - r)}{N} \Phi^\infty >0.
\end{aligned}
\end{align*}
Then for a given small $\varepsilon > 0$,  there exists a time $t_1>0$ such that
\begin{equation}\label{G-8}
\mathcal{R}_{22} = \frac{1}{r-n_i} \sum_{k= n_i+1}^r \hat{\nu}_k^{(n_i, n_{i+1}]}  + \frac{1}{r-n_i} \frac{\kappa}{N} \sum_{k= n_i+1}^r \sum_{l = r+1}^{n_{i+1}} \Phi(x_l-x_k) \ge 3\varepsilon, \quad t\ge t_1.
\end{equation}
\newline
\noindent $\star$ (Estimate of $\mathcal{R}_{23}$):~Due to the assumption of induction and from \eqref{G-2}, we have
\begin{align*}
\begin{aligned}
&\lim_{t \to +\infty} \mathcal{R}_{23}\\
 &= \lim_{t \to +\infty} \frac{1}{r-n_i} \frac{\kappa}{N} \sum_{k= n_i+1}^r \left( \sum_{l=1}^{n_i} \Phi(x_l - x_k) + \sum_{l = n_{i+1} + 1}^N \Phi(x_l - x_k)\right) \\
&= \frac{1}{r-n_i} \frac{\kappa}{N} \sum_{k= n_i+1}^r \left( -n_i \Phi^\infty + (N- n_{i+1}) \Phi^\infty\right) = \frac{\kappa(N- n_{i+1} - n_i)}{N} \Phi^\infty.
\end{aligned}
\end{align*}
Thus, for a sufficiently small $\varepsilon > 0$, there exists a time $t_2 > 0$ such that
\begin{equation}\label{G-9}
\mathcal{R}_{23} \ge \frac{\kappa(N- n_{i+1} - n_i)}{N} \Phi^\infty - \varepsilon, \quad t \ge t_2.
\end{equation}
Therefore, we combine \eqref{G-7}, \eqref{G-8}, and \eqref{G-9} to obtain that
\begin{equation}\label{G-10}
\frac{d}{dt} (\bar{x}^{(n_i, r]} - (v_{i+1}^\infty + \varepsilon)t) \ge 3\varepsilon - \varepsilon -\varepsilon = \varepsilon >0, \quad t \ge  \max\{t_1, t_2\}.
\end{equation}
It follows from \eqref{G-6} and \eqref{G-10} that
\[ \frac{d}{dt} \bar{x}^{(n_i, r]}  > \frac{d}{dt} \bar{x}^{(n_i, n_{i+1}]}, \quad t \ge \max \{t_0, t_1, t_2\}.\]
Then, there exists a time $t_c > \max \{t_0, t_1, t_2\}$ such that
\begin{equation}\label{G-11}
\bar{x}^{(n_i, r]}(t) > \bar{x}^{(n_i, n_{i+1}]}(t), \quad \text{for} \ t \ge t_c.
\end{equation}
However, the well-ordering of particles implies that
\[\bar{x}^{(n_i, r]}(t) \le \bar{x}^{(n_i, n_{i+1}]}(t), \quad t \ge 0,\]
which gives a contradiction to \eqref{G-11}. Thus, the particles in $\mathcal{I}_{i+1}$ form a cluster, that is, do not depart from each other. Then we can use the same arguments as in Step $A$ to show that other particles depart from the particles in $\mathcal{I}_{i+1}$. Thus, by the method of induction, we conclude that we exactly have an asymptotic $N_c(\kappa)$-cluster formation and that each ensemble $\mathcal{I}_i$ exhibits a maximal cluster.
\newline

\noindent $(ii)$~ From $(i)$, it is known that each ensemble $\mathcal{I}_i$ is a maximal cluster. Then, for $k \in \mathcal{I}_i$ and $l \not\in \mathcal{I}_i$, we have
\begin{equation}\label{G-12}
\lim_{t \to +\infty} |x_k(t) - x_l(t)| = +\infty.
\end{equation}
We consider the velocity of center of mass for the group $\mathcal{I}_i$,
\begin{align*}
\begin{aligned}
&\frac{d}{dt} \bar{x}^{(n_{i-1}, n_{i}]}\\
 &= \frac{1}{n_i - n_{i-1}} \sum_{k = n_{i-1} +1}^{n_i} [\nu_k + \frac{\kappa}{N}
\underset{l \ne k}{\sum_{l=1}^N} \Phi(x_l - x_k)] \\
&= \bar{\nu}^{(n_{i-1}, n_i]} + \frac{1}{n_i - n_{i-1}} \frac{\kappa}{N} \sum_{k = n_{i-1} +1}^{n_i} \sum_{l =1 }^{n_{i-1}} \Phi(x_l-x_k) + \frac{1}{n_i - n_{i-1}} \frac{\kappa}{N} \sum_{k = n_{i-1} +1}^{n_i}  \sum_{l = n_i +1}^N \Phi(x_l -x_k).
\end{aligned}
\end{align*}
Thus, it yields from \eqref{G-12} that
\begin{align*}
\begin{aligned}
&\lim_{t \to +\infty} \frac{d}{dt} \bar{x}^{(n_{i-1}, n_{i}]} \\
&= \bar{\nu}^{(n_{i-1}, n_i]} - \frac{1}{n_i - n_{i-1}} \frac{\kappa}{N} (n_i - n_{i-1})n_{i-1} \Phi^\infty  +  \frac{1}{n_i - n_{i-1}} \frac{\kappa}{N}(n_i - n_{i-1})(N- n_i) \Phi^\infty \\
&= \bar{\nu}^{(n_{i-1}, n_i]} + \frac{\kappa(N - n_i - n_{i-1}) \Phi^\infty}{N}.
\end{aligned}
\end{align*}
Therefore the asymptotic group velocity $v_i^\infty$ of cluster $\mathcal{I}_i$ is given by the following explicit formula:
\[v_i^\infty := \bar{\nu}^{(n_{i-1}, n_i]} + \frac{\kappa(N - n_i - n_{i-1}) }{N} \Phi^\infty, \quad 1 \le i \le N_c(\kappa).\]
\newline
\noindent $(iii)$~Suppose tha we have $p$-cluster flocking configuration $\bigsqcup_{i =1}^p \mathcal{J}_i$, we next show that
\begin{equation}\label{G-13}
p = N_c(\kappa) \quad \text{and} \quad \mathcal{J}_i = \mathcal{I}_i, \ i = 1,2,\ldots,p.
\end{equation}
We consider the group $\mathcal{I}_1$. From Step $A.1$, we know that the particles in $\mathcal{I}_1$ will not depart from each other. Then, we have $\mathcal{I}_1 \subset \mathcal{J}_1$. Whereas according to the assumption that $\mathcal{J}_1$ is maximal, from Step $A.2$ we know that $\mathcal{J}_1$ must satisfy the algorithm \eqref{f-algorithm}. Thus, we have $\mathcal{J}_1 \subset \mathcal{I}_1$. Then, we obtain $\mathcal{J}_1 = \mathcal{I}_1$. Combining the arguments in Step $B$, we can obtain \eqref{G-13} by induction.
  \qed
 \vspace{0.3cm}

 \end{appendix}

\end{document}